\documentclass[11pt,reqno]{amsart}
\usepackage{amssymb}
\usepackage{a4wide}
\usepackage{amsfonts}
\usepackage{mathrsfs}
\usepackage{amsmath}
\usepackage{amssymb, amsmath}
\usepackage{color}

\usepackage{fullpage}

\newcommand{\newcom}{\newcommand}

\newcom{\al}{\alpha}
\newcom{\be}{\beta}
\newcom{\eps}{\epsilon}
\newcom{\del}{\delta}
\newcom{\Del}{\Delta}
\newcom{\ga}{\gamma}
\newcom{\Ga}{\Gamma}
\newcom{\ka}{\kappa}
\newcom{\Lam}{\Lambda}
\newcom{\lam}{\lambda}
\newcom{\Om}{\Omega}
\newcom{\om}{\omega}
\newcom{\Si}{\Sigma}
\newcom{\si}{\sigma}
\newcom{\tht}{\theta}
\newcom{\dtri}{\nabla}
\newcom{\td}{\tilde}
\newcom{\tri}{\triangle}
\newcom{\oo}{\infty}
\newcom{\vphi}{\varphi}
\newcom{\calb}{{\mathcal B}}
\newcom{\calc}{{\mathcal C}}
\newcom{\cD}{{\mathcal D}}
\newcom{\cF}{{\mathcal F}}
\newcom{\cG}{{\mathcal G}}
\newcom{\cI}{{\mathcal I}}
\newcom{\cL}{{\mathcal L}}
\newcom{\cM}{{\mathcal M}}
\newcom{\cP}{{\mathcal P}}
\newcom{\cR}{{\mathcal R}}
\newcom{\cS}{{\mathcal S}}
\newcom{\cQ}{{\mathcal Q}}
\newcom{\caly}{{\mathcal Y}}
\newcom{\calz}{{\mathcal Z}}
\newcom{\bfz}{{\bf Z}}
\newcom{\R}{\mathbb R}
\newcom{\N}{\mathbb N}
\newcom{\Z}{\mathbb Z}
\newcom{\C}{\mathbb C}
\newcom{\E}{\mathbb E}

\newcom{\f}{\frac}
\newcom{\di}{\displaystyle\int}
\newcom{\ds}{\displaystyle\sum}
\newcom{\dl}{\displaystyle\lim}
\newcom{\ov}{\overline}
\newcom{\sset}{\subset}
\newcom{\wt}{\widetilde}
\newcom{\p}{\partial}
\newcom\na{\nabla}
\newcom{\co}{\cdot}
\newcom{\suml}{\sum\limits}
\newcom{\supl}{\sup\limits}
\newcom{\intl}{\int\limits}
\newcom{\infl}{\inf\limits}
\newcom{\disp}{\displaystyle}
\newcom{\non}{\nonumber}
\newcom{\no}{\noindent}
\newcom{\QED}{$\square$}
\def\ef{\hphantom{MM}\hfill\llap{$\square$}\goodbreak}

\newtheorem{athm}{\bf \t}[section]
\newenvironment{thm} [1] {\def\t{#1}\begin{athm} \bf \rm} {\end {athm}}
\newcom{\bthm}{\begin{thm}}\newcom{\ethm}{\end{thm}}

\newcom{\beq}{\begin{equation}}
\newcom{\eeq}{\end{equation}}
\newcom{\ben}{\begin{eqnarray}}
\newcom{\een}{\end{eqnarray}}
\newcom{\beno}{\begin{eqnarray*}}
\newcom{\eeno}{\end{eqnarray*}}

\newtheorem{prop}{Proposition}
\newtheorem{theoreme}[prop]{Theorem}
\newtheorem{lem}[prop]{Lemma}
\newtheorem{cor}[prop]{Corollary}

\newtheorem{rem}[prop]{Remark}

\numberwithin{equation}{section}
\numberwithin{prop}{section}

\numberwithin{equation}{section}

\begin{document}

\title[Multi-solitons water-waves]{Multi-solitons and related solutions \\ for the water-waves system}

\author{Mei Ming}

\author{Frederic Rousset}

\author{Nikolay Tzvetkov}

\begin{abstract}
The main result of this work is the construction of  multi-solitons solutions that is to say solutions that are time asymptotics to a sum of decoupling solitary waves 
 for the full water waves system with surface tension.
%This is a generalization of similar results obtained for simplified model equations such as the KdV equation. The construction is quite different compared with
%the one solitons  one obtained in the classical work by Amick-Kirchgassner, we use a  since one has a dispersive tail added to the main multi-soliton core 
%whose  control turns out to be quite delicate.  
\end{abstract}
\maketitle
%%%%%%%%%%%%%%%%%%%%%    
%%%%%%%%%%%%%%%%%%%%
\section{Introduction}

We  consider the motion of an  irrotational,  incompressible
fluid with constant density. We consider the situation
 where the fluid domain is a strip with a rigid bottom and a free surface:
$$
\Omega_t=\{Y=(X,z)\in\R^{d+1}\,:\, -H <z<\eta(t,X)\},
$$
where $t$ is the time, $d=1,2$ is the horizontal dimension,  $H$  is a parameter defining the fixed  bottom $z=-H$  and
 $z= \eta(t,X)$  is the equation of the unknown free surface at time $t$.
 We shall say that we are in the one dimensional case when  $X=x \in \mathbb{R}$
 and in the two-dimensional case when $X= (x,y) \in \mathbb{R}^2$. 
 A large part of the paper will be devoted to the one-dimensional situation. We denote by $u$ the speed of the fluid, since the motion is irrotational, it is given by
 $u=\nabla_{Y}\Phi= (\nabla_{X}\Phi,\partial_{z}\Phi)$ for some scalar
function $\Phi$ and hence we find that  inside the fluid domain $\Omega_{t}$, 
\beq
\label{inside}
\nabla_{Y}\cdot u=  \Delta_{Y} \Phi = 
(\Delta_X+\partial_{z}^2)\Phi=0\,.
\eeq 
On the boundaries  of $\Omega_{t}$, we make the usual 
assumption that no fluid particles cross the boundary.
At  the bottom of the fluid this reads 
\beq
\label{fond}
\partial_{z}\Phi(t,X,-H)=0
\eeq
and on the free surface,  this yields the kinematic condition 
\beq
\label{surface1}
\partial_{t}\eta(t,X)+\nabla_{X} \Phi(t,X,\eta(t,X))\cdot\nabla_{X} \eta(t,X)-\partial_{z}\Phi(t,X,\eta(t,X))=0\,.
\eeq
On the free surface, we also need to impose the pressure,  taking into account the surface tension and using the Bernouilli law to eliminate the pressure, 
we find that:
 \beq
\label{surface2}
\partial_{t}\Phi(t,X,\eta(t,X))+\frac{1}{2}|\nabla_{Y}\Phi(t,X,\eta(t,X))|^{2}+g\eta(t,X)=
b\nabla_{X} \cdot\frac{\nabla_{X} \eta(t,X)}{\sqrt{1+|\nabla_{X} \eta(t,X)|^{2}}}\,.
\eeq
The number $b$ is the surface tension coefficient and $g$ is the gravitational constant.
The term $g\eta(t,X)$ is the trace of the gravitational force $gz$ on the free surface.

It is classical  to rewrite the system \eqref{inside}, \eqref{surface1}, \eqref{surface2}
as a system  involving unknowns defined on  the free surface only \cite{Zakharov}. For that purpose,
let us define the following   Dirichlet-Neumann  operator:  for given $\eta(X)$  
$\varphi(X)$, we define $\Phi(X,z)$ as the (well-defined) solution of the elliptic boundary value problem
$$
(\Delta_X+\partial_{z}^2)\Phi=0,\quad{\rm in}\quad
\{(X,z)\,:\,-H<z<\eta(X)\},
\quad
\Phi(X,\eta(X))=\varphi(X),\quad \partial_{z}\Phi(X,-H)=0, 
$$
and we define the Dirichlet-Neumann operator as
\beno
(G[\eta]\vphi)(X)&:=&(\p_z\Phi-\na_X\eta\cdot\na_X\Phi)|_{z=\eta(X)}\\
&=&\sqrt{1+|\na_X\eta|^2}(\na_{X,z}\Phi\cdot {\bf n})|_{z=\eta(X)},
 \eeno 
 where $\bf n$ is the unit outward normal vector on the free
 surface at the point $z=\eta(X)$.

This allows to rewrite the system only in terms of the unknowns 
$$
(\eta(t,X), \varphi(t,X)):=(\eta(t,X),\Phi(t,X,\eta(t,X)))\,.
$$
In the one-dimensional case,  the 1D water-wave problem can thus  be written as \beq\label{water-wave
system}\left\{\begin{array}{ll} 
\displaystyle{\p_t\eta=G[\eta]\varphi }\\
 \displaystyle{\p_t\varphi=-\f12|\p_x\varphi|^2+\f12\f{(G[\eta]\varphi+\p_x\varphi\p_x\eta)^2}{1+|\p_x\eta|^2}
-g\eta+b\p_x\Big(\f{\p_x\eta}{\sqrt{1+|\p_x\eta|^2}}\Big)}
\end{array}\right.\eeq 
By introducing the notations $U=(\eta,\varphi)^t$ and 
\[ \cF(U)=
\Big(G[\eta]\varphi,\quad
-\f12|\p_x\varphi|^2+\f12\f{(G[\eta]\varphi+\p_x\varphi\p_x\eta)^2}{1+|\p_x\eta|^2}
-g\eta+b\p_x\Big(\f{\p_x\eta}{\sqrt{1+|\p_x\eta|^2}}\Big)\Big)^t,\]
we shall write  the water-wave system (\ref{water-wave system}) in the abstract form
\beq\label{simple form for WW system} 
\p_tU=\cF(U). 
\eeq

We know from \cite{AK} that  for suitable  parameters $g$, $b$ and $h$,  there exist
solitary  wave solutions $Q_c(x-ct)=(\eta_c(x-ct),\varphi_c(x-ct))^t$ at speed $c\sim\sqrt{gH}$. Here is a precise statement.
 %%%%%%%%
 \begin{theoreme}[Amick-Kirchg\"assner \cite{AK}]\label{theoOS}
Suppose that 
\begin{equation}\label{restr}
\alpha=\frac{gH}{c^2}=1+\varepsilon^2,\quad \beta= \frac{b}{H c^2}>\frac{1}{3}\,.
\end{equation}
Then there exists $\varepsilon_0$ such that for every
$\varepsilon\in (0,\varepsilon_0)$  (which fixes the speed) there is a solution of (\ref{water-wave system}) under  the form
$$
Q_c(x-ct)=\big(\eta_{c}(x-ct), \varphi_{c}(x-ct)\big)^t=\Big(H\eta_{\varepsilon}(H^{-1}(x-ct)),\, cH\varphi_{\varepsilon}(H^{-1}(x-ct))\Big)^t
$$ 
with
$$
\eta_{\varepsilon}(x)=\varepsilon^2 \Theta_1(\varepsilon x,\varepsilon),
\quad
\varphi_{\varepsilon}(x)=\varepsilon \Theta_2(\varepsilon x,\varepsilon),
$$
where $\Theta_1$ and $\Theta_2$  satisfy:
$$
\exists\, d>0,\quad \forall\,\alpha\geq 0 ,\quad \exists\, C_{\alpha}>0,\quad 
\forall\, (x,\varepsilon)\in\R\times (0,\varepsilon_0),\quad
|(\partial^{\alpha}_{x}\Theta_1)(x,\varepsilon)|\leq C_{\alpha} e^{-d|x|}  
$$
and
$$
\exists\, d>0,\quad \forall\,\alpha\geq 1 ,\quad
\exists\, C_{\alpha}>0,\quad 
\forall\, (x,\varepsilon)\in\R\times (0,\varepsilon_0),\quad
|(\partial^{\alpha}_{x}\Theta_2)(x,\varepsilon)|\leq C_{\alpha} e^{-d|x|}  \,.
$$
Moreover $\Theta_1$ is even and $\Theta_2$ is odd.
\end{theoreme}
%%%%%%%%%%
The main aim of  this paper is to construct multi-solitons type solutions for the water-waves system
 \eqref{water-wave system}. This means that we want to construct a solution of \eqref{water-wave system}  that
 tends  to a sum of solitary waves with different speeds when $t$ goes to infinity.
 For the sake of readability of the paper, we shall only consider the case of  two solitary waves, 
 the extension to an arbitrary numbers is straightforward (our proof will not use any particular symmetry or specificity related to
  the 2-solitons case).
 The construction of such solutions for semi-linear equations like the KdV equation or the Nonlinear-Schr\"odinger equation
 has been intensively studied recently, we refer for example to \cite{Martel, Martel-Merle, Krieger-Martel-Raphael, Cote-Martel-Merle, Cote-LeCoz, Mizu}.
 We also refer to an earlier closely related work by Merle \cite{Merle} which seems to initiate this line of research. 

 In this paper, we shall show that the construction can be also performed for quasilinear equations (or even fully nonlinear equations) by focusing on the physically interesting
 example of the water-waves system \eqref{water-wave system}. The new difficulty that arises  in the case of  quasilinear or fully nonlinear  equations like \eqref{water-wave system}
  is that the only well-posedness result which is known in the vicinity of the solitary wave  is a local well-posedness result  in $H^s$ for $s$ sufficiently large which comes from the high
  order energy method.  Note that the global or almost global existence results  like the ones of  \cite{Wu1,Wu2}, \cite{Germain-Masmoudi-Shatah1,Germain-Masmoudi-Shatah2} are obtained in a regime where solitary waves do not exist and that the local existence results for rough data as obtained in   \cite{Alazard-Burq-Zuily} still require
   a regularity much higher than the one which is controlled by the Hamiltonian.
  
  This makes the perturbative analysis more delicate and requires new ideas with respect to the semi-linear setting.
  
We thus consider 
two solitons  $Q_{c_1}(x-c_1t)$ and $Q_{c_2}(x-h-c_2t)$ 
of  (\ref{water-wave system}) with $c_1<  c_2$.
We suppose that $c_1$ and $c_2$ satisfy \eqref{restr} with suitable choices of the small parameters $\varepsilon_{1, 2}$.  
We also suppose that $h>0$ is large
enough. 
We define
\begin{equation}\label{M}
M(t,x):=Q_{c_1}(x-c_1t)+Q_{c_2}(x-h-c_2t)
\end{equation}
as the two-soliton function. 
We will focus on the case where each  solitary wave is stable in the following sense. 
Under our assumptions  \eqref{restr} on the speed $c$ of a solitary wave, 
%let us introduce the non-dimensional parameters
%$$ \alpha_{i}=  \f {gH}{c^2_i}, \quad \beta_{i} =  {\f b {H c_{i}^2}}, \quad i=1, \,2$$
%which are respectively the Froud number and the Bond number. Assuming that $\beta_{i}>1/3$ and that
% $\alpha_{i}= 1 + \varepsilon_{i}^2$, for $\varepsilon_{i}>0$ sufficiently small,  the existence of a nice  exponentially decreasing solitary wave   is provided by the classical 
 %result of Amick and Kirchgassner \cite{AK}.  
 it was  proven in  \cite{Mielke}  that for sufficiently small  corresponding  parameter $\varepsilon$,   the solitary wave $Q_{c}$   is stable in the sense that the second  derivative of the Hamiltonian at the solitary wave  restricted to a natural co-dimension $2$ subspace is positive. 
We shall assume that the  speeds  $c_{1}, c_{2}$ are such  that this property is verified (see Proposition \ref{Propstab1d}).
Our main result reads:
\begin{theoreme}\label{multi-soliton theorem}  
Let us fix $s\geq 0$.
Suppose that the speeds $c_1< c_2$ satisfy \eqref{restr} with  parameters $\varepsilon_1$, $\varepsilon_2$. 
Define $M$ by \eqref{M}.
Then there exists $\varepsilon^*$ such that for $\varepsilon_{1},\varepsilon_2 \in (0, \varepsilon^*]$ and
 $h$ sufficiently large, we have  that there
exists a (semi) 
global solution $U(t)=(\eta,\varphi)^t$  to the water-wave system (\ref{water-wave system}) satisfying 
$$
U-M\in \mathcal{C}_{b}([0,\infty); \,H^s(\R)\times
H^s(\R))
$$
and
\[
\lim_{t\rightarrow +\infty} \|U(t)-M(t)\|_{H^s}=0.
\]
\end{theoreme}
%%%
The assumption that $\varepsilon_{1,2}$ are sufficiently small will be  only used 
 in order to ensure the existence of  smooth exponentially decreasing  solitary waves (given by Theorem~\ref{theoOS} for example) and  that Proposition~\ref{Propstab1d} below
holds true.  Indeed,  for  the construction of multi-solitary waves that we shall perform, the  main ingredients that are required are besides   the existence
of  smooth localized solitary waves, the stability property given by  Proposition~\ref{Propstab1d}  of each solitary wave  and  the local well-posedness
 in $H^s$ for $s$  sufficiently large of the nonlinear system with the existence of a quasilinear hyperbolic system type energy estimate. 
 Instead of the smallness assumption on $\varepsilon$ in Theorem~\ref{multi-soliton theorem}, we could assume the existence of the solitary waves
  (solitary waves can  also be obtained by variational methods for example,  \cite{Buffoni}) and that each of them satisfies the stability property of
   Proposition~\ref{Propstab1d}. 

We have focused on water waves with surface tension, nevertheless, 
 since  the existence of   solitary waves is also known  (see \cite{Iooss-Kirchgassner} for example) and  since some of them  are linearly stable,  \cite{Pego-Sun}
 (note that nevertheless an estimate like the one of Proposition~\ref{Propstab1d} does not hold in this case, the number of negative directions of $L$ would be infinite), 
 it  could be possible  to  perform a related construction  for water waves without surface tension.
 
  Finally, let us point out that  the assumption that $h$ is sufficiently large is only used in order to get a solution on $[0,+\infty[$, an equivalent statement would be to take 
  $h=0$  and to get a multi-soliton solution on the interval $[T_{0}, +\infty[$ with $T_{0}$ sufficiently large.

There are two main steps in the   proof of  Theorem \ref{multi-soliton theorem}. The first step is to  construct a smooth  approximate solution
 of \eqref{water-wave system} that tends exponentially fast  to $M$ as $t$ goes to infinity.  This approximate solution $U^a$ is under the form:
 $$ U^a = M + \sum_{l=1}^N V^l$$
  where each $V^l$ is  smooth and verifies the crucial property that it is  decaying  in $H^s$ like $e^{-l \eps_{0}(c_{2}- c_{1})t}$ for some $\eps_{0}$. Each $V^l$ solves
   a linear problem with  source  term. The existence of some $V^l$ with this decay property will be proven by using the spectral properties
    of the linearization of the water waves system \eqref{water-wave system}  about each solitary wave.
 
  Once the approximate solution is  sufficiently accurate (i.e for $N$ sufficiently large)
 which basically means that the remainder term in the equation has a sufficiently fast decay in time, we shall construct an exact solution
 as a sum of the approximate solution and remainder term that solves a nonlinear equation.   The main difficulty is to prove  the existence of a solution on $[0,  + \infty) $
  for this problem having at hand only  a  local Cauchy Theory in $H^s$ for $s$  large. 
   
Note that this kind of iterated constructions is related to  Grenier's argument \cite{Grenier} in order to prove that linear instability implies nonlinear 
instability for quasilinear equations that has been used  in \cite{RT}.  
    
The main arguments that we shall use to prove Theorem~\ref{multi-soliton theorem} can also be used in order to sharpen the transverse instability result 
proven in \cite{RT}  and construct  for the two-dimensional  water-waves  system that is to say  when the  fluid domain is
     $$
      \Omega_{t}= \{ (X, z) \in \mathbb{R}^3, \, -H <z<\eta(t,X) \},
     $$
     a solution on $[0, + \infty)$ of the system which is different from the solitary wave (and all its translates) and converge to the solitary wave
    as time goes to infinity.  The result that we shall prove is the following.
    \begin{theoreme}
\label{Instability Thm} 
Let us fix $s\geq 0$.
Suppose that $c$ satisfies \eqref{restr}. 
For $\varepsilon$ sufficiently small  there exists a global solution $U$ of the 2-D water waves system
 with initial data $U_0$ satisfying $U-Q_{c}\in
\mathcal{C}_{b}([0,\infty);H^s(\R^2)\times H^s(\R^2))$.
%and $1-\|\eta_0\|_{L^\infty}>0$.
Moreover, one has 
\beq
\label{conddy}
\p_yU_0\neq 0\eeq
 and
\[
\lim_{t\rightarrow +\infty}\|U(t,x,y)-Q_{c}(x-c t)\|_{H^s}=0\,.
\]
% for every  $\sigma, $ $0 \leq \sigma \leq s$.
  \end{theoreme}
  We shall recall the formulation of the 2D water-waves system in Section~\ref{section2D}.
 \begin{rem} 
 By the remark after \cite[Theorem 1.5]{RT2} we know that this theorem implies the transverse instability of the solitary wave.  
\end{rem}
%%%
This result  can  be compared to classical results about the existence of  strongly stable manifolds for
ordinary differential equations or semilinear partial differential equations. Results as in Theorem~\ref{Instability Thm}  were in particular, 
obtained for semilinear partial differential equations  in \cite{Duyckaerts-Merle,Combet} for example and also in \cite{RT2} for the KP-I
equation.
 As previously, the main difficulty for the proof of this result
comes from the fact that the water-waves system is not semilinear.
The proof of Theorem~\ref{Instability Thm} also relies on the construction of a well-chosen
approximate solution and of a remainder that solves a nonlinear system. The construction
 of the approximate solution   relies  on  spectral information and semi-group
 estimates that  are contained in \cite{RT}. Consequently, we shall first
 present the proof of Theorem~\ref{Instability Thm} in Section~\ref{section2D}, this allows
  us to essentially focus on the construction of a remainder that is defined on the whole time
  interval $[0,+\infty[$. These arguments will be also useful for the proof of Theorem~\ref{multi-soliton theorem}.
  
  \bigskip
  
The paper is organized as follows.  
The next section will be devoted to the various steps of the proof of Theorem~\ref{Instability Thm}.
In Section~\ref{added}, we collect some useful bounds for the Dirichlet-Neumann map and we prove the key coercivity property of
Proposition~\ref{Propstab1d}. 
We shall study in Section~\ref{section error} the error that produces in the equation
 the sum  $M$ of two solitary waves. Then in Section~\ref{roximate}, we shall
 construct a suitable approximate solution. This will be the most difficult part in the proof,
 the crucial step will be to  prove that the fundamental solution of the linearized equation
 about $M$ has a sufficiently small exponential growth. The final step of the proof of Theorem~\ref{multi-soliton theorem}
  is given in Section~\ref{section end of the proof}.
%%%%%%%%%%%%%%%%%%%%%%%%%%%%%%%%
 %%%%%%%%%%%%%%%%%%%%%%%%%%%%%%%%%%%%%%%%%%%
 \section{Proof of Theorem \ref{Instability Thm}}\label{section2D}
 In this section we consider the two-dimensional water-waves system. The fluid domain at time  $t$ is defined by
\[
\Omega_t=\{(X,z)\in\R^3\,|\,-H<z<\eta(t,X)\}
\] 
where $X=(x,y)$, $H>0$ is a constant and $\eta(t,X)$ is the free surface at time $t$.
We use the  Zakharov formulation recalled in the introduction  \cite{Zakharov,L} to write the system  for the unknowns $\eta(t,X)$ which is the free-surface and $\varphi(t,X)$ which
is the trace on the free surface of the velocity potential as:
\beq\label{original WW} 
\left\{\begin{array}{l}
\displaystyle{\p_t\eta=G[\eta]\vphi,}\\
\displaystyle{\p_t\vphi=-\f12|\na_X\vphi|^2+\f12\f{(G[\eta]\vphi+\na_X\vphi\cdot\na_X\eta)^2}{1+|\na_X\eta|^2}
-g\eta+b\na_X\cdot\f{\na_X\eta}{\sqrt{1+|\na_X\eta|^2}}}
\end{array}\right. 
\eeq 
where again $b$ is the coefficient of surface tension, $g$ is the gravity coefficient and $G$ is the Dirichlet-Neumann operator.

Since we study  here a single solitary wave with speed $c$, we change frame
$(x,y,z)$ to $(x-ct,y,z)$. This leads to a new system \beq\label{WW}
\left\{\begin{array}{l}
\p_t\eta=c\p_x\eta+G[\eta]\vphi,\\
\displaystyle{\p_t\vphi=c\p_x\vphi-\f12|\na_X\vphi|^2+\f12\f{(G[\eta]\vphi+\na_X\vphi\cdot\na_X\eta)^2}{1+|\na_X\eta|^2}
-g\eta+b\na_X\cdot\f{\na_X\eta}{\sqrt{1+|\na_X\eta|^2}}.}
\end{array}\right. \eeq 
We also  perform the following change of
variables
\begin{equation}\label{rescale}
\eta(t,X)=H\tilde \eta(\f cHt,\f1HX),\qquad\vphi(t,X,z)=cH\tilde
\vphi(\f cHt,\f1HX)
\end{equation}
and simply note $(\tilde \eta,\tilde\vphi)$ again as
$(\eta,\vphi)$ to have the dimensionless water-waves system
\beq\label{scaled WW} 
\left\{\begin{array}{l}
\p_t\eta=\p_x\eta+G[\eta]\vphi,\\
\displaystyle{\p_t\vphi=\p_x\vphi-\f12|\na_X\vphi|^2+\f12\f{(G[\eta]\vphi+\na_X\vphi\cdot\na_X\eta)^2}{1+|\na_X\eta|^2}
-\al\eta+\be\na_X\cdot\f{\na_X\eta}{\sqrt{1+|\na_X\eta|^2}}}
\end{array}\right. \eeq with 
\[\al=\f{gH}{c^2},\qquad\be=\f
b{Hc^2}.
\] 
Observe that there is a slight abuse of notation, since in \eqref{scaled WW} the map $G[\eta]$ is defined as above but with $H=1$.
If we note $U=(\eta,\vphi)^t$,  the system (\ref{scaled WW})
can be rewritten as 
\beq\label{simple form for WW}
\p_t U=\cF(U)
\eeq 
with
\[
\cF(U)=\left(\begin{array}{l}
\p_x\eta+G[\eta]\vphi\\
\displaystyle{\p_x\vphi-\f12|\na_X\vphi|^2+\f12\f{(G[\eta]\vphi+\na_X\vphi\cdot\na_X\eta)^2}{1+|\na_X\eta|^2}
-\al\eta+\be\na_X\cdot\f{\na_X\eta}{\sqrt{1+|\na_X\eta|^2}}}
\end{array}\right).
\]
 The solitary wave $Q_{c}$ of the original system \eqref{original WW} becomes a stationary wave
solution for the water-waves system (\ref{scaled WW}) that we shall denote in this section by 
$Q_\varepsilon(x)=(\eta_\varepsilon(x),\vphi_\varepsilon(x))^t$ or simply as $Q(x)$. We want to show that there exists
a global solution of system (\ref{scaled WW}) near the solitary wave.
%%%%%%%%%%%%%%%%%%%%%%%%%%%%%%%%%%%%%%%%%%%%%%%%%%%%%%%%
\subsection{The linearized operator}
\label{sectionlin1}
As in \cite{RT}, we linearize the water-waves system (\ref{scaled WW}) around the solitary wave $Q_\varepsilon=(\eta_\varepsilon,\vphi_\varepsilon)^t$.  Let us set 
\[
Z_\varepsilon:=Z[Q_\varepsilon]=\f{G[\eta_\varepsilon]\vphi_\varepsilon+\partial_x \eta_\varepsilon \partial_x \vphi_\varepsilon}
{1+|\partial_x \eta_\varepsilon|^2},\quad
v_\varepsilon:=\p_x\vphi_\varepsilon-Z_\varepsilon\p_x\eta_\varepsilon,
\] and 
\[
P_\varepsilon\eta:=
\be\na_X\cdot\left[\f{\na_X\eta}{(1+(\partial_x \eta_\varepsilon)^2)^\f12}
-\f{(\na_X\eta_\varepsilon\cdot\na_X\eta)\na_X\eta_\varepsilon}{(1+(\partial_x \eta_\varepsilon)^2)^\f32}\right].
\] 
Note that since $G[\eta_\varepsilon]\vphi_\varepsilon=-\partial_x \eta_\varepsilon$, we obtain that $Z_\varepsilon$ and $v_\varepsilon$ are decaying exponentially together with all their derivatives  thanks
 to Theorem \ref{theoOS}
($\varphi_\varepsilon$ occurs with a derivative in $v_\varepsilon$).
%%%
The linearization of (\ref{scaled WW}) about $Q_\varepsilon$ reads
\beq\label{linearized system 1}\p_t U=J\Lam U,\eeq where
$U=(\eta,\vphi)^t$, 
$
J=\left(\begin{matrix}0&1\\-1&0\end{matrix}\right)
$ 
is a skew-symmetric matrix and
\[
\Lam= \left(\begin{matrix} -P_\varepsilon+\al+Z_\varepsilon
G[\eta_\varepsilon](Z_\varepsilon\cdot)+Z_\varepsilon\p_x v_\varepsilon & (v_\varepsilon-1)\p_x-Z_\varepsilon
G[\eta_\varepsilon]\\
-\p_x((v_\varepsilon-1)\cdot)-G[\eta_\varepsilon](Z_\varepsilon\cdot) & G[\eta_\varepsilon]
\end{matrix}\right)
\] is a symmetric operator on $L^2 \times L^2$. 
The formula for the differential of the Dirichlet-Neumann operator with respect to $\eta$ which allows to obtain the above
expression is recalled in Proposition \ref{propDN1} (5) below.
 As in \cite{L}, we can get a simpler form of the linearized system by the change of unknowns
\beq
\label{good-unknown}V_1=\eta,\quad V_2=\vphi-Z_\varepsilon\eta.
\eeq
We get  for $V=(V_1,V_2)^t$ the system
\beq\label{linearized system 2} \p_tV=JLV, \eeq where $L$  is still a   symmetric
defined  by 
\[
L=\left(\begin{matrix}-P_\varepsilon+\al+(v_\varepsilon-1)\p_x Z_\varepsilon &
(v_\varepsilon-1)\p_x\\
-\p_x((v_\varepsilon-1)\cdot) & G[\eta_\varepsilon]
\end{matrix}\right).
\] 
Notice that  since $Q_\varepsilon$ does not depend on $y$, the study of system (\ref{linearized system 2}) can be simplified by using Fourier transform in $y$. In fact, if
for some $k\in \R$,
\[V(x,y)=e^{iky}W(x),\] then \[
LV=e^{iky} L(k) W \] 
with a 
symmetric operator $L(k)$   defined  by 
\beq\label{operator L(k)}
L(k)=\left(\begin{matrix}-P_{\varepsilon,k}+\al+(v_\varepsilon-1)\p_x Z_\varepsilon &
(v_\varepsilon-1)\p_x\\
-\p_x((v_\varepsilon-1)\cdot) & G_{\varepsilon,k}
\end{matrix}\right).
\eeq Here \[
P_{\varepsilon,k}u=\be\big\{\p_x\big[(1+(\p_x\eta_\varepsilon)^2)^{-\f32}\p_x
u\big]-k^2(1+(\p_x\eta_\varepsilon)^2)^{-\f12}u\big\},
\] and $G_{\varepsilon,k}$ is such that
\[
G[\eta_\varepsilon](f(x)e^{iky})=e^{iky}G_{\varepsilon,k}(f(x)).
\]
%%%%%%%%%%%%%%
\subsection{The eigenvalue problem $JL(k)U=\sigma U$}
We shall need some results about the spectrum and  the  eigenvalues of $JL(k)$  seen as an unbounded operator on $L^2(\mathbb{R}) \times L^2(\mathbb{R})$
 with domain $H^2(\mathbb{R}) \times H^1(\mathbb{R})$ which are essentially contained in \cite{RT}.
\begin{lem}\label{eigenvalue} 
We have the following spectral properties of the operators $JL(k)$.
\begin{itemize}
\item There exists $\varepsilon^*$ such that for every $\varepsilon \in (0, \varepsilon^*],$ the solitary wave $Q_{\varepsilon}$
 is spectrally stable against one-dimensional perturbation:  $\si(JL(0))\subset i\R$ where $\sigma(\cdot)$ denotes the spectrum.
 
 \item For
any $k\neq 0$, the essential spectrum of $JL(k)$ is such that  $\sigma_{ess}( J L(k))\subset i \mathbb{R}$ and $JL(k)$  has at most one   simple eigenvalue of positive real part. 
 \item If $\sigma$ is an eigenvalue of $JL(k)$ then so is $-\sigma$.
\item If $\sigma$ is an eigenvalue of $JL(k)$, $k\neq 0$, with non-zero real part   then $\si\in \R$.
\end{itemize}
\end{lem}
Note that the combination of above properties of the eigenvalues also yield that there is at most one eigenvalue of negative real part for $JL(k)$, $k \neq 0$.
\begin{proof} 
These statements are already contained in  Lemma 5.1 in \cite{RT}.
Note that the first statement in the above lemma was obtained as a consequence of the work of Mielke \cite{Mielke}.
The only additional point is to notice that thanks to the reversibility symmetry of the water waves system and  the symmetry of  the solitary wave  one has that if $U(x)=(\eta(x), \varphi(x))$ is an eigenfunction of $JL(k)$
corresponding to an eigenvalue $\sigma$
then $\tilde U(x)= (\eta(-x), -\varphi(-x))^t$ is an eigenfunction corresponding to the eigenvalue $-\sigma$ (this symmetry of the spectrum could also be obtained by using the Hamiltonian structure).
The last point is a consequence of the symmetry of the spectrum and the fact that there is at most one simple eigenvalue of  positive real part.
\end{proof}
Thanks to Lemma~\ref{eigenvalue} and in particular the symmetry of the spectrum, we have the following counterparts of \cite[Proposition~5.2] {RT} and \cite[Theorem 5.3]{RT} respectively.
%%%%%%%%%%%%%%%%%%%%
\begin{lem}\label{location of eigenmode} 
Consider the eigenvalue problem
\begin{equation}\label{eigen}
JL(k)U=\sigma U,\quad U\in H^2(\R)\times H^1(\R).
\end{equation}
Then :

\noindent(1) $\exists K>0$ such that if $|k|>K$,  there is no solution of \eqref{eigen} satisfying $\si<0$.

\noindent(2) $\exists M>0$ such that if $ |k|\leq K$, there is no solution of \eqref{eigen} satisfying $\si\leq -M$.
 \end{lem}
 
 From now on, we shall only consider solitary waves $Q_{\varepsilon}$  with $ \varepsilon \in (0, \varepsilon^*]$ 
 as stated in Theorem~\ref{Instability Thm}
 so that the above spectral
 properties are matched. 
%%
%%%%%%%%%%%%%%%%%%%%%%%%%%%%%%%%%%%
\begin{prop}\label{existence of negative eigenvalue} 
For $\varepsilon \in (0,\varepsilon^*]$, 
there exists
$\si>0$, $k\neq 0$ and a non-trivial $U\in H^2\times H^1$ such
that\[ JL(k)U=-\si U.\] 
\end{prop}
Finally, we also have as a consequence of the above results:
%%%%%%%%%%%%%%
\begin{lem}\label{Analicity of eigenmode} 
For every   $(\si_0,k_0)$
satisfying $k_0\neq 0$, $\mbox{Re } \si_0<0$ and $\si_0\in \si(JL(k_0))$,
the set $\{(\si,k)\,|\,\si\in \si(JL(k))\,\}$ near $(\si_0,k_0)$ is
the graph of an analytic curve $k\mapsto \si(k)$ with $\si(k)$ a real 
 simple eigenvalue of $JL(k)$. Moreover, we can select an eigenvector $V(k)$ of $JL(k)$ depending analytically on $k$.
\end{lem}
\begin{proof}
Suppose that $\si_0\in \si(JL(k_0))$  with $k_0\neq 0$, $\mbox{Re } \si_0<0$. From Lemma \ref{eigenvalue}, $\sigma_{0}$ is necessarily a simple real eigenvalue.  Let $V_0\in H^2\times H^1$  an  eigenvector,   $JL(k_0)V_0=\sigma_0 V_0$.
Consider the map $F(V,k,\sigma)\equiv JL(k)V-\sigma V$. Then $D_{V,\sigma}F(V_0,k_0,\sigma_0)(U,\mu)=JL(k_0)U-\sigma_0 U-\mu V_0$.
Thanks to Lemma \ref{eigenvalue},  $JL(k_0)-\sigma_0$ is Fredholm with index zero and  its kernel is one-dimensional.
This implies that the kernel of $D_{V,\sigma}F(V_0,k_0,\sigma_0)$ is also one dimensional.
Indeed, if $(U,\mu)$ is  such that $D_{V,\sigma}F(U_0,k_0,\sigma_0)(U,\mu)=0$ then $(JL(k_0)-\sigma_0)^2U=0$ which, thanks to Lemma~\ref{eigenvalue}, implies that
$U=\lambda V_0$, $\lambda\in\R$. This in turn implies that $\mu=0$ and thus the kernel of  $D_{V,\sigma}F(V_0,k_0,\sigma_0)$ is spanned by $(V_0,0)$.
Once again, thanks to Lemma~\ref{eigenvalue}, we obtain that $V_0$ is not in the image of $JL(k_0)-\sigma_0$ 
(otherwise $\sigma_0$ would not be a simple eigenvalue of $JL(k_0)$). Therefore $D_{V,\sigma}F(V_0,k_0,\sigma_0)$ is surjective. We are in position to apply the simplest form of the Lyapounov-Schmidt method. This ends  the proof of Lemma~\ref{Analicity of eigenmode}.
\end{proof}
%%%%%%%%%%%%%%%%%%%%%%%%%%%%%%%%%%%%%%%%%%%%%%%%%%%%%%%%%%%%
\subsection{Construction of the approximate solution}
Let us define for  $U=(\eta,\vphi)^t$,  the norms 
\[
\|U(t)\|_{E^s}=\sum_{0\leq \al+\be+\ga\leq
s}\|\p^\al_t\p^\be_x\p^\ga_yU(t,\cdot)\|_{L^2(\R^2)}\,.
\] 
%In this section we assume that the solutions for linear problems
%are smooth enough. 
\begin{prop}\label{construct U^0} There
exists $U^0(t,x,y)\in \cap_{s\ge0}E^s$ satisfying\beq\label{U^0 eqn}
\p_tU^0=J\Lam U^0 \eeq 
 such that for every $s\ge 0$, every $\eps_0>0$ one has
\[
%c^{-1}_{s,\eps_0}e^{-\si_0 t}\leq 
\|U^0(t)\|_{E^s}\leq c_{s,\eps_0} e^{-(\si_0-\eps_0)
t},\qquad t\ge 0
\] 
where  $-\si_0$ is the smallest possible eigenvalue of
$JL(k)$ (and thus $\sigma_0$ is the largest possible eigenvalue of $JL(k)$). 
\end{prop}
Note that we use the notations $\eps_{0}$ and $\varepsilon_{0}$ for different parameters.

\begin{proof}[Proof of Proposition~\ref{construct U^0}] 
The construction is close to the one of  Proposition~6.1 in \cite{RT}.
The situation is even simpler here since we need less precise information on the asymptotic behavior. 
We already know from \eqref{good-unknown} that  it is equivalent to solve  (\ref{U^0
eqn}) for $U^0$  and   \beq\label{V^0
eqn}\p_tV^0=JL V^0\eeq with
\[
U^0=PV^0,\quad P=\left(\begin{matrix}1 & 0\\
Z_\varepsilon & 1\end{matrix}\right).
\] 
 for $V^0$. 

First of all, one should locate some negative eigenvalue of $JL(k)$. We
already know from Proposition \ref{existence of negative eigenvalue} that there exists $k\neq 0$ such that $JL(k)$ has an
eigenvalue $-\del<0$. Moreover, thanks to Lemma \ref{location of
eigenmode}, we know that the negative eigenvalues $-\si$ of $JL(k)$ can only be
found when $|k|<K$ and $\si<M$.

We try to find out the largest $\si$ such that $-\si$ is still a
negative eigenvalue of $JL(k)$. Define the set
$\Omega=\{\,k\,|\,\exists\, -\si \in \si(JL(k)), \,-\si<-\del/2 \}$.
One can see that $\Omega$ is a bounded non-empty open set and that 
$k\mapsto -\si(k)$ is continuous in $\bar \Omega$. We fix $k_0,
\si_0$ such that \[ -\si_0=-\si(k_0)=\inf
\{-\si(k)|\,k\in\bar\Omega\}<-\del/2<0.
\]  Let us  fix $\eps_0>0$
and  an  interval $I_0 \subset \Omega$ small
enough with  $k_0\in I_0$
so  that $-\si_0\leq -\si(k)\leq -\si_0+\eps_0$ in
$I_0$ by the continuity of $k\mapsto -\si(k)$. Thanks to  Lemma~\ref{Analicity of eigenmode}, we can can choose  an eigenvector  $ V(k)$ 
depending analytically on $k$ on  $I_0$. 
By elliptic regularity, we have that $V(k) \in H^\infty$.  Finally 
%Choose a disk $B(-\si_0,r)\subset
%\{\,Re\si<0\,\}$ such that for any $k\in I_0$, $-\si(k)\in
%B(-\si_0,r)$ and $\p B(-\si_0,r)$ does not include any eigenvalue of
%$JL(k)$ for $k\in I_0$. So now $(-\si-JL(k))^{-1}$ is well-defined
%on $\p B(-\si_0,r)\times \bar I_0$. Let
%\[
%V(k)=\f1{2\pi i}\int_{\p
%B(-\si_0,r)}(-\si-JL(k))^{-1}U(k_0)d\si,\qquad k\in \bar I_0.
 since one has $\si(k)=\si(-k)$ and $ V(k)=V(-k)$, 
let  us set $I=I_0\cup
-I_0$ and 
\[
V^0(t,x,y)=\int_I e^{-\si(k)t}e^{iky}V(k)dk,\qquad t\ge 0.
\]  Then  $V^0$ is real and is a solution of (\ref{V^0 eqn}). We have for any $s,\al\in\N$ that
\[
\|\p^\al_t V^0(t,\cdot)\|^2_{H^s(\R^2)}=C\int_I
e^{-2\si(k)t}\sum_{s_1+s_2\leq s} |\si(k)|^{2\al}k^{2s_2}|\p^{s_1}_x
V(k)|^2_{L^2(\R)} dk.
\]  and hence,  there exist numbers  $c_{s,\al, \eps_{0}}$ such
that for any $t\ge 0$ \[  \|\p^\al_t
V^0(t,\cdot)\|_{H^s(\R^2)}\leq c_{s,\al,\eps_0} e^{- (\si_0-\eps_0) t}.
\] 
This yields similar estimates for  $U^0=P V^0$. This ends  the proof of Proposition~\ref{construct U^0}.
\end{proof}
%%%%%%%%
\begin{prop}\label{construct U^a} For any $M>0$, there exists
\[
U^a=U^0+\sum^{M+1}_{j=1}\del^jU^j,\qquad U^j\in C^\infty
(\R_+,H^\infty(\R^2)),\quad \del\in\R
\] such that for every $j$, one has $U^j(0)=0$ and the estimates\[
\|U^j(t)\|_{E^s}\leq C_{s,j}e^{-(j+1)(\si_0-\eps_0)t},\qquad \forall
t\ge0
\]
for $\eps_{0}$ sufficiently small ($\eps_{0}<\sigma_{0}/2$) for  some numbers $C_{s,j}$ independent of $t$  and $\si_0$ the eigenvalue chosen in Proposition~\ref{construct U^0}.
 Moreover, define $V^a=Q+\del U^a$.  Then $V^a$ is an approximate solution of
(\ref{simple form for WW}) in the sense that\[
\p_tV^a-\cF(V^a)=R^{ap}
\] where  $R^{ap}$ satisfying the estimate
\[
\|R^{ap}\|_{E^s}\leq C_{M,s}\del^{M+3}e^{-(M+3)(\si_0-\eps_0)
t},\qquad t\ge 0.
\]
\end{prop}

\noindent{\bf Proof.} We follow the idea of proof of \cite[Proposition~6.3]{RT}. The Taylor expansion of $\cF$ is
\[
\cF(Q+\del
U)=\cF(Q)+\sum^{M+2}_{k=1}\f{\del^k}{k!}D^k\cF[Q](U,\dots,
U)+\del^{M+3}R_{M,\del}(U).
\] Plugging the expansion of $U^a$ into (\ref{simple form for WW})
gives the equations for $j\ge 1$ that \beq\label{eqn for U^j} \p_t
U^j-J\Lam U^j=\sum^{j+1}_{p=2}\sum_{\stackrel{0\leq
l_1,\dots,l_p\leq j-1}{ l_1+\cdots+l_p=j+1-p}}\f
1{p!}D^p\cF[Q](U^{l_1},\dots,U^{l_p}). \eeq We note the right-hand
side of (\ref{eqn for U^j}) as $S^j$. Applying Fourier transform in
$y$ to (\ref{eqn for U^j}), we get that
\[
\p_t\hat U^j-J\Lam(k)\hat U^j=\hat S^j,\qquad j\ge 1.
\] So we need to consider first the equation 
\[
\p_t U-J\Lam(k)U=F.
\] In order to  estimate the solution, we introduce 
\[ |U(t)|^2_{X^s_k}:=\sum_{0\leq \al+\be\leq
s}\Big(|\p^\al_t\p^\be_x
U_1(t,\cdot)|^2_{H^1(\R)}+|\p^\al_t\p^\be_x
U_2(t,\cdot)|^2_{\dot{H}^{\f12}_k(\R)} \Big),
\] with
\[
|\vphi|^2_{\dot{H}^{\f12}_k(\R)}=\left|\f{|D_x|}{1+|D_x|^{\f12}}\vphi\right|^2_{L^2(\R)}
+|k|^2|\vphi|^2_{L^2(\R)}.
\]
%%%%%%
We have the following statement.
\begin{prop}\label{estimate for U^j} 
Fix $\ga>\si_0>0$. Assume  that $F(t,x,k)$ satisfies uniformly for
$|k|\leq K$ the estimate
\beq\label{F estimate} 
\sum_{\al+\be\leq s}\|\p^\al_t\p^\be_x F(t,\cdot,k)\|_{L^2}\leq M_s e^{-\ga t},
\qquad t\ge 0. \eeq 
Then we can find a solution $U$ of 
\[\p_tU-J\Lam(k)U=F,\]
defined for $t\geq 0$ such that there exists constant $C_{s}$ depending on $M_{s+s_0}$ for some fixed $s_0\ge 0$ such
that uniformly for  $|k|\leq K$, we have
\[
|U(t,\cdot)|_{L^2}+|U(t,\cdot)|_{X^s_k}\leq C_{s}e^{-\ga
t},\qquad t\ge 0.
\]
\end{prop}
\begin{proof}
Let $V=P^{-1}U$, then the equation for $U$ is equivalent
to the equation
\beq
\label{Vsemieq}
\p_t V=JL(k)V+P^{-1}F.\eeq
%\noindent {\bf Step 1.} Estimate for the homogeneous system. We need to
%estimate the homogeneous system for $V(t)$:
%\[
%\p_tV=JL(k)V,\quad V|_{t=0}=V_0.
%\] We use the following decomposition 
%$
%V(t)=V_0+W(t)
%$
%where $W(t)$ solves
%\[
%\p_t W-JL(k)W=JL(K)V_0,\quad W|_{t=0}=0.\] 
%We have the crude estimate
%$$
%\sum_{\al+\beta\le s}\|\p^\al_t\p^\beta_x(JL(k)V_0)\|_{L^2}\le C_1(s)(|V_0|_{L^2}+|V_0|_{X^{s+2}_k}),
%$$
Let $\bar\si>0$ be such that $\gamma>\bar{\sigma}>\sigma_0$.
By a simple lifting argument, we get from  \cite[Proposition~6.4]{RT} and the symmetry of the spectrum pointed out in Lemma \ref{eigenvalue} that  the semi-group estimate
%\[
%|W(t)|_{L^2}+|W(t)|_{X^{s}_k}\le C_2(s) \,e^{\bar\si t},\quad t\ge 0,
%\] 
%where $C_2(s)$ depends on $C_1(s+s_0)$ for some fixed $s_0$.
 $$
\|e^{t\,JL(k)}V_0\|_{L^2\cap X^{s}_k}\le
C_s \,e^{\bar\si t}(|V_0|_{L^2}+|V_0|_{X^{s+s_0}_k})
,\quad t\geq 0\,,
$$
holds 
for some fixed derivative loss $s_0$ (we could avoid it, but this is not important in our construction). 
By using again the reversibility of the system, we actually obtain that
\begin{equation}
\label{semi}
\|e^{t\,JL(k)}V_0\|_{L^2\cap X^{s}_k}\le
C_s \,e^{\bar\si | t |}(|V_0|_{L^2}+|V_0|_{X^{s+s_0}_k})
,\quad  \forall t \in \mathbb{R}.
\end{equation}

%where $C_s$ depends on $|f|_{L^2}+|f|_{X^{s+s_0}_k}$.
%%%
%\noindent{\bf Step 2.} Inhomogeneous case. 
Let us choose the solution of \eqref{Vsemieq} given by 
\[
V(t)=-\int^\infty_t e^{(t-\tau)JL(k)}P^{-1}F(\tau)d\tau.
\] 
Then,  thanks to \eqref{semi}, we have
$$ 
|V(t)|_{L^2}+|V(t)|_{X^s_k}
\le C_s \int^\infty_t e^{\bar\si(\tau-t)}e^{-\ga \tau}d\tau
\leq
C_s e^{-\ga t},
$$
with $C_s$ having the claimed structure.
This ends the proof of Propoposition~\ref{estimate for U^j}.
\end{proof}
%%%%%%%%
\noindent{\bf End of the proof of Proposition~\ref{construct U^a}.} We can
finish the proof by induction on $j$ as in the proof of Proposition~6.3 in
\cite{RT}.  Assume that we already built $(U^l)_{l \leq j-1}$ whose Fourier transforms with respect to the $y$ variable are compactly supported in $k$ 
 and that satisfy  uniformly for $|k| \leq K_{l}$ the   estimates
  $$|\hat U^l(t,k)|_{E^s} \leq C_{s,l} e^{-(l+1)(\si_0-\eps_0)t}, \quad   l\le j-1$$
  where the $ | \cdot |_{E^s}$ norm for functions of $t$ and $x$ is naturally defined as 
  $$ |V(t)|_{E^s}= \sum_{| \alpha | \leq s } | \partial_{t,x}^\alpha V(t) |_{L^2(\mathbb{R})}.$$
   To construct  $U^j$, we first observe that 
 thanks to  the induction assumption we obtain as in the proof of Proposition ~6.3 in \cite{RT} (hence in particular by using Proposition~3.9 in \cite{RT}
  to handle the terms coming from the Dirichlet-Neumann operator) that 
$$\|\hat S^j(t, k)\|_{E^s}\leq \tilde C_{s,j}e^{-(j+1)(\si_0-\eps_0)t}$$
uniformly  for $|k| \leq K_{j}$.
Then we  define $U_{j}$ by
$$ \hat U^j(t,k)= - P \int^\infty_t e^{(t-\tau)JL(k)}P^{-1}\hat S^j(\tau, k)d\tau.$$
 Thanks to the semigroup estimate \eqref{semi}, we get since  $\ga=(j+1)(\si_0-\eps_0)>\si_0>0$ that
 $$  \|\hat U^j(t,k)\|_{E^s} \leq C_{s,j}e^{-(j+1)(\si_0-\eps_0)t}$$
and the estimate for $U^j$ follows by using the Bessel identity and the fact that $\hat U^j(t,k)$ is compactly supported in $k$.
%.\ef
%%%%%%%%%%%%%%
\subsection{Proof of  Theorem \ref{Instability Thm}}
Recall that we already defined $V^a=Q+\del U^a$ in Proposition~\ref{construct U^a} as an approximate solution $U$ of
 the water-waves system  (\ref{simple form for WW}). In order to get a true  solution of  (\ref{simple form for WW}),  we
still need to consider a  remainder $U^R$ such that 
\[
U=V^a+U^R,\qquad 
\] 
becomes an exact solution. The system for $U^R$ is
\beq\label{eqn for remainder U}\left\{\begin{array}{l} \p_t
U^R=\cF(V^a+U^R)-\cF(V^a)-R^{ap},\qquad t>0,\\
 U^R(0) \quad\hbox{to be fixed later}.\end{array}\right. \eeq Before
solving the system for $U^R$ we need to introduce more notations.
For $\al=(\al_0,\al_1,\al_2)\in \N^3$ we note
$
\p^\al=\p^{\al_0}_t\p^{\al_1}_x\p^{\al_2}_y.
$
For $U(t)=(U_1(t),U_2(t))^t$ and $k\in \N$ we define $X^k$ by
\[
\|U(t)\|^2_{X^k}=\sum_{|\al|\le k}(\|\p^\al
U_1(t)\|^2_{H^1}+\|\p^\al U_2(t)\|^2_{H^\f12})
\] and we note
\[
\|U\|_{X^k_{t,T}}=\sup_{t\le \tau\le T}\|U(\tau)\|_{X^k}.
\]
We define $W^k$ and $W^k_{t,T}$ by
\[
\|u(t)\|_{W^k}=\sum_{|\al|\le k}\|\p^\al u(t)\|_{L^\infty(\R^2)},\quad 
\|u\|_{W^k_{t,T}}=\sup_{t \le  \tau \le T}\|u(\tau)\|_{W^k}\,.
\]
We will use an approximate sequence of solutions $\{U^n\}$ for
water-waves problem to prove that there exists a global-in-time
solution $U^R(t)$ of (\ref{eqn for remainder U}). Let $\{T_n\}$ be a
strictly increasing sequence such that $T_n>0$ and $T_n\rightarrow +\infty$ as
$n\rightarrow \infty$. First of all, we define $U^n(t)$ as the
solution of the water-waves system \beq\label{eqn for U^n}
\left\{\begin{array}{l}\p_tU^n=\cF(V^a+U^n)-\cF(V^a)-R^{ap},\qquad
 t\le T_n,\\ U^n(T_n)=0.\end{array}\right. \eeq
 Note that we solve the problem  backward in time.  For the water-waves system there is no problem to do
so  because of the reversibility.  We have local well-posedness for \eqref{eqn for U^n} as in \cite{RT} (see also \cite{Alazard-Burq-Zuily}, \cite{Beyer-Gunther}, 
 \cite{Schweizer}). The first step is thus  to prove that $U^n$ is defined on the whole interval $[0,T^n]$
  and that it satisfies an appropriate estimate.
This will be a consequence of  the following a priori estimate 
%%%%%%%%%%%%%%%%%%%%%%%%
\begin{prop}\label{estimate for U^n} Let
$U^n(t)$ be a smooth solution of (\ref{eqn for U^n}) on $[T, T_n]$
satisfying $1-\|\eta^a(t)\|_{L^\infty}-\|\eta^n(t)\|_{L^\infty} \geq c_{0}>0$
for $t\in [T,T_n]$. Then for $m\ge 2$, $s\ge 5$ and $t\in [T, T_n]$
we have the estimate 
\beno \|U^n(t)\|^2_{X^{m+3}}&\le&
\omega(\|R^{ap}\|_{X^{m+3}_{t,T_n}}+\|V^a\|_{W^{m+s}_{t,T_n}}+\|U^n\|_{X^{m+3}_{t,T_n}})\\
&&\quad
\times\Big[\|R^{ap}\|^2_{X^{m+3}_{t,T_n}}+\int^{T_n}_t(\|U^n(\tau)\|^2_{X^{m+3}}+\|R^{ap}(\tau)\|
^2_{X^{m+3}})d\tau\Big] \eeno 
with 
$\omega:\R^{+}\rightarrow [1,+\infty]$ is a continuous increasing  function.
\end{prop}
%%%%%%%%%%%%%%%%%%%%%%%%%%%%%%%%
%The proof of this proposition is quite technical, it follows by an adaptation of the proof of \cite[Theorem~7.1]{RT}. 
%Let us only recall the main line of the argument.
%After differentiating three times the equation \eqref{eqn for U^n}, one finds a suitable energy functional $E_m^n(t)$ such that
%\[
%\|U^n(t)\|^2_{X^{m+3}}\leq 
%\omega(\|R^{ap}\|_{X^{m+3}_{t,T_n}}+\|V^a\|_{W^{m+s}_{t,T_n}}+\|U^n\|_{X^{m+3}_{t,T_n}})
%E^n_m(t)
%,\qquad t\in [0,T_n]
%\] 
%and such that one has the a priori estimates
%\[ -\f d{dt} E^n_m(t)\le\
%\omega(\|R^{ap}\|_{X^{m+s}_{t,T_n}}+\|V^a\|_{W^{m+s}_{t,T_n}}+\|U^n\|_{X^{m+3}_{t,T_n}})
%(\|U^n(t)\|^2_{X^{m+3}}+\|R^{ap}\|^2_{X^{m+3}}) \]
%and 
%\[
%|E^n_m(T_n)|\le
%\omega(\|R^{ap}(T_n)\|_{X^{m+s}}+\|V^a(T_n)\|_{W^{m+s}}+\|U^n(T_n)\|_{X^{m+3}})
%\|R^{ap}(T_n)\|^2_{X^{m+3}}.
%\] 
%The last three estimates imply  Proposition~\ref{estimate for U^n}.
Proposition  ~\ref{estimate for U^n} can be directly obtained from \cite{RT}, Theorem 7.1.
Indeed, let us  define the isometry $\tilde \cdot$ by  
$$ \tilde  U(\tau, \tilde x, \tilde y)= \big(U_{1}(T^n - \tau, -\tilde x,  -\tilde y), - U_{2} (T^n - \tau, -\tilde x,  -\tilde y)\big)^t.$$
 Then, we note that
  $U^n$ solves \eqref{eqn for U^n} if and only if  $\tilde U^n(\tau, \tilde x , \tilde  y)$
% defined by 
% $$\tilde  U^n(\tau, \tilde x, \tilde y)= \big(\eta^n(T^n - \tau, -\tilde x,  -\tilde y), -\varphi^n (T^n - \tau, -\tilde x,  -\tilde y)\big) , \quad \tau \in [0, T^n]$$
 solves
 $$ \partial_{\tau} \tilde U^n = \mathcal{F}(\tilde V^a(\tau)+ \tilde{U}^n(\tau)) - \mathcal{F}(\tilde V^a(\tau)) + \tilde{R}^{ap}(\tau), \quad \tau \in [0, T^n] $$
 with the initial data $\tilde U^n(0)= 0$.

We can now convert the a priori estimate  of Proposition~\ref{estimate for U^n} into the following existence result.
%%%%%%%%%%%%%%%%%%%%%%%%
\begin{prop}\label{existence of U^n} Let $m\ge 2$.
Then there exists  $M$ large enough in the definition of $V^{a}$  and $\delta$ sufficiently small such that the following holds true.
For every $T_n>0$ there exists a unique  solution of 
%$V^a\in
%W^{m+s}_{0,T_0}$ and $R^{ap}\in X^{m+3}_{0,T_0}$ where $s\ge 3$ is  a fixed 
%integer. Then for $T_0\ge
%T_n>0$, there exists a solution $U^n(t)\in
%L^\infty([0,T_n],\,X^{m+3})$ for system 
(\ref{eqn for U^n}) defined on $[0,T^n]$ and 
satisfying 
$$
1-\|\eta^a\|_{L^\infty}-\|\eta^n(t)\|_{L^\infty}
>0,\quad 
 \|U^n(t)\|_{X^{m+3}}\le
C_{M,m}\del^{M+3}e^{-(M+3)(\si_0-\eps_0)t},\qquad t\in[0,T_n].
$$
\end{prop}
%%%%%%%%%%
\begin{proof}
From \cite[Section~8]{RT} for example and the above reversibility argument,  we  know
the local-in-time existence for the solution $U^n(t)$ of system
(\ref{eqn for U^n}) going backwards from time $T_n$. We want to prove 
 that this solution exists on the whole  time interval $[0,T_n]$.
By Proposition~\ref{construct U^a} and Proposition~\ref{estimate for U^n}, we have that, at least for $t$ close to $T_n$ 
\begin{multline}
\label{eqpartie1final1}
\|U^n(t)\|_{X^{m+3}}^2 
\leq
\omega(C+\|U^n\|_{X^{m+3}_{t,T_n}}+C_{M,m}\delta)
\times
\Big(\int^{T_n}_t\|U^n(\tau)\|^2_{X^{m+3}}d\tau
\\
 +C_{M,m}\del^{2(M+3)}e^{-2(M+3)(\si_0-\eps_0) t}\Big).
\end{multline}
Define \[ T^*=\inf\{T\in[0,T_n]\,:\, \forall t\in[T,\,T_n],\,\|U^n(t)\|_{X^{m+3}}
\le 1,\,1-\|\eta^a\|_{L^\infty}-\|\eta^n(t)\|_{L^\infty}\geq c_{0}>0\}.\]
 We deduce from  \eqref{eqpartie1final1} that 
\[
\|U^n(t)\|^2_{X^{m+3}}\le\omega(C+C_{M,m}\delta)\Big(\int^{T_n}_t\|U^n(\tau)\|^2_{X^{m+3}}
d\tau+C_{M,m}\del^{2(M+3)}e^{-2(M+3)(\si_0-\eps_0) t}\Big)
\]
for $t\in[T^*,T_n]$. 
Now we shall fix the value of $M$ and impose a smallness condition on $\delta$. 
Let $M$ and $\delta$ be such that 
\[
2(M+3)(\si_0-\eps_0)>\omega(C+C_{M,m}\delta)
\] 
Such a choice is possible thanks to the continuity of $\omega$. Indeed we can first take $M$ large enough so that $2(M+3)(\si_0-\eps_0)>\omega(C)+2$ and
then $\delta$ small enough so that $|\omega(C+C_{M,m}\delta)-\omega(C)|<1$.
Therefore, we arrive at the bound
\begin{equation*}
\f
d{dt}\Big(-e^{\omega(C+C_{M,m}\delta)t}\int^{T_n}_t\|U^n(\tau)\|^2_{X^{m+3}}d\tau\Big)
\le 
C_{M,m}\del^{2(M+3)}
e^{\omega(C+C_{M,m}\delta)t-2(M+3)(\si_0-\eps_0)t}.
\end{equation*}
Integrating on both sides with respect to time from $t$ to $T_n$ leads to
\[
\int^{T_n}_t\|U^n(\tau)\|^2_{X^{m+3}}d\tau\le
C_{M,m}\del^{2(M+3)}e^{-2(M+3)(\si_0-\eps_0)t},
\] which gives
\beq\label{small estimate for U^n} \|U^n(t)\|^2_{X^{m+3}}\le
C_{M,m}\del^{2(M+3)}e^{-2(M+3)(\si_0-\eps_0)t} \eeq for any
$t\in[T^*,T_n]$. Now we let $\del$ small enough such that
$C_{M,m}\del^{2(M+3)}<1$ and $1-\|\eta_{\varepsilon}\|_{L^\infty} - C_{M,m} \delta >c_{0}.$
Then by definition, we have  $T^* \leq 0$, hence   
 the solution $U^n(t)$
for (\ref{eqn for U^n}) can be extended to the whole  time interval $[0,T_n]$ with the claimed estimates.
\end{proof}
%%%%%%%%%%
We can now finish the proof of Theorem~\ref{Instability Thm} by invoking some standard compactness arguments. 
\\

Step 1.
Existence of the global solution $U(t)$. Thanks to Proposition~\ref{estimate for U^n} and 
Proposition~\ref{existence of U^n}, we get the
approximation sequence $\{U^n\}$ of solutions of (\ref{eqn for U^n})
which satisfy
\[
\|U^n(t)\|_{X^{m+3}}\le C_{M,m}\del^{M+3}e^{-(M+3)(\si_0-\eps_0)
t},\qquad \forall t\in[0,T_n].\] 
Let $\psi\in C_0^{\infty}(-1/2,1/2)$ be a bump function such that $\psi(x)=1$ for $x\in (-1/4,1/4)$.
We extend $U^n(t)$ as zero for $t>T_n$ and we set 
$$
\tilde U^n(t)=\psi(t/T_n)U^n(t),\quad t\geq 0.
$$
%Extend $U^n(t)$  as
%\[\tilde U^n(t)=\left\{\begin{array}{ll} U^n(t),\quad t\in[0,
%T_n-1],\\
%(T_n-t)U^n(t),\quad t\in[T_n-1,T_n],\\
%0,\quad t> T_n.\end{array}\right.
%\] 
Then, we have
$$
\partial_{t}\tilde U^n(t)=\frac{1}{T_n}\psi'(t/T_n)U^n(t)+\psi(t/T_n)\partial_t U^n(t),\quad t>0
$$
and thus new sequence $\{\tilde U^n(t)\}$ satisfies
\[
\|\tilde U^n(t)\|_{H^{m+4}\times H^{m+7/2}}\le
C_{M,m}\del^{M+3}e^{-(M+3)(\si_0-\eps_0) t},\quad t\geq 0
\] and
\[
\|\p_t\tilde U^n(t)\|_{H^{m+3}\times H^{m+5/2}}\le
C_{M,m}\del^{M+3}e^{-(M+3)(\si_0-\eps_0) t},\qquad \forall t\ge 0.
\]
By a standard compactness argument, we obtain  that there exists a subsequence
$\{\tilde U^{n_k}(t)\}$ and  \[
U^R(t)\in L^\infty([0,\infty),\, H^{m+4}\times H^{m+7/2})\] such that
\[
\tilde U^{n_k}\rightarrow U^R \quad\hbox{in}\quad
C_{loc}(\R_+,H^{m+3}_{loc}\times H^{m+5/2}_{loc})\quad
\hbox{as}\quad n_k\rightarrow \infty
\] and
\[
\|U^R(t)\|_{H^{m+4}\times H^{m+7/2}}\le
C_{M,m}\del^{M+3}e^{-(M+3)(\si_0-\eps_0)t}\qquad \forall t\in
[0,\infty).
\] 
Moreover $U^R(t)$ is the solution of (\ref{eqn for remainder U}) since $U^{n_k}(t)$ is a solution of (\ref{eqn for remainder U}) on $(0,T_{n_k}/4)$.
Going back to water-waves system (\ref{simple form for WW}) we deduce that there is a global solution
$U(t)=V^a(t)+U^R(t)$ for system (\ref{simple form for WW}).
\\

\noindent Step 2. Behavior when $t\rightarrow +\infty$. By the
definition of $V^a$ we have
\[
U(t)=V^a(t)+U^R(t)=Q+\sum^{M+1}_{j=0}\del^{j+1}U^j(t)+U^R(t),
\] and from Proposition~\ref{construct U^a} and (\ref{small estimate for
U^n}) we know that for any $t\in [0,\infty)$ \beno
&&\|U^j(t)\|_{E^s}\le C_{M,s}e^{-(j+1)(\si_0-\eps_0)
t},\\
&& \|U^R(t)\|_{X^{s+3}}\le
C_{M,s}\del^{M+3}e^{-(M+3)(\si_0-\eps_0)t}\eeno with $s\ge 2$. This yields
\[
\|U^j(t)\|_{H^s}\le C_{M,s}e^{-(j+1)(\si_0-\eps_0)t},\quad
\|U^R(t)\|_{H^s}\le C_{M,s}e^{-(M+3)(\si_0-\eps_0)t},\qquad \forall
t\in[0,\infty)
\] which shows that
\[
\lim_{t\rightarrow +\infty}\|U(t)-Q(x)\|_{H^s}=0.
\]

\noindent Step 3.  It remains to check the
 condition  \eqref{conddy} for $U$.
 Since \[
U(t)=V^a(t)+U^R(t)=Q(x)+\del U^a(t)+U^R(t),
\] we have at  $t=0$ that
\[
U(0)=Q(x)+\del\sum_{j=0}^{M+1}\del^jU^j(0)+U^R(0)
\] and
\[
\p_y( U(0))=\del\p_y(U^0(0))+
\sum_{j=1}^{M+1}\del^{j+1}\p_y(U^j(0))+\p_y(U^R(0)).
\] We know from the definition of $U^0$ that $\p_y(U^0(0))\neq 0$.
 We can use Proposition \ref{construct U^a} and  step 1 to get 
  that
  $$  \|U^j(0)\|_{E^s} \leq C_{s,j}, \quad  \|U^R(0)\|_{X^{m+3}}\le
C_{M,m}\del^{M+3}.$$
This yields  that $\p_y(U(0))\neq 0$ for $\delta$  possibly  smaller.
This ends the  proof of Theorem~\ref{Instability Thm}.
\ef
%%%%%%%%%%%%%%%%%%%%%%%%%%%%%%%%%%%%%%%%%%%%%%%%%%%%%%%%%%%%%%%%%%%%%%%
%%%%%%%%%%%%%%%%%%%%%%%%%%%%%%%%%%%%%%%%%%%%%%%%%%%%%%%%%%%%%%%%%%%%%%
\section{Preliminaries for the proof of Theorem \ref{multi-soliton theorem}}\label{added}
\subsection{The Dirichlet-Neumann operator}
 Let us recall that the Dirichlet-Neumann  operator $G[\eta]$  is  defined  by 
\[
G[\eta]\psi=\sqrt{1+|\p_x\eta|^2}\p_n\varphi|_{z=\eta}
\] where $n$ is the unit outward normal vector and $\varphi$ solves
\[
\left\{\begin{array}{ll} \Del\varphi=0,\quad\hbox{in}\quad
-H<z<\eta(x)\\
\varphi|_{z=\eta}=\psi,\quad \p_n\varphi|_{z=-H}=0
\end{array}\right.
\] 
We can rewrite the elliptic problem on $ -H<z<\eta(x)$ as an
equivalent problem on a flat strip $\cS=\R\times [-1,\,0]$:
\beq\label{elliptic problem-new} \left\{\begin{array}{ll}
\na_{x,z}\cdot P\na_{x,z}\tilde{\varphi}=0,\quad\hbox{in}\quad
\cS\\
\tilde{\varphi}|_{z=0}=\psi,\quad \partial_{z}\tilde{\varphi}|_{z=-1}=0
\end{array}\right.
\eeq where 
$\tilde{\varphi}(x,z)=\varphi(x,Hz+(z+1)\eta)$ and \[
P=\left(\begin{matrix} H+\eta & -(z+1)\p_x\eta\\
-(z+1)\p_x\eta & \f{1+(z+1)^2(\p_x\eta)^2}{H+\eta}
\end{matrix}\right).
\]  
With the notation $\p^P_n=(0, 1)^t\cdot
P\,\na_{x,z}$, one can see that the D-N operator associated to the
above problem can be written as 
\beq
\label{defDNbis}
G[\eta]\psi=\p^P_n\tilde{\varphi}|_{z=0}= -\partial_{x} \eta\partial_{x} \tilde{\varphi}|_{z=0} 
+ \f { 1 + (\partial_{x} \eta)^2} {H+ \eta} \partial_{z}\tilde{\varphi}|_{z=0}.
\eeq
Let us recall the following properties:
\begin{prop}\label{propDN1} 
For $\eta \in H^\infty(\mathbb{R})$ with $H+ \eta \geq c_{0}>0$, we have
\begin{enumerate}
\item $G[\eta]$ is symmetric on $L^2(\mathbb{R})$:
$$ (G[\eta]u, v)= (u, G[\eta], v), \quad \forall\, u,v \in H^{1\over 2}(\mathbb{R})$$
\item There exists $c>0,$  $C>0$ such that for every $u \in H^{1 \over 2}(\mathbb{R})$
\begin{align}
\label{DNC}
& | (G[\eta]u, v ) | \leq C |\mathfrak P u|_{L^2} |\mathfrak  P v|_{L^2},  \forall u, \, v \in H^{1 \over 2}(\mathbb{R})\\
\label{DNc}
& (G[\eta] u, u ) \geq c |\mathfrak P u|_{L^2}^2, \, \forall u \in H^{1 \over 2 }(\mathbb{R})
\end{align}
where $\mathfrak P$ is the Fourier multiplier 
$$
\mathfrak{P}= (1 - \partial_{x}^2)^{-\frac{1}{4}}\partial_{x}.
$$
\item the linear operator $ G[\eta]: H^{s+1}(\mathbb{R}) \rightarrow H^s(\mathbb{R})$ is   continuous for every $s \in \mathbb{R}$
\item We have the following commutator estimates
$$ | [ \partial_{x}^s, G[\eta]] u  |_{H^{1 \over 2}} \leq C_{s} | \mathfrak{P} u|_{H^s}, \quad \forall u \in H^{s+{1 \over 2}}(\mathbb{R}),$$
$$ |(f\partial_{x}u, G[\eta]u)| \leq  C_{f} |\mathfrak P u|_{L^2}, \quad \forall u \in H^{1\over 2}(\mathbb{R}), \, \forall f \in H^\infty(\mathbb{R}).$$
\item We have the explicit expression for the derivative of the Dirichlet-Neumann operator with respect to $\eta$:
$$ D_{\eta}( G[\eta]u) \cdot\zeta= - G[\eta]( \zeta Z) - \partial_{x}(v \zeta)$$
with 
$$  Z = Z[\eta, u]= { G[\eta ]u + \partial_{x} \eta \partial_{x} u  \over  1 + | \partial_{x} \eta |^2 }, \, v= v[\eta, u]= \partial_{x } u- Z \partial_{x} \eta.$$
\item Finally, for $s>1/2$, we have that
$$\big| D_{\eta}^j (G[\eta] u) \cdot (h_{1}, \cdots, h_{j}) \big|_{H^{s- {1 \over 2}}} \leq  C_{s} \big|\mathfrak P u\big|_{H^s} \, \prod_{i=1}^j |h_{i}|_{H^{s+1}}.$$
\end{enumerate}
\end{prop}
For the proof of these statements we refer to \cite{L} or \cite{RT2} section~3. We have assumed that $\eta$ is smooth since we shall use this proposition
  when $\eta$ is a solitary wave. Most of the estimates are actually still true when $\eta$ only has limited  Sobolev regularity.
  
  In order to perform our  construction  of multisolitons, we shall also need  results about  the action of the Dirichlet-Neumann operator on
   localized functions.
\begin{prop}\label{D-N e-decay estimate} Assume that $\psi\in
C^\infty_{b}(\R)$ has an exponential decay:  
\beq\label{decaypsi}
\exists\, d>0,\,\, \forall \,\al\in\N, \alpha \geq 1,\exists\, C_{\alpha},\forall \, x\in\R,
|\p^\al_x\psi(x)|\le C_\al e^{-d|x|}.
\eeq
Then for $\eta\in H^\infty(\R)$ with $H+ \eta \geq c_{0}>0$, 
$G[\eta]\psi$ also has an exponential decay, that is, for any
$\al\in\N$, there exist a constant $C_\al$ depending on $\al$ and
$0<\eps<d$ independent of $\al$ such that for every $x\in\R$,
\[
\big|\p^\al_x \big(G[\eta]\psi \big)(x)\big|\le C_\al e^{-\eps|x|}.
\]\end{prop}
\bigskip
 \begin{rem} 
We only assume that the derivatives of the function $\psi$ are exponential decaying while the function itself is only  bounded as it is the case for the solitary
 waves (see \eqref{decay2} below). Note that this still yields that   $G[\eta ] \psi$ is  exponentially decaying. The heuristic reason is that 
  the Dirichlet-Neumann operator behaves like a derivative.  Again, in Proposition~\ref{D-N e-decay estimate}, we are not interested 
in the way the estimates depend on the regularity of the surface since we will use it for solitary waves.

The result of Proposition~\ref{D-N e-decay estimate} uses in an essential way that  the fluid domain is bounded in the $z$ direction (we use a Poincar\'e inequality). 
The statement of Proposition~\ref{D-N e-decay estimate} does not hold in the case of an infinite bottom because of a singularity 
at the low frequencies which affects the propagation of the exponential decay. 
\end{rem}
%%%%%%%%
\begin{proof}[Proof of Proposition~\ref{D-N e-decay estimate}]
We use as in \cite{RT} the decomposition
\beq
\label{decDN} \tilde\varphi(x,z)=\varphi_0(x,z)+u(x,z)\
\eeq
with $\varphi_0$  that solves the elliptic problem 
$$ - \Delta_{x,z} \varphi_0=0, \, (x,z) \in \mathcal{S}, \quad  \quad \varphi_0(x,0)= \psi, \, \partial_{z}\varphi_0 (x,-1)=0.$$
This allows  to transform the  elliptic
problem (\ref{elliptic problem-new}) for  $\varphi$ into an elliptic problem with homogeneous boundary conditions
 for  $u$:
\beq
\label{ellipticu}
\left\{\begin{array}{ll} -\na_{x,z}\cdot (P\na_{x,z}u)=\na_{x,z}\cdot
(P\na_{x,z}\varphi_0),\quad\hbox{in}\quad
\cS\\
u|_{z=0}=0,\quad \partial_{z}u|_{z=-1}=0
\end{array}\right.
\eeq
At first, we shall  study the decay properties of  $\varphi_0$. We first observe that $\varphi_0$ is given by the explicit formula
 \beq
 \label{psiHdef} \widehat{\varphi_0}(\xi,z)=\f {\cosh\big( \xi(z+1)\big) }{ \cosh (\xi) }\hat \psi (\xi)\eeq
 where $\,\widehat \cdot\,$ stands for the Fourier transform in the $x$ variable.
  From this expression, we get in particular that
  $$  \mathcal{F}_{x}(\partial_{x}  \varphi_0
  )(\xi,z)=\f {\cosh\big( \xi(z+1)\big)} {\cosh (\xi) } \mathcal{F}_{x}( \partial_{x}\psi )(\xi).$$
  The exponential decay will follow from Paley-Wiener type arguments.
  Since $\partial_{x} \psi$  and its derivatives have exponential decay,  we get that  $\mathcal{F}_{x}( \partial_{x}\psi )$
  has an holomorphic extension to $| \mbox{Im}\, \xi| <d$ and by integration  by parts that it satisfies for every $\delta \in (0, d)$  the estimate
  $$ \big|\mathcal{F}_{x}( \partial_{x}\varphi_0)( \xi, z)  \big| \leq\f {C_{N} }{ 1 +  |\xi|^N}, \quad \forall \xi, \, z , \,   |\mbox{Im }\, \xi |\leq \delta, \, z \in [-1, 0]$$
   for every $N \in \mathbb{N}$.
  Since $\xi \mapsto \f {\cosh\big( \xi(z+1)\big)} { \cosh (\xi) }$  has an holomorphic  bounded  (uniformly in $z$) extension to 
   $|\mbox{Im } \,\xi |\leq  \delta$ for any $\delta \in (0, \pi/2)$,  we can use contour deformation to write that
   $$\partial_{x} \varphi_0(x, z) =  \int_{\mathbb{R}} e^{ix \xi}   \f{\cosh\big( \xi(z+1)\big) }{ \cosh (\xi) } \mathcal{F}_{x}( \partial_{x}\psi )(\xi) \, d \xi
    =  \int_{Im\,  \xi = \delta }   e^{ix \xi}  \f {\cosh\big( \xi(z+1)\big)}{\cosh (\xi) } \mathcal{F}_{x}( \partial_{x}\psi )(\xi) \, d \xi$$
    for any $\delta$ such that  $|\delta|  < \mbox{Min }(d, \pi/2)$. This yields by choosing $\delta =  2\eps \, \mbox{sign }x$, with $\eps$ sufficiently small the estimate
    $$  | \partial_{x} \varphi_0(x, z)| \lesssim  e^{- 2 \eps |x|} , \, \forall (x,z) \in \mathcal{S}.$$
    In a similar way, we get from \eqref{psiHdef} that
    $$ \partial_{z}\hat \varphi_0 (x,z) =  
    \f{\sinh \big( \xi(z+1) \big) }{\cosh \xi} \xi \hat \psi(\xi) =  \f{\sinh \big( \xi(z+1) \big) }{\cosh \xi} \f 1 i\mathcal{F}_{x}(\partial_{x}\psi)(\xi)$$
  and hence, since $\partial_{x} \psi$ and its derivatives have exponential decay, the same arguments as above yield the estimate
  $$  
  | \partial_{z} \varphi_0(x, z)| \lesssim  e^{- 2 \eps |x|} , \, \forall (x,z) \in \mathcal{S}.
  $$
  The estimates  for higher order derivatives  can also be obtained from the same arguments, we find in the end:
  \beq
  \label{estpsiH}
 \forall \beta, \,  | \beta| \geq 1, \quad  | \partial_{x,z}^\beta \varphi_0(x, z)| \lesssim  C_{\beta} e^{- 2  \eps |x|} , \,  \quad \forall (x,z) \in \mathcal{S}.
  \eeq
  It remains to estimate the solution  $u$ of \eqref{ellipticu}.  Let us define  $v(x,z)$ such that 
  \[
u(x,z)=e^{-\eps\langle x\rangle }v(x,z),\quad \langle x\rangle:=(1+x^2)^{\frac{1}{2}}
\] 
with $\eps>0$ small enough and to be fixed later. One has the elliptic problem for
$v(x,z)$  
\beq\label{elliptic problem for v(x,z)}
\left\{\begin{array}{ll} -\na_{x,z}\cdot(
P\na_{x,z}v)-
[\na_{x,z}\cdot P\na_{x,z},\,e^{\eps\langle x\rangle}]e^{-\eps\langle x\rangle}v= e^{\eps\langle x\rangle}\na_{x,z}\cdot
(P\na_{x,z} \varphi_0 ),\quad\hbox{in}\quad
\cS\\
v|_{z=0}=0,\quad \partial_{z }v|_{z=-1}= 0
\end{array}\right.
\eeq 
We shall first estimate   Sobolev norms of $v(x,z)$.

\noindent 1) Lower-order elliptic estimate for $v$.  By
 integration by parts, we get  that \beno & &\int_\cS P\na_{x,z}v\cdot
\na_{x,z}vdxdz \\
& &=\int_\cS 
\Big([\na_{x,z}\cdot P\na_{x,z},\,e^{\eps\langle x\rangle}]e^{-\eps\langle x\rangle}v\Big)
\, v dx dz  +  \int_\cS  e^{\eps\langle x\rangle}
\na_{x,z}\cdot
(P\na_{x,z} \varphi_0  ) \,v\,
dxdz\\
& & := A_1+A_2. \eeno
 For the left hand side,  we have by  the assumption  on $\eta$ that 
\[\int_\cS P\na_{x,z}v\cdot
\na_{x,z}vdxdz\ge c\|\na_{x,z}v\|^2\]
for some $c>0$ independent of $\eps$. Here and in the sequel the norm 
$\| \cdot \|$ is the $L^2(\mathcal{S})$ norm.
Next, we can compute
$$
[\na_{x,z}\cdot P\na_{x,z},\,e^{\eps\langle x\rangle}]
=
(H+\eta)[\partial_x^2,e^{\eps\langle x\rangle}]-2(z+1)\partial_x\eta[\partial_x,e^{\eps\langle x\rangle}]\partial_z
$$
which yields
 $$ |A_{1}| \leq  C \big(  \eps \| \na_{x,z} v \| +  \eps^2 \|v\| \big) \|v\|.$$
 For the second term in the right hand side, we can use \eqref{estpsiH} to get
 $$ |A_{2}| \leq C_{\eps} \|v\| \leq \eps^2 \|v \|^2 + C_{\eps},$$
 where the last inequality comes from the Young inequality,
  for some harmless number $C_{\eps}$ (that depends on $\eps$).
%On the
%other hand, we have \beno A_1&=&\int_\cS
%e^{\eps(1+x^2)^\f12}[\na_{x,z}\cdot
%P\na_{x,z},\,e^{-\eps(1+x^2)^\f12}]v\, vdxdz+\int_\cS
%e^{\eps(1+x^2)^\f12}\na_{x,z}\cdot
%P\na_{x,z}(\chi(z)\psi(x))\,vdxdz\\
%&:=& A_{11}+A_{12}, \eeno and one can see from the above
%computations that
%\[
%A_{11}\le \eps C\|\na_{x,z}v\|_2\|v\|_2
%\] and
%\[ A_{12}\le C\int_\cS e^{-(d-\eps)|x|}|v(x,z)|dxdz\\
%\le C (\eps\|v\|^2+\eps^{-1}) \] with constant $C=C(|\eta|_{H^3})$
%and $\eps<d$. Moreover, we also need to estimate $A_2$. Since the
%lower boundary condition of $v$ reads\beno \p^P_n
%v|_{z=-1}&=&-e^{\eps(1+x^2)^\f12}[\p^P_n,\,e^{-\eps(1+x^2)^\f12}]v|_{z=-1}\\
%&=&
%-e^{\eps(1+x^2)^\f12}\left[(z+1)\p_x\eta\p_x-\f{1+(z+1)^2(\p_x\eta)^2}{1+\eta}\p_z\,,\,
%e^{-\eps(1+x^2)^\f12}\right]v|_{z=-1}\\
%&=&
%-e^{\eps(1+x^2)^\f12}(z+1)\p_x\eta(\p_xe^{-\eps(1+x^2)^\f12})v|_{z=-1}=0,
%\eeno we have $A_2=0$.
 Summing all these estimates up one gets that
\[
\|\na_{x,z}v\|^2\le
C(\eps^2\|v\|^2+\eps\|\na_{x,z}v\| \|v\|+ C_{\eps})
\]
To conclude, we  note that since  $v|_{z=0}=0$ and $\mathcal{S}$ is bounded in the $z$ direction, we have the Poincar$\acute{e}$ inequality:
\[
\|v\|\le C\|\na_{x,z}v\|.
\] Plugging this into the above estimate yields 
\[
\|\na_{x,z}v\|\le C,\quad \|v\|\le C.
\]
by taking $\eps$ sufficiently small.

\noindent 2) Higher-order estimates for $v$. 
These estimates  will follow from an induction argument and standard elliptic regularity  theory.
 Indeed, we can write  the elliptic problem \eqref{elliptic problem for v(x,z)} under
  the form
 $$ - \nabla_{x,z} \cdot \big( P \nabla_{x,z} v) = F, \quad (x,z) \in \mathcal{S}, \quad v(x,0)= 0, \, \partial_{z} v(x,-1)= 0.$$
  From standard elliptic regularity theory, we  have that for $s \geq 0$
  $$ \|v\|_{H^{s+ 2}(\mathcal{S})} \leq  C \big( \|v\|_{H^1(\mathcal{S})} + \|F\|_{H^{s}(\mathcal{S})}  \big).$$ 
  By using the estimate \eqref{estpsiH}  and a standard commutator estimate, we get
  $$ \|F\|_{H^{s}(\mathcal{S})} \leq C \big( 1  +  \|v\|_{H^{s+1}(\mathcal{S})}  \big).$$
  Consequently starting from the $H^1$ estimate that we have already proven, we get  by induction that
   for every $s\geq 0$, 
  $
  %\label{estDNv}
   \| v \|_{H^s(\mathcal{S})} \leq C_{s}.
  $
   By Sobolev embedding, we thus obtain that  $e^{ \eps( 1 + x^2)^\f12} u$ and all its derivatives are bounded in $\mathcal{S}$
   that is to say:
   \beq
   \label{estDNu}
  \forall \beta, \,  | \partial_{x, z}^\beta u (x,z) | \leq C_{\beta} e^{- \eps |x|}, \quad \forall (x,z) \in \mathcal{S}.
  \eeq
To end the proof of Proposition \ref{D-N e-decay estimate}, it suffices to combine the expression \eqref{defDNbis} of the Dirichlet-Neumann operator
(note that it always involves a derivative applied to $\varphi$) with the  decomposition \eqref{decDN} and the estimates \eqref{estpsiH}, \eqref{estDNu}.
This ends the proof of Proposition~\ref{D-N e-decay estimate}.  
\end{proof}
We shall also use  the following corollary.
\begin{cor}\label{D-N e-decay estimate corollary} Let
$\psi\in C^\infty_{b}(\R)$ with an exponential decay property
as in Proposition~\ref{D-N e-decay estimate}.
Then for $\eta\in H^\infty(\R)$ such that $H + \eta \geq c_{0}>0$  and every  $a\in R$, $G[\eta]
(\psi(x-a))$ also has an exponential decay, i.e. there exists
constant $0<\eps<d$ independent of $a$ and $\al$ such that
\[
|\p^\al_xG[\eta](\psi(\cdot -a)) (x)|\le C_\al e^{-\eps |x-a|}\,.
\] 
\end{cor}
\begin{proof}
Let us introduce the translation operator
%\beq
%\label{translation operator}
$
(\tau_{a} f)(x):= f(x-a).
$
%\eeq
The result follows by observing  that 
 $G[\eta] \big( \tau_{a} \psi)=  \tau_{a} \big( G[\tau_{-a} \eta] \psi \big)$
 and by using Proposition~\ref{D-N e-decay estimate}.
 \end{proof}
%The proof is similar as the proof of Prop
%\ref{D-N e-decay estimate} except we replace $\psi(x)$ with
%$\psi(x-a)$ and $e^{-\eps(1+x^2)^\f12}$ with
%$e^{-\eps(1+|x-a|^2)^\f12}$.  \ef
%%%%%%%%%%%%%%%%%%%%%%%%%%%%%%%%%%%%%%%%%%%%%%%%%%%%%%%%%%%%%%%%%%%%%%
%%%%%%%%%%%%%%%%%%%%%%%%%%%%%%%%%%%%%%%%%%%%%%%%%%%%%%%%%%%%%%%%%%%%%%%%
\subsection{Linear stability properties of the solitary wave}
\label{sectionstab}
In this section, we study the linearization of the water-waves system about a solitary wave $Q_{c}(x-ct)$ given by Theorem \ref{theoOS}.
The  main results of this section (in particular Proposition \ref{Propstab1d}) are essentially contained in \cite{Mielke}.
For the sake of completeness, we shall give proofs that use a  slightly different framework which is more adapted  to our purpose.

It is convenient here to go into the moving frame by changing $x$ into $x-ct$.
 As in section \ref{sectionlin1}, the linearized equation reads
 $$ \partial_{t} U=  J \Lambda_{c} U$$
 with
 $$ \Lambda_{c}
= \left(\begin{matrix} -P_c+ g +Z_c
G[\eta_\eps](Z_c\cdot)+Z_c\p_x v_c & (v_c-c)\p_x-Z_c
G[\eta_{c}]\\
-\p_x((v_c-c)\cdot)-G[\eta_c](Z_c\cdot) & G[\eta_c]
\end{matrix}\right)
$$ 
and $Z_{c}= Z[Q_{c}], \, v_{c}= v[Q_{c}]$ are defined  in  (5) Proposition \ref{propDN1}. 
 The operator  $P_{c} $ is the second order  elliptic operator defined by 
 $$ P_{c} = b \partial_{x} \Big( {\partial_x  \cdot \over \big( 1+ (\partial_{x} \eta_{c})^2 \big)^{3 \over 2} } \Big).$$
 We can also   get a simpler form of the linearized system by the change of unknowns
\beq
\label{good-unknown1} V= R_{c} U, \quad  R_{c}=  \left(\begin{array}{cc}  1 & 0 \\  - Z_{c} & 1 \end{array} \right) 
\eeq
 which yields 
\beq\label{linLc} \p_tV=JL_{c}V, \eeq with a  symmetric
operator $L_{c}$  defined  by 
\[
L_{c}=\left(\begin{matrix}-P_c+ g +(v_c-c)\p_x Z_c &
(v_c-c)\p_x\\
-\p_x((v_c-c)\cdot) & G[\eta_c]
\end{matrix}\right).
\] 
The two operators are related by the property
\beq
\label{equivstab}
 L_{c}= (R_{c}^{-1})^t \Lambda_{c} R_{c}^{-1}.
 \eeq
Note that in this section, we see  $c$ as a more natural parameter than  $\varepsilon$  for the solitary wave  and the objects depending on it since
we have not written  the system in a non-dimensional form, but that $L_{c}$ is conjugated to the operator $L(0)$
 studied previously via the scaling transformation \eqref{rescale}.
 
 Thanks to Proposition \ref{propDN1}, the quadratic form associated to $L_{c}$ is naturally defined on the space
  $X^0= H^1(\mathbb{R}) \times \dot H^{1 \over 2}_{*} (\mathbb{R})$
  where $H^{1 \over 2}_{*}(\mathbb{R})$ is a modified homogeneous Sobolev space defined  by
  $$  \dot H^{1 \over 2}_{*}(\mathbb{R}) = \{ u \in \mathcal{S}'(\mathbb{R}), \, \mathfrak Pu \in L^2(\mathbb{R})\}.$$
  \begin{rem}
  \label{remanaf}
  Note that with our definition of $\mathfrak P$, we have that $ \dot H^{1 \over 2}_{*}(\mathbb{R}) \subset L^2_{loc}(\mathbb{R)}$.
  Indeed if $u \in   \dot H^{1 \over 2}_{*}(\mathbb{R})$, then  $v = \mathfrak P u \in L^2$. Let us choose $\chi(\xi)$ a smooth compactly supported
   function such that $\chi = 1$ in the vicinity of zero. Then since 
   \beq
   \label{X0L^21} (1-\chi (D)) u=  \mathfrak P^{-1} (  1 - \chi (D) ) v,\eeq 
    and that $ \frac{1- \chi (\xi)}{\mathfrak P(\xi)}$  is bounded,  we get that 
    $(1 -\chi(D)) u \in L^2(\mathbb{R}).$ Next we note that 
    \beq
    \label{X0L22} \chi u = C + \int_{0}^x (1 - \partial_{x}^2 )^{1 \over 4} (\chi v) \, dy\eeq
    for some constant $C$. Since  $(1 - \partial_{x}^2 )^{1 \over 4} (\chi v) \in L^2(\mathbb{R})$, this yields that $\chi(D) u \in L^2_{loc}(\mathbb{R})$
     and hence that 
     \beq
     \label{X0L2}
     u= \chi u + (1- \chi) u \in L^2_{loc}(\mathbb{R}).\eeq
   \end{rem}
   On $ \dot H^{1 \over 2}_{*}(\mathbb{R})$, we shall use the semi-norm $ |u|_{\dot H^{1 \over 2}_{*}}= | \mathfrak P u |_{L^2}$
    and hence on $X^0$, we set
    $$ |U|_{X^0}= |U_{1}|_{H^1} + |U_{2}|_{\dot H^{1 \over 2}_{*}}.$$
    We could make $X^0$ a Banach space by taking the quotient by the linear space 
    \beq
    \label{0X0}
     0_{X^0}= \{(0, \lambda),\,  \lambda \in \mathbb{R}\}.
     \eeq
      Nevertheless, 
    we refrain from doing it since it is very convenient for us to have that $X^0$ is a subspace of $L^2_{loc}$ (or $\mathcal{S}'$).
\begin{prop}\label{Propstab1d}
Suppose that $c$ satisfies \eqref{restr}. There exists $\varepsilon^*$ such that for every $\varepsilon \in (0, \varepsilon^*]$, we have for 
 $Q_c=(\eta_c,\varphi_c)$ the corresponding solitary wave of speed $c$ that 
 there exists $C>0$ such that for every $U=(U_1,U_2)\in X^0$ such that $(U, J R_{c} \partial_{x} Q_{c})= (U_{1}, \partial_{x} \eta_{c})=0$, 
$$
(L_cU,U)\geq C^{-1}|U|_{X^0}^2.
%-C(U,J\partial_x Q_c)^2-C(U_1,\partial_x\eta_c)^2,
$$
%where 
%$$ 
%L_{c}=\left(\begin{matrix}-P+g+(v_{c}-c)\p_xZ _{c}&
%(v_{c}-c)\p_x\\
%-\p_x((v_{c}-c)\cdot) & G[\eta_{c}]\end{matrix}\right)
%$$  
%with
%$$
%P=b\partial_x\big((1+(\partial_x\eta_c)^2)^{-\frac{3}{2}}\partial_x\cdot\big),\quad
%Z_c=\frac{G[\eta_c]\varphi_c+\partial_x\eta_c\partial_x\varphi_c}{1+(\partial_x\eta_c)^2}
%$$
%and 
%$v_c=\partial_x\varphi_c-Z_c \partial_x \eta_c$.
\end{prop}
The orthogonality condition that appears in this proposition is with respect to  $J R_{c} \partial_{x} Q_{c}$
because of the relation \eqref{equivstab}   since we are dealing with the operator $L_{c}$. 
 If we translate the result into a positivity property in  terms of $\Lambda_{c}$, we recover the natural condition $( J \partial_{x} Q_{c}, U)=0$.
%%%%%%%%%%%%%%%%%%
\begin{rem}
 A more natural orthogonality condition related  to the  general Grillakis-Shatah-Strauss  framework \cite{GSS} would be
  to show that  if  
  \beq
  \label{orthobis} (U, J R_{c} \partial_{x} Q_{c})= (U,  R_{c}\partial_{x} Q_{c})=0
  \eeq
   then
  \beq
  \label{319bis} (L_{c} U, U) \geq C^{-1} |U|_{X^0}^2.
  \eeq
%\begin{equation}\label{faux}
%(L_cU,U)\geq C^{-1}\|U\|_{X^0}^2-C(U,J\partial_x Q_c)^2-C(U,\partial_x Q_c)^2\,.
%\end{equation}
%Unfortunately, such an estimate is false here. 
%Indeed,  let us take two real numbers $\lambda_1\neq 0$ and $\lambda_2\neq 0$ such that the vector $v$ defined as
%$
%v=\lambda_1 \partial_x Q_c+\lambda_2 (0,1)^t
%$
%satisfies 
%$
%(v,\partial_x Q_c)=0.
%$
%We also have
%$$
%(v,J\partial_x Q_c)=0,\quad L_c v=0, \quad \|v\|_{X^0}\geq |\lambda_1|\|\partial_x \eta_c\|_{H^1}> 0.
%$$
%Therefore \eqref{faux} fails for $U=v$.
Nevertheless, note that the condition $(U,  R_{c}\partial_{x} Q_{c})=0$ is not really well defined on $X^0$ because 
  the scalar product $( U_{2}, \partial_{x} \phi_{c})$   would not be uniquely defined  in the quotient since $\varphi_{c}$ has different limits at $\pm \infty$.
  Using this remark differently, we observe that 
   for   $U$ under the form 
   $$ U= \lambda_{1}  R_{c} \partial_{x} Q_{c} + \lambda_{2}(0, 1)^t, $$
     for every $\lambda_{1}$,  one can find some  $\lambda_{2}$ such  that  $(U, J R_{c} \partial_{x} Q_{c})= (U,  R_{c}\partial_{x} Q_{c})=0$. This yields that we  can always find some $U$ such that
    $L_c U=0$ and that satisfies the orthogonality conditions \eqref{orthobis}  but  that $U \notin 0_{X^0}$  since one  can always chose $\lambda_{1} \neq 0$. Therefore,  \eqref{319bis} under \eqref{orthobis} is false.
\end{rem}
%%%%%%%%%%%%%%%%%%%%%%%%%%
\begin{proof}
We first prove the following weaker version of the statement.
\begin{lem}\label{A}
Let $U=(U_1,U_2)\in X^0$,  $U \notin 0_{X^0}$   and such that 
$$
(U,J R_{c}\partial_x Q_c)=(U_1,\partial_x\eta_c)=0.
$$
Then $(L_cU,U)>0$.
\end{lem}
\begin{proof}[Proof of Lemma~\ref{A}]
Lemma~\ref{A} is a variation on \cite[Proposition~4.8]{RT} with different orthogonality conditions.
It turns out that the orthogonality conditions used in \cite[Proposition~4.8]{RT} are not appropriate for our purpose here.
Using the rescaling \eqref{rescale}, we already know thanks to  \cite[Proposition~4.8]{RT} that:
%%%%%%%%%%%%%%%%%%%%%
\begin{prop}\label{0}
There exists a co-dimension two subspace $X^0_1$ 
 of $X^0$ 
and a constant $c_0$ such that for every $U\in X^0_1$, 
$$
(L_cU,U)\geq c_0 |U|_{X^0}^2\,.
$$
As a consequence, thanks to \eqref{equivstab} and using that $R^t_{c} X^0\subset X^0$, we obtain that there exists a 
co-dimension two subspace $X^0_2$ of $X^0$ and a constant $\tilde{c}_0$ such that for every $U\in X^0_2$, 
$$
(\Lambda_c U,U)\geq \tilde{c}_0 |U|_{X^0}^2\,.
$$
\end{prop}
%%%%%%%%%%%%%%%%%%%%%%
We shall now prove  Lemma~\ref{A} by using Proposition \ref{0}. 
Note that because of  \eqref{equivstab} and using that $(R_c^{-1})^tJ=JR_c$, it is equivalent to prove that
 for every $U\notin 0_{X^0}$ and such that $(U,  J\partial_{x} Q_{c})= (U_{1}, \partial_{x} \eta_{c})=0$, one has
 $ (\Lambda_{c}U, U)>0$.
Recall that the solitary wave $Q_c$ is a critical point of the Hamiltonian $H(\eta,\varphi)$, defined by
$$
H(\eta,\varphi)=\frac{1}{2}\int_{-\infty}^{\infty}\Big(G[\eta]\varphi\varphi +g\eta^2+2b(\sqrt{1+(\partial_x \eta)^2}-1)-2c\eta\partial_x\varphi\Big).
$$
Thus $\nabla H(Q_c)=0$. Differentiating the last relation with respect to $c$ and $x$ respectively gives
\begin{equation}\label{difff}
\Lambda_c\partial_c Q_c=J\partial_x Q_c,\quad \Lambda_c\partial_x Q_c=0.
\end{equation}
By contradiction, let us assume  that there exists $y=(y_1,y_2)\in X^0$, $y \notin 0_{X^0}$  and satisfying 
$$
(y,J\partial_x Q_c)=(y_1,\partial_x\eta_c)=0,\quad (\Lambda_cy,y)\leq 0\,.
$$
Set $Y={\rm span}(y,\partial_xQ_c,\partial_c Q_c)$.
We claim that $\dim Y=3$.
In order to prove this claim, we can use  the following lemma.
%%%%%%%%%%%%%%%%%%%%%%%%%%
\begin{lem}\label{B}
Under the assumptions of Proposition~\ref{Propstab1d}, $\partial_xQ_c$ and  $\partial_c Q_c$ are not co-linear and 
\begin{equation}\label{neg}
(\partial_c Q_c,J\partial_x Q_c)<0\,.
\end{equation}
\end{lem}
%%%%%%%%%%%%%%%%%%%%%%%%%
We will give the proof of Lemma~\ref{B} later. Let us use it to prove that the vector space $Y$ is three-dimensional.
If we suppose that
$$
y=\alpha \partial_x Q_c+\beta\partial_c Q_c,\quad \alpha,\beta\in \R
$$
then by taking the scalar product (the distributional duality) with $J\partial_x Q_c$ and $(\partial_x\eta_c,0)^t$, we get
\begin{equation}\label{dd}
0=(y,J\partial_x Q_c)=\beta (\partial_c Q_c,J\partial_x Q_c),\quad
0=(y_1,\partial_x\eta_c)=\alpha\|\partial_x \eta_c\|_{L^2}^2+\beta(\partial_c\eta_c,\partial_x\eta_c).
\end{equation}
Using Lemma~\ref{B} and \eqref{dd}, we obtain that $\alpha=\beta=0$. Therefore $\dim Y=3$.

Now, a similar argument shows that $Y\cap 0_{X^0}=\{(0,0)^t\}$. Indeed, it suffices to use that $(0,1)^t$ is orthogonal to 
$J\partial_x Q_c$ and $(\partial_x\eta_c,0)^t$ and to take the scalar product with these vectors in a relation of type
$$
\lambda_1y+\lambda_2 \partial_x Q_c+\lambda_3\partial_c Q_c=(0,1)^t,\quad \lambda_1,\lambda_2,\lambda_3\in\R.
$$
Next, for $\mu_1,\mu_2,\mu_3\in\R$, we can write by invoking Lemma~\ref{B} and \eqref{difff},
$$
\big(\Lambda_c(\mu_1 y+\mu_2 \partial_x Q_c+\mu_3\partial_c Q_c),\mu_1 y+\mu_2 \partial_x Q_c+\mu_3\partial_c Q_c\big)
=\mu_1^2(\Lambda_cy,y)+\mu_3^2(\partial_c Q_c,J\partial_x Q_c)\leq 0\,.
$$
Therefore $(\Lambda_cU,U)\leq 0$ for every $U\in Y$. 
But since $\dim Y=3$ there exists a nontrivial $z\in Y$ such that $z\in X^0_2$ where $X^0_2$ 
is the space involved in the statement of Proposition~\ref{0}.
%(a nontrivial combination of the three vectors spanning $Y$ can become orthogonal to the two vectors involved in the orthogonality conditions defining $X$).
Therefore $|z|_{X^0}=0$. This implies that $z\in 0_{X^0}$. But $z\in Y$ and 
$Y\cap 0_{X^0}=\{(0,0)^t\}$ which implies $z=0$. Contradiction.
This ends  the proof of Lemma~\ref{A}.
\end{proof}
%%%%%%%%%%%%%%%%%%%%%%%%%%%%%%%%%%%%%
\begin{proof}[Proof of Lemma~\ref{B}]
Thanks to \cite{AK} (see also Theorem \ref{theoOS}), we have by setting $X=x/H$
\begin{equation}\label{eta}
{\eta_c \over H}(x)=-\varepsilon^2{\rm ch}^{-2}\Big(
\frac{\varepsilon X}{2(\beta-1/3)^{\frac{1}{2}}}
\Big)+{\mathcal O}(\varepsilon^4 e^{-c\varepsilon |X|})
\end{equation}
and 
\begin{equation}\label{phi}
{\varphi_c \over c\,H }(x)=-2\varepsilon
(\beta-1/3)^{\frac{1}{2}}
{\rm th}^{-2}\Big(
\frac{\varepsilon X}{2(\beta-1/3)^{\frac{1}{2}}}
\Big)+{\mathcal O}(\varepsilon^3),
\end{equation}
where $\beta=\frac{b}{Hc^2}$ and $\varepsilon=\sqrt{\frac{gH}{c^2}-1}$.
Then $\partial_c\varepsilon<0$ and using that
\begin{equation}\label{iii}
(\partial_c Q_c,J\partial_x Q_c)=
\partial_c
\int_{-\infty}^{\infty}\eta_{c}(x)\partial_x \varphi_c(x)dx
\end{equation}
we deduce \eqref{neg} by substituting \eqref{eta} and \eqref{phi} in \eqref{iii}.
Finally, one also deduces from \eqref{eta} and \eqref{phi}  that $\partial_x Q_c$ and $\partial_c Q_c$ are not co-linear.
This completes the proof of Lemma~\ref{B}.
\end{proof}
%%%%%%%%%%%%%%%%%%%
Let us now come back to the proof of Proposition~\ref{Propstab1d}. 
%It suffices to prove that
%$$
%(L_cU,U)\geq C^{-1}\|U\|_{X^0}^2\, ,
%$$
%for every $U\in X^0$ satisfying
%$
%(U,J R\partial_x Q_c)=(U_1,\partial_x\eta_c)=0\,.
%$
We argue by contradiction.
Suppose that there exists a sequence $(U^n)$ such that 
\beq
\label{hypcont}
|U^n|_{X^0}=1,\quad \lim_{n\rightarrow\infty}(L_c U^n,U^n)=0,\quad
(U^n,J R_{c}\partial_x Q_c)=(U^n_1,\partial_x\eta_c)=0\,.
\eeq
Then (up to the extraction of a subsequence),  there exists $\tilde{U}_1\in H^1$ such that $(U^n_1)$ converges weakly in $H^1$ to $\tilde{U}_1$, and,  by using Remark \ref{remanaf}, there exists
$\tilde{U}_{2} \in \dot H^{1\over 2}_{*}$ such that $U^n_{2}$  converges weakly in $L^2_{loc}$  towards  $\tilde U_{2}$ and $\mathfrak P  U^n_{2}$  converges weakly in $L^2$ 
 towards  $\mathfrak P \tilde U_{2}$.
% $V\in L^2$ such that
%$
%((1+|D_x|^{1/2})^{-1}\partial_x U^n_2)
%$
%converges weakly in $L^2$ to $V$.
%Set 
%$
%\tilde{U}_2=(1+|D_x|^{1/2})\int_{0}^x V(y)dy
%$
%so that
%$
%(1+|D_x|^{1/2})^{-1}\partial_x \tilde{U}_2=V.
%$
Next, we set $\tilde{U}=(\tilde{U}_1,\tilde{U}_2)\in X^0$.
We have that $\tilde{U}$ satisfies the same orthogonality conditions as $U^n$:
%%%%%%%%%%
%\begin{lem}\label{OC-lim}
%One has
\beq
\label{orthoproof}
(\tilde{U},J R_{c}\partial_x Q_c)=(\tilde{U}_1,\partial_x\eta_c)=0\,.
\eeq
%\end{lem}
%%%%%%%%%
%\begin{proof}
Indeed, the second assertion is a direct consequence of the weak $L^2$ convergence of  $U^n_1$ to $\tilde{U}_1$.
In order to prove the first one, we write
$$
0=(U^n,J R_{c} \partial_x Q_c)=(U^n_1,\partial_x\varphi_c- Z_{c}\partial_{x} \eta_{c})-(U_2^n,\partial_x\eta_c)\,.
$$
From  the  weak $L^2$ convergence of  $(U^n_1)$ to $\tilde{U}_1$, we obtain that 
$$
(U^n_1,\partial_x\varphi_c- Z_{c}\partial_{x} \eta_{c}) \rightarrow (\tilde{U}_1,\partial_x\varphi_c - Z_{c} \partial_{x} \eta_{v})
$$
and by observing that 
\beq
\label{trick1}
 (U_2^n,\partial_x\eta_c)=-\Big(\mathfrak P U_2^n,(1- \partial_{x}^2)^{1/4}\eta_c\Big), 
\eeq
since $(1- \partial_{x}^2)^{1/4}\eta_c\in L^2$, we also get by weak $L^2$ convergence that
$$
 (U_2^n,\partial_x\eta_c) \rightarrow  -\Big(\mathfrak P \tilde U_{2},(1- \partial_{x}^2)^{1/4}\eta_c\Big)= (\tilde{U}_{2}, \partial_{x}\eta_{c})$$
% converges to
% $$
% -(V,(1+|D_x|^{1/2})\eta_c)
% =
% -\Big(\frac{\partial_x \tilde{U}_2}{1+|D_x|^{1/2}},(1+|D_x|^{1/2})\eta_c\Big)
%=
%(\tilde{U}_2,\partial_x\eta_c)
%$$
(the last scalar product  is well  defined thanks to the exponential decay of $\partial_x\eta_c$).
%This completes the proof of Lemma~\ref{OC-lim}.
%\end{proof}
%%%%%%%%%%%%
This proves that $\tilde U$ verifies \eqref{orthoproof}.

Let us now complete the proof of Proposition~\ref{Propstab1d}.
Let us set 
$$ \mathcal{E}_{c}= -  b \partial_{x}\big( \zeta_{c}  \partial_{x} \cdot \big) + g + (v_{c}- c) \partial_{x} Z_{c}, \quad \zeta_{c}=  ( 1 +  |\partial_{x} \eta_{c} |^2 )^{- {1 \over 2 } },  $$
we shall first conjugate to get a leading order part with constant coefficients:
$$\tilde L_{c}= A L_{c} A, \quad   A =
 \left( \begin{array}{cc} m_{c}  & 0 \\ 0 & 1 \end{array} \right),  \quad m_{c}= {1 \over \sqrt{\zeta_{c}}}$$
 hence
 $$\tilde{L}_{c}= \left( \begin{array}{cc}  m_{c} \mathcal{E}_{c}( m_{c}  \cdot) & (v_{c}- c) \, m_{c} \partial_{x} \\  -  \, \partial_{x}( (v_{c}- c)m_{c} \cdot) &  G[\eta_{c}] \end{array} \right).$$ 
Let us set
$$ V_{1}^n=  {1 \over m_{c}} U_{1}^n, \quad V_{2}^n = U_{2}^n, $$
so that
 \beq
 \label{LLtilde}\big(L_{c} U^n, U^n\big)= \big( \tilde{L}_{c} V^n, V^n \big).\eeq
Note that we still have that
$ V_{1}^n$  is  bounded in $H^1$. We can thus assume (after extracting a subsequence) that $V_{1}^n$
 converges strongly in $H^s_{loc}$ for every $s<1$ towards $\tilde V_{1}= 1/m_{c} \tilde U_{1}.$
 For convenience, we also set $\tilde{V}_{2}= \tilde{U_{2}}$ so that $V_{2}^n$ converges weakly in $L^2_{loc}$ towards $\tilde{V}_{2}$ and that  $\mathfrak P V_{2}^n$
  converges weakly in $L^2$ towards $ \mathfrak P \tilde{V}_{2}.$

Next,  we  observe  that we can write the decomposition
 \beq
 \label{Ltildedec}\tilde L_c=L_1+K_{1}, \, L_{1}= \left(\begin{array}{cc}  - b \partial_{x}^2+  g & - c\partial_{x} \\ c\partial_{x} &  G[\eta_{c}] \end{array} \right)\eeq
  and 
$$ K_{1}=  \left( \begin{array}{cc}  (v_c-c) m_{c}^2\partial_x Z_c + g(m_{c} -1) + b \,m_{c} \partial_{x}( m_{c} \zeta_{c}) \partial_{x} +  b \,m_{c} \partial_{x} \big(\zeta_{c}\partial_{x} m_{c} \cdot \big) &   \big(v_{c} m_{c}- c(m_{c}-1) \big) \partial_{x} \\ - \partial_{x}((v_{c} m_{c}  + c( 1- m_{c}) )\cdot) & 0 \end{array} \right).$$
 For   $K_1$,  which is a relatively compact perturbation,   we obtain
 \beq
 \label{K1lim}
\lim_{n\rightarrow \infty} (K_1V^n,V^n)=(K_1\tilde{V},\tilde{V}).
\eeq
Indeed, we can use the same trick as in \eqref{trick1}, the exponential decay of  $v_{c}$,  $\partial_{x} Z_{c}$, $m_{c}-1$, $\partial_{x} m_{c}$, $\partial_{x}^2 m_{c}$  and the fact that  $V_{1}^n$  converges strongly  in $H^s_{loc}$ for every $s<1.$

To analyze $L_{1}$, we use the following factorization  of  its associated quadratic form: for every $V$
$$ (L_{1} V, V) =  \big((- b\, \partial_{x}^2 + g - c^2 M^{-1} )V_{1}, V_{1} \big) +  \big( M \big( \partial_{x} V_{2} - c M^{-1} V_{1} \big), \partial_{x} V_{2} 
 - cM^{-1} V_{1} \big)$$
 where 
 $$ M= - \partial_{x}^{-1} G[ \eta_{c} ] \partial_{x}^{-1}.$$
 Note that $M$ is  a well-defined  operator (of order $-1$), positive  and  invertible thanks to  Lemma 4.5 of \cite{RT}.
 Let us set 
 $$M_{0}= - \partial_{x}^{-1} G[ 0] \partial_{x}^{-1}, $$
 we can rewrite
 $$  (L_{1} V, V) =  (\tilde{L}_{1} V, V)  - c^2  \big((M^{-1} - M_{0}^{-1}) V_{1}, V_{1} \big).$$
 with
  \beq
  \label{tildeL1fact} (\tilde{L}_{1} V, V)= \big((- b\,  \partial_{x}^2 + g - c^2 M _{0}^{-1}) V_{1}, V_{1} \big) +  \big( M \big( \partial_{x} V_{2} - c M^{-1} V_{1} \big), \partial_{x} V_{2} 
 - cM^{-1} V_{1} \big).\eeq
 Note that $M- M_{0}$ is a compact operator on $H^1$ (see the remark at the bottom of p. 295 and the top of p. 296 in \cite{RT}), consequently, 
 we have
 \beq
 \label{Mlim} \lim_{n}   \big((M - M_{0}) V_{1}^n, V_{1}^n \big) =   \big((M - M_{0}) \tilde V_{1}, \tilde V_{1} \big).
 \eeq
 To pass to the limit in $(\tilde L_{1} V^n, V^n)$, we note that $\tilde{L}_{1}$ is a non-negative operator. Indeed, $M$ is nonnegative
  and the first part of $\tilde{L}_{1}$ is a Fourier multiplier with symbol
  $$m (\xi) = b \xi^2  + g - c^2 {  \xi \over \mbox{tanh}  H\xi} = {c^2 \over H} \Big( \beta (H \xi)^2 + \alpha  - { H\xi \over  \mbox{tanh}   H\xi} \Big)$$
  which is positive since $\beta >1/3$ and  $ \alpha >1$ thanks to \eqref{restr}.  Consequently, we get from the weak convergence properties that
  $$ 
\liminf_{n\rightarrow\infty}(\tilde L_1 V^n, V^n)\geq  (\tilde L_1 \tilde{V}, \tilde{V}).
$$
  
  Gathering the previous transformations, we thus get that
\begin{align}
\label{finabs}
\liminf_{n\rightarrow\infty}  \big( L_{c} U^n , U^n \big) =  \liminf_{n\rightarrow\infty}  \big( \tilde L_{c} V^n , V^n \big) 
 &  \geq  \big(\tilde L_{1}\tilde V, \tilde V \big)  - c^2  \big((M - M_{0}) \tilde V_{1}, \tilde  V_{1} \big)  + \big( K_{1} \tilde V , \tilde V) \\
   & = \big( \tilde{L}_{c} \tilde V, \tilde V \big)  = \big( L_{c} \tilde U, \tilde U \big).
   \end{align}
   Consequently, thanks to \eqref{hypcont}, we find that  $\big( L_{c} \tilde U, \tilde U \big) \leq 0$
  Since $\tilde U$ verifies \eqref{orthoproof}, we get thanks to 
  Lemma~\ref{A}   that $\tilde{U} \in 0_{X^0}$ that is to say $ \tilde{U}=(0,\lambda)^t$ for some constant
$\lambda\in \mathbb{R}$.
Since $K_1$ is such that $(K_1((0,\lambda)^t),(0,\lambda)^t)=0$, we obtain by using again  \eqref{LLtilde}, \eqref{Ltildedec}, \eqref{K1lim}, \eqref{Mlim}
  that as $n\rightarrow \infty$,
$$
(L_{c}U^n,U^n)=(\tilde L_1 V^n,V^n)+o(1).
$$
Moreover, since the symbol $m(\xi)$ is positive and  the operator $M$ is non-negative we get from
 \eqref{tildeL1fact} that for some $c_{0}>0$ independent of $n$, we have
 $$ (L_{c}U^n,U^n) \geq c_{0} \|V^n_{1} \|_{H^1}^2 +o(1).$$
 This yields that $V^n_{1}$ converges strongly to $0$ in $H^1$. This is sufficient to also obtain that
 $$  \lim_{n}(\partial_{x} V_{2}^n, V_{1}^n)=0$$
 and hence we find by using again \eqref{LLtilde}, \eqref{Ltildedec} that
 $$ \big( L_{c} U^n, U^n)= (G[\eta_{c}] U_{2}^n, U_{2}^n) + o(1).$$
 Thanks to \eqref{DNc} in Proposition \ref{propDN1}, this implies that 
 $ \mathfrak{P} U_{2}^n$ converges strongly to zero in $L^2$.
  We have thus obtained that $ \lim_{n }\|U^n \|_{X^0}= 0$ which contradicts  the  assumption that
   $\|U^n\|_{X^0}=1$ in \eqref{hypcont}.

This ends  the proof of Proposition~\ref{Propstab1d}.
\end{proof}
%%%%%%%%%%%%%%%%%
From Proposition  ~\ref{Propstab1d}, one can get that
\begin{prop}
\label{Propstab1dbis}
For every $U \in X^0$ there  exists a unique decomposition
\beq
\label{decstab} U= \alpha J R_{c} \partial_{x} Q_{c} + \beta R_{c} \partial_{x} Q_{c} + V\eeq
with $V \in X^0$ such that 
\beq
\label{orthoV}
(V, J R_{c} \partial_{x} Q_{c})= (V_{1}, \partial_{x} \eta_{c}) = 0.
\eeq
 Moreover, there exists $c_{0}>0$ and $C>0$ such that for every $U \in X^0$  written under the form \eqref{decstab}, one has
 $$ (L_{c} U, U) \geq c_{0}|V|_{X^0}^2 -  C | \alpha|^2. $$
\end{prop}
Note that in the decomposition \eqref{decstab}, $V$ is not orthogonal to the  $R_{c} \partial_{x} Q_{c}$.
This decomposition has better properties than the orthogonal decomposition that one would get from proposition \ref{Propstab1d}
 by choosing $V$ orthogonal to $JR_{c}\partial_{x} Q_{c}$ and $( \partial_{x} \eta_{c}, 0)^t$. The reason is that
  $R_{c} \partial_{x} Q_{c}$ is in the kernel of $L_{c}$ while  $( \partial_{x} \eta_{c}, 0)^t$ is not.

\begin{proof}[Proof of Proposition~\ref{Propstab1dbis}]
If
   \beq
  \label{deux} U= \alpha JR_{c} \partial_{x} Q_{c}+ \beta   R_{c} \partial_{x}Q_{c} +  V\eeq
  with  $V $ satisfying \eqref{orthoV} then  $\alpha$  and  $\beta$  are necessarily  determined by  
  \beq
  \label{trois} \alpha={ ( U, J R_{c}\partial_{x} Q_{c}) \over  |J R_{c} \partial_{x}Q_{c}|_{L^2}^2},  \quad
   \beta = { (U, (\partial_{x}\eta_{c},0)^t ) \over (R_{c}\partial_{x}Q_{c}, (\partial_{x} \eta_{c}, 0)^t ) }
- \alpha (J R_{c}\partial_{x} Q_{c}, (\partial_{x} \eta_{c},0)^t)\eeq
and thus are well-defined since  $(R_{c}\partial_{x}Q_{c}, (\partial_{x}\eta_{c}, 0)^t ) = | \partial_{x} \eta_{c}|_{L^2}^2 \neq 0$.
This proves the uniqueness and one directly verifies that with $\alpha,\beta$ defined by \eqref{trois}, the function $V$ in \eqref{deux} satisfies the orthogonality
conditions \eqref{orthoV}.   

Next, since $L_{c}  R_{c}\partial_{x} Q_{c}= \Lambda_{c} \partial_{x} Q_{c}= 0$, we get by using the decomposition \eqref{deux} that 
  $$ (L_{c}U,U)=  \alpha^2 (L_{c}J R_{c} \partial_{x} Q_{c}, J R_{c}\partial_{x}Q_{c}) + 2 \alpha (V, L_c J R_{c}\partial_{x} Q_{c}) + (L_{c}V, V).$$
  Therefore, by using Proposition \ref{Propstab1d} for $V$, we obtain
  $$ (L_{c}U, U) \geq {c_{0} \over 2} \|V\|_{X^0}^2  - C \alpha^2.$$
  This completes the proof of Proposition~\ref{Propstab1dbis}.
\end{proof}
As a simple corollary, we can get stability estimates in $X^0$ for the linearized equation
\beq
\label{stabfin} \partial_{t} U= JL_{c}U, \quad U_{/t=0}= U^{0}.\eeq
%%%%
\begin{cor}\label{linn}
There exists $C>0$ such that for every $U^{0} \in X^0$, the solution of  $U(t)$ of \eqref{stabfin} satisfies the estimate
$$ |U(t)|_{X^0} \leq C (1 + |t|)|U^{0}|_{X^0}.$$
\end{cor}
\begin{proof}
From the explicit expression   \eqref{trois} for $\alpha$, we get $ \partial_{t} \alpha = 0$, and therefore
\begin{equation}\label{alpha}
|\alpha(t)|= |\alpha(0)|\leq C(\|U^0_1\|_{L^2}+|(U^0_2,\partial_x\eta_c)|)\leq C|U^0|_{X^0}\,.
\end{equation}
 Since $U$ solves \eqref{stabfin}, we observe that 
  $\partial_t (L_{c}U, U)= 0.$
  Therefore, we get from Proposition \ref{Propstab1dbis} and \eqref{alpha} that
  $$ 
  |V(t) |^2_{X^0} \leq  C\big((L_{c}U^0, U^0)+|\alpha(t)|^2)
\leq    C|U^0|^2_{X^0}\,.
  $$
 Moreover, from the explicit expression \eqref{trois} for  $\beta$, we find that   
   $$ \partial_{t} \beta =   { (JL_{c}U, (\partial_{x}\eta_{c},0)^t ) \over (R_{c}\partial_{x}Q_{c}, (\partial_{x}\eta_{c}, 0)^t ) }  
    = 
   { (JL_{c}(\alpha J R_{c}\partial_{x} Q_{c} + V), (\partial_{x }\eta_{c},0)^t ) \over (R_{c}\partial_{x}Q_{c}, (\partial_{x} \eta_{c}, 0)^t ) }.$$
 Hence
   $$ 
   |\beta(t) | \leq C|U^0|_{X^0}\ + C\int_{0}^t (|\alpha(s)| + |V(s)|_{X^0} ) ds \leq C (1 + |t|)|U^{0}|_{X^0}.
$$
 This completes the proof of Corollary~\ref{linn}  
\end{proof}
%%%%%%%%%%%%%%%%%%%%%%%%%%%%%%%%%%%%%%
\section{Error produced by a sum of solitary waves}\label{section error}
We shall now begin the construction that will allow to get Theorem \ref{multi-soliton theorem}. We denote  the
two different soliton solutions for the water-wave system as 
(\ref{water-wave system}) as
\beno 
Q_{c_1}(x-c_1t)=
Q_1(x-c_1t)&=& (\eta_1(x-c_1t),\varphi_1(x-c_1t))^t\\
Q_{c_2}(x-h-c_2t)=Q_2(x-h-c_2t)&=& (\eta_2(x-h-c_2t),\varphi_2(x-h-c_2t))^t, \eeno 
and we set
$$ 
M(t,x)= Q_{1}(x-c_{1}t) + Q_{2}(x- h - c_{2}t ):= (\eta_M(t,x),\varphi_M(t,x))^t
$$
to be the superposition of the two solitary waves.
Sometimes we  shall simply write $Q_1=(\eta_1,\varphi_1)^t$ and
$Q_2=(\eta_2,\varphi_2)^t$ for convenience.
Thanks to Theorem~\ref{theoOS}, we have 
\beq
\label{decay1}
 \exists d>0, \, \forall \alpha, \, \exists C_{\alpha},\, \forall x \in \mathbb{R},  \quad |\partial_{t, x}^\alpha \eta_{i}| \leq C_{\alpha} e^{- d |x - c_{i} t - (i-1)h| }, \quad i=1, \, 2
\eeq 
and
\beq
\label{decay2}
 \exists d>0, \, \forall \alpha\in \mathbb{N},\, \alpha \geq 1, \, \exists C_{\alpha},\, \forall x \in \mathbb{R},  \quad |\partial_{t,x}^\alpha \varphi_{i}| \leq C_{\alpha} e^{- d |x - c_{i} t - (i-1 )h| }, \quad i=1, \, 2.
\eeq 
Note that $\varphi_{i}$ is bounded but not exponentially decaying, nervetheless, derivatives of $\varphi_{i}$ are exponentially decaying.
In this section, we shall establish that $M$ solves the water-waves system \eqref{water-wave system}
 up to a small exponentially decaying term:
\begin{prop}\label{M(t,x) system} The two-soliton
$M(t,x)$  solves:
\[\left\{\begin{array}{ll}
\p_t\eta_M=G[\eta_M]\varphi_M+R_1\\
\p_t\varphi_M=-\f12|\p_x\varphi_M|^2+\f12\f{(G[\eta_M]\varphi_M+\p_x\varphi_M\p_x\eta_M)^2}
{1+|\p_x\eta_M|^2}
-g\eta_M+b\p_x\left(\f{\p_x\eta_M}{\sqrt{1+|\p_x\eta_M|^2}}\right)+R_2
\end{array}\right. \] where the remainder $R_M(t,x):=(R_1,R_2)^t$ has an exponential
decay in time, that is, there exist constants $C_s$ and
$\eps_0>0$ 
%with $\eps$ from Prop \ref{D-N e-decay estimate}
such that for any $s\geq 0$
\[|R_M(t)|_{E^s}\le C_s e^{-\eps_0h}e^{-\eps_0(c_2-c_1)t},\quad \forall t\ge 0.\] 
\end{prop}

Let us recall the notation 
$$ | U(t) |_{E^s} = \sum_{| \alpha | \leq s} | \partial_{t,x}^\alpha U |_{L^2}.$$ 

Note that in the sequel we use again  both  $\varepsilon$ and $\eps$ for different parameters.

 In order to prove Proposition~\ref{M(t,x) system}, the main difficulty is to study the interaction of the two solitary waves  via the Dirichlet-Neumann operator,
 i.e we need to study for example $G[\eta_{1}] \varphi_{2}$, thus the crucial ingredient will  be Proposition \ref{D-N e-decay estimate}.
  %%%%%%%%%%%%%%%%%%%%%%%%%%%%%%%%%%%%%%%%%%%%%%%%%%%%% 
\subsection{Proof of Proposition \ref{M(t,x) system}}
The basic idea is that  the interaction between the solitary waves is weak because they are localized and far way with different speeds.
We shall use many times in the proof the following elementary lemma:
\begin{lem}
\label{leminteraction}
 For $c_{2}>c_{1}$, $\eps >0$ and $\eps_{0}\in (0, \eps)$,  there exists $C>0$  such that for every $h \geq 0$, $t \geq 0,$ 
 $$ \int_{\mathbb{R}} e^{- \eps |x- c_{1} t|} e^{- \eps | x - h - c_{2} t|} \, dx \leq C e^{ - \eps h} e^{-\eps_{0}(c_{2}- c_{1}) t}.$$
\end{lem}
\begin{proof}
It suffices to decompose the integration domain in the three regions
 $x \geq h + c_{2}t$, $x \leq c_{1}t$ and $ c_{1}t \leq x \leq h + c_{2} t$.
 \end{proof}
 
  Let us  prove Proposition \ref{M(t,x) system}. 
 Since $Q_1$ and $Q_2$ are two solutions  of 
(\ref{water-wave system}),  we  can sum up the first equations  of the two  systems  to get
\beno
\p_t\eta_M&=&\p_t\eta_1+\p_t\eta_2=G[\eta_1]\varphi_1+G[\eta_2]\varphi_2\\
&=&
G[\eta_M]\varphi_M+G[\eta_1]\varphi_1-G[\eta_M]\varphi_1+G[\eta_2]\varphi_2-G[\eta_M]\varphi_2.\eeno
So we have the first equation for $M$:
\[
 \p_t\eta_M=G[\eta_M]\varphi_M+R_1
 \] with
$R_1=G[\eta_1]\varphi_1-G[\eta_M]\varphi_1+G[\eta_2]\varphi_2-G[\eta_M]\varphi_2$.
Next we need to estimate $R_1$. Using the shape-derivative formula
for  the Dirichlet-Neumann operator (see (5) Proposition \ref{propDN1}),  we can compute that
\beno G[\eta_1]\varphi_1-G[\eta_M]\varphi_1&=& -\int^1_0D_\eta
G[\eta_1+s\eta_2]\varphi_1\cdot \eta_2ds\\
&=&
\int^1_0\left[G[\eta_1+s\eta_2](\eta_2Z_{1s})+\p_x(\eta_2(\p_x\varphi_1
-Z_{1s}\p_x(\eta_1+s\eta_2))\right]ds \eeno where
\[
Z_{1s}=\f{G[\eta_s]\varphi_1+\p_x\eta_s\p_x\varphi_1}{1+|\p_x\eta_s|^2}
\] with the notation $\eta_s=\eta_1+s\eta_2$. Next, we can use Corollary~\ref{D-N e-decay estimate corollary} to get
\[
|\p^\al_{t,x}(G[\eta_s]\varphi_1)|\le C_\al e^{-\eps|x-c_1t|}.
\quad\hbox{for}\quad \al\in\N,
\] 
Indeed, since the solitary waves depend on $x-c_{i}t$, one can always convert a time derivative into a space derivative.
We thus get
\[
|\p^\al_{t,x}Z_{1s}|\le C_\al e^{-\eps|x-c_1t|}, \quad\hbox{for}\quad
\al\in\N.
\] 
With these estimates and using Lemma~\ref{leminteraction}, we get
\beno 
|G[\eta_1]\varphi_1-G[\eta_M]\varphi_1|_{E^s}&\le &
C\int^1_0\left(|\eta_2Z_{1s}|_{H^{s+1}}+|\eta_2\p_x\varphi_1|_{E^{s+1}}
+|\eta_2Z_{1s}|_{E^{s+1}}\right)ds \\
&\le& C_s e^{-\eps_0h}e^{-\eps_0(c_2-c_1)t},\quad \hbox{for}\quad
t\ge 0.
\eeno 
Similarly we have the estimate for
$|G[\eta_2]\varphi_2-G[\eta_M]\varphi_2|_{E^s}$. Summing these two
estimates up leads to
\[
|R_1|_{E^s}
 \le C_s e^{-\eps_0h}e^{-\eps_0(c_2-c_1)t},\quad \hbox{for}\quad
t\ge 0.
\] Now we deal with the equation for $\varphi_M$. From the second
equation of system (\ref{water-wave system}) for both $Q_1$ and
$Q_2$ one can write down that \beno
\p_t\varphi_M&=&\p_t\varphi_1+\p_t\varphi_2\\
&=&-\f12|\p_x\varphi_M|^2+\f12\f{(G[\eta_M]\varphi_M+\p_x\varphi_M\p_x\eta_M)^2}
{1+|\p_x\eta_M|^2}
-g\eta_M+b\p_x\left(\f{\p_x\eta_M}{\sqrt{1+|\p_x\eta_M|^2}}
\right)+R_2 \eeno where $R_2=R_{21}+R_{22}+R_{23}$ with \beno
R_{21}&=&\f12|\p_x\varphi_M|^2-\sum^2_{i=1}\f12|\p_x\varphi_i|^2=\p_x\varphi_1\p_x\varphi_2,\\
R_{22}&=&-\f12\f{(G[\eta_M]\varphi_M+\p_x\varphi_M\p_x\eta_M)^2}
{1+|\p_x\eta_M|^2}+\sum^2_{i=1}\f12\f{(G[\eta_i]\varphi_i+\p_x\varphi_i\p_x\eta_i)^2}
{1+|\p_x\eta_i|^2},\\
R_{23}&=&-b\p_x\left(\f{\p_x\eta_M}{\sqrt{1+|\p_x\eta_M|^2}} \right)
+b\sum^2_{i=1}\p_x\left(\f{\p_x\eta_i}{\sqrt{1+|\p_x\eta_i|^2}}
\right). \eeno 
With the same arguments that we have used for the estimate of  $R_1$, we get  that $R_2$
also satisfies the same exponential-decay estimate
\[
|R_2|_{E^s}
 \le C_s e^{-\eps_0h}e^{-\eps_0(c_2-c_1)t},\quad \hbox{for}\quad
t\ge 0.
\] 
This ends the proof of Proposition~\ref{M(t,x) system}.
\ef
%%%%%%%%%%%%%%%%%%%%%%%%%%%%%%%%%%%%%%%%%%%%%%%%%%%%%%%%%%%%%%%%%%%%
%%%%%%%%%%%%%%%%%%%%%%%%%%%%%%%%%%%%%%%%%%%%%%%%%%%%%%%%%%%%%%%%%%%%
\section{Construction of  an approximate solution} \label{roximate}
If we take
$\del=e^{-\eps_0h}>0$,  $\del$ will be small enough if we choose
$h>0$ large enough later. The remainder  $R_M$ in the system for
$M(t,x)$ can be rewritten as $R_M=\del\tilde R_M$ with
\beq\label{estimate R_M}|\tilde R_M|_{H^s}
 \le C_s e^{-\eps_0(c_2-c_1)t},\quad \hbox{for}\quad
t\ge 0.\eeq Let
\[
V(t,x)=\sum^N_{l=1}\del^lV_l(t,x)
\] with unknowns $V_l(t,x)$ to be constructed later.
 We want to show that $U^a(t,x)=M(t,x)+V(t,x)$
 is an approximate solution for the water-wave system under the
 following sense:
%%%%%%%%%%%%%%%%%%%%%%%%%%%%%%%%%%%%%%%%%%%%%
\begin{prop}\label{approximate solution M+V} For any
$N\in\N$, there exists
\[
U^a(t,x)=M(t,x)+V(t,x)=M(t,x)+\sum^N_{l=1}\del^l V_l,\quad\hbox{with}\quad
V_l\in C^\infty (\R_+,H^\infty(\R))
\] and a small constant $\del=e^{-\eps_0h}>0$ such that for every $l$, one has
the estimates\[
|V_l(t)|_{E^k}\le h^\f{2l-1}4C_{k,l}e^{-l\eps_0(c_2-c_1)t},\qquad
\forall t\ge0
\] with some constant $C_{k,l}$.
 Moreover, $U^a$ is an approximate solution of
(\ref{simple form for WW system}) in the sense that\[
\p_tU^a-\cF(U^a)=R_{ap}
\] where  $R_{ap}$ satisfies  the estimate
\[
|R_{ap}|_{E^s}\le
C_{N,s}h^{\frac{2N+1}{4}}\del^{N+1}e^{-\eps_0(N+1)(c_2-c_1)t},\quad\hbox{for}\quad
t\ge 0,
\]
\end{prop} 

In order to prove this proposition,  we use again  the Taylor expansion of $\cF$
\[
\cF(M+V)=\cF(M)+\sum^N_{l=1}\f1{l!}D^l\cF[M](V,\dots,V)+R_{N,\del}(V)
\] where  the first derivative of $\cF$ is $D\cF=J\Lam[M]$ where
\[
J=\left(\begin{matrix}0 & 1\\ -1 & 0\end{matrix}\right),\quad
\Lam[M]=\left(\begin{matrix}- \cP_M+g+Z_MG[\eta_M](Z_M\cdot)+Z_M\p_xv_M
 & v_M\p_x-Z_MG[\eta_M]\\
-\p_x(v_M\cdot)-G[\eta_M](Z_M\cdot) & G[\eta_M]\end{matrix}\right)
\] with the notations
\beno Z_M=
Z[\eta_M,  \varphi_M], \, 
%=\f{(G[\eta_M]\varphi_M+\p_x\eta_M\p_x\varphi_M)}{1+|\p_x\eta_M|^2},\\
v_M=v[\eta_M, \varphi_M], \, 
%= \p_x\varphi_M-Z_M\p_x\eta_M, \\
\cP_M=\cP[\eta_M,\varphi_M].
%f=b\p_x\left(\f{\p_xf}{(1+|\p_x\eta_M|^2)^\f12}
%-\f{(\p_x\eta_M\p_xf)\p_x\eta_M} {(1+|\p_x\eta_M|^2)^\f32}\right)=
%b\p_x\left(\f{\p_xf}{(1+|\p_x\eta_M|^2)^\f32}\right).
\eeno 
The notations $Z$ and $v$ are introduced in
Proposition~\ref{propDN1} while $\cP$ is defined by
$$
\cP[\eta,\varphi]u=b\partial_x \Big((1+(\partial_x\eta)^2)^{-\frac{3}{2}}\partial_xu\Big).
$$
We shall also introduce the  notations \[Z_i=Z[Q_i],\quad
 v_i=v[Q_i],\quad
\cP_i=\cP[Q_i],\quad  i=1, \, 2.\] Plugging
$V(t,x)=\sum^N_{l=1}\del^lV_l(t,x)$ into the system
 leads to  linear problems  with source terms for $V_k$. The system for $V_1$
is
\beq
\label{systV1}
\p_tV_1-J\Lam[M]V_1=-\tilde R_M,
\eeq  the system for $V_2$ is
\[
\p_tV_2-J\Lam[M]V_2=\f12D^2\cF[M](V_1,V_1)
\] and the  general equation for $V_l$ is
\beq
\label{systVl}
\p_tV_l-J\Lam[M]V_l=\sum^l_{p=2}\sum_{\stackrel{1\le
l_1,\dots,l_p\leq l-1}{l_1+\dots+l_p=l}}\f1{p!}D^p\cF[M](V_{l_1},\dots,V_{l_p}).
\eeq

Before solving these systems, 
let us  fix and recall  some notations to be used in the remaining part of the paper.
%By $\Lam$ we denote the  Fourier multiplier  $\Lam=(1+|D|^2)^\f12$.
%We define
%\[ |U|^2_{X^k}=\sum_{0\leq \al\le k}
%\Big(|\p^\be_x U_1|^2_{H^1(\R)}+|\p^\be_x U_2|^2_{\dot{H}^{\f12}_{*}(\R)}\Big)\quad\hbox{with}\quad 
%|\vphi|^2_{\dot{H}^{\f12}_{*}(\R)}=\left| \mathfrak P \vphi\right|_{L^2(\mathbb{R})}^2
%\] 
For   $U(t,x)=(U_1,U_2)^t$, we  define
\beno|U(t)|^2_{X^k}&=&\sum_{0\leq \al+\be\leq k}
\Big(|\p^\al_t\p^\be_x
U_1(t, \cdot)|^2_{H^1(\R)}+|\p^\al_t\p^\be_x U_2(t, \cdot)|^2_{\dot{H}^{\f12}_{*}(\R)}
\Big)
\\
|U(t)|_{W^k}&=& \sup_{\al+ \beta  \le k}|\p^\al_{t} \partial_{x}^\beta U(t,\cdot)|_{L^\infty}\,,
\eeno 
where 
\[ |\vphi|^2_{\dot{H}^{\f12}_{*}(\R)}=\left| \mathfrak P \vphi\right|_{L^2(\mathbb{R})}^2.
\] 

%and \[
%|U|_{Y^k_{[0,\infty)}}=\sup_{t\in [0,\infty)}|U(t)|_{Y^k}\quad\hbox{and}\quad
%|U|_{W^k_{[0,\infty)}}=\sup_{t\in[0,\infty)}|U(t)|_{W^k}.\]

Note that $X^0$ is the natural energy space for the water-waves system.

%Finally, let us recall that we  also use  for $U=(U_{1}, U_{2})$ the notation
%\[
%\|U(t)\|^2_{X^}=\sum_{\al+\be\le k}|\p^\al_t\p^\be_xU(t,\cdot)|^2_{L^2(\R)}.
%\]
%%%%%%%%%%%%%%%%%%%%%%%%%%%%%%%%%%%%%%%%%%%%%%%%%%%%%%%%%
%%%%%%%%%%%%%%%%%%%%%%%%%%%%%%%%%%%%%%%%%%%%%%%%%%%%%%%%%
\subsection{ The homogeneous linear system}
The main ingredient  in the proof of Proposition \ref{approximate solution M+V}
will be a rather precise  estimate of  the growth rate of the fundamental solution of 
the linear homogeneous
equation
 \beq\label{homogeneous linear equation} \p_tV-J\Lam[M]V=0
\eeq
which corresponds to the linearization of the water-waves system about  the multi-solitary wave $M$.
As before, with 
\[
R=\left(\begin{matrix}1 & 0\\
-Z_M & 1\end{matrix}\right)
\] 
we can perform the  change of unknowns  $U=RV$ to get a simpler
linear system which is equivalent to (\ref{homogeneous linear
equation}) \beq\label{new homogeneous linear equation}
\p_tU-JL[M]U=0 \eeq where
\[
L[M]=\left(\begin{matrix} -\cP_M+g+a_M & v_M\p_x\\
-\p_x(v_M\cdot) & G[\eta_M]
\end{matrix}\right)
\] is a self-adjoint operator with the notation $a_M=a[M]=v_M\p_xZ_M+\p_tZ_M$
 and we will take $L_M=L[M], G_M=G[\eta_M]$ for convenience. We shall also use the notations 
 $$a_i=a[Q_i],\quad  L_i=L[Q_i], \quad i=1, \, 2$$
 and so on. Note that this is a short hand for:
 $$ L_{1}u(t,x)= L[Q_{c_{1}}(x-c_{1}t)] u(t,x), \quad  L_{2}u(t,x)= L[Q_{c_{2}}(x - h -c_{2}t)] u(t,x).$$
 
 The main result  of this section will be:
 \begin{theoreme}
 \label{theo homogeneous}
 For $\varepsilon \in (0, \varepsilon^*]$, 
there exists $h_{0}$  and $C>0$  such that for every $h \geq h_{0}$, 
the solution of \eqref{new homogeneous linear equation} with initial datum $U(\tau)$ at $t= \tau$
 satisfies the estimate
 \begin{multline*} |U(t)|_{X^k}+\sum_{
\al\le k}|\p^\al_tU_2|_{L^2} \le
h^\f14 C_k(|U(\tau)|_{X^k}+\sum_{\al\le k}|\p^\al_tU_2(\tau)|_{L^2})\big(1+\eps_0(c_2-c_1)(t- \tau)\big)^k
e^{\eps_0(c_2-c_1)
(t- \tau)/2}, \\ \forall t\geq \tau \geq 0.
\end{multline*}
 \end{theoreme}
 
\begin{rem}Note that in the above estimate, the right hand-side can be expressed in terms of usual Sobolev regularity of the initial data by using 
 the system \eqref{new homogeneous linear equation} to express the time derivatives of the solution at  the initial time $t=\tau$.
 Let us denote by $S_{M}(t,\tau)$ the fundamental solution of the (non-autonomous) system \eqref{new homogeneous linear equation}. 
  The estimate of Theorem \ref{theo homogeneous} can thus  be rewritten under the form:  for every $k \leq 0$, there exists $C_{k}(h,\epsilon_{0})>0$ such that
  \beq
  \label{semigroup1} |S_{M}(t,\tau) U|_{E^k} \leq C_{k} h^{1 \over 4} ( 1+  \epsilon_{0}(t-\tau)^k) |U|_{H^{s(k)}}  e^{\eps_0(c_2-c_1)
(t- \tau)/2}, \\ \forall t\geq \tau \geq 0.
\eeq
 We do not need for our argument a sharp estimate of the number $s(k)$. A straightforward possibility, is to take  $s(k)= 2k+1$
  since  a time derivative of the solution always costs at most two space derivatives of the initial data.
 The meaning of this result is that  when each solitary wave is stable (this corresponds to $\varepsilon$ small), then by choosing $h$
  sufficiently large, we can get an arbitrary slow exponential growth rate for the fundamental solution of \eqref{new homogeneous linear equation},
   the special form of this growth rate that  we have chosen is just one that it is sufficient for the proof of Proposition \ref{approximate solution M+V}.
   Note that the shape that we have chosen is linked in particular to the decay rate of the remainder $R_{M}$ in Proposition \ref{M(t,x) system} 
\end{rem}

We shall split the proof of the above estimate into many steps. For notational convenience, we shall give the proof only in the case $\tau = 0$
 which gives the worse constraint of $h_{0}$. The general case can be deduced  from this one by replacing
  $t$ by $t- \tau$, $x$ by $ x- c_{1} \tau$  and thus in the multi-solitary wave $M$,   $h$ by $ \tilde h = h+ (c_{2} - c_{1}) \tau \geq h$.
  
  During the proof $C$ is a positive number which change from line to line but which is independent of $h$ for $h \geq 1$ and $t$ for $t \geq 0$.

We shall first define a 
decomposition of unity in order to localize our energy estimates in the vicinity of each solitary wave.
We take $\chi^0 \in \mathcal{C}^\infty(\mathbb{R})$ such that 
%\[
%\sum^3_{i=1}\chi^2_i=1,\] where \beno
%\chi_1(t,x)&=&\chi^0_1\left(\f{x-c_1t}{h/2+(c_2-c_1)t/2}\right),\\
%\chi_1(t,x)&=&\chi^0_2\left(\f{x-h-c_2t}{h/2+(c_2-c_1)t/2}\right),\\
%\chi_3(t,x)&=&\chi^0_3\left(\f{x-h/2-(c_1+c_2)t/2}{h/2+(c_2-c_1)t/2}\right).
%\eeno and $\chi^0_i$ ($i=1,2,3$) are $C^\infty(\R)$ functions
%satisfying
\[
\chi^0(x)=\left\{\begin{array}{ll}1,\quad x\le
0,\\
0,\quad x\ge
1\end{array}\right.
%\chi^0_2(x)=\left\{\begin{array}{ll}0,\quad x\le
%-1,\\
%1,\quad x\ge
%0\end{array}\right.\chi^0_3(x)=\left\{\begin{array}{ll}0,\quad
%|x|\ge
%1,\\
%\neq 0,\quad |x|< 1.\end{array}\right.
\] 
and we  define 
$$ \tilde \chi_{1}(t,x)= \chi^0\Big( \f{ x - \f h 4 - c_{m} t }{\f h4} \Big), \, c_{m}= \f{c_{1}+ c_{2}}{2}, \quad \tilde\chi_{2}(t,x) = 1- \tilde \chi_{1}(t,x).$$
Finally, we take
\beq\label{unity decomposition}
  \chi_{1}(t,x)= \f { \tilde \chi_{1} }{(\tilde \chi_{1}^2 + \tilde \chi_{2}^2)^\f12},
 \quad  \chi_{2}(t,x)=\f  { \tilde \chi_{2} }{ (\tilde \chi_{1}^2 + \tilde \chi_{2}^2)^\f12}.\eeq
Note that  these functions are  smooth and bounded and  defined such that $\chi_{1}^2 +\chi_{2}^2= 1.$

The main  properties of these functions that we shall use are:

\begin{lem}
\label{lemchi}
%\label{chiproplem}
The above $\chi_{i}$ satisfy
\beq
\label{chiprop1}
\forall \beta, \, | \beta | \geq 1, \,  |\partial_{t,x}^\beta \chi_{i} (t,x) | \lesssim  \f 1 {h^{|\beta|} }, \quad i=1, \, 2.
\eeq
Moreover,  for any $\eps >0$, we have
\beq
\label{chiprop2}
| e^{- \eps |x - c_{1} t|} \chi_{2} | \leq  { C_{\eps} \over  h}, \quad  
| e^{- \eps |x - h - c_{2} t|} \chi_{1} | \leq  { C_{\eps} \over  h}, \quad \forall t \geq 0, 
\, x \in \mathbb{R}, h \geq 1
\eeq
for some $C_{\eps}>0$.
\end{lem}
%%%%%%%%
\begin{proof}
The  estimate \eqref{chiprop1} is clear. 
For the first one in (\ref{chiprop2}), we observe that $\chi_{2}$ is supported in $x \geq c_{m} t + h/4$. 
Since in this region $x-c_{1}t \geq (c_{m}- c_{1})t + h/4 \geq 0$, we immediately
 get
 $$ | e^{- \eps |x - c_{1} t|} \chi_{2} | \lesssim  e^{- \eps(c_{m}- c_{1})t} e^{- \eps {h\over 4}} 
\lesssim \f1h.$$ 
The second estimate follows by observing that  $\chi_{1}$ is supported in 
$x \leq  h/2 + c_{m}t$. 
This completes the proof of Lemma~\ref{lemchi}.
\end{proof}
%%%%%%
\subsection{Lower-order energy estimate}
The first step in the proof of Theorem \ref{theo homogeneous} will be to prove the estimate for $k=0$.
We shall thus prove
\begin{prop}
\label{prop low order}
Under the assumptions of Theorem \ref{theo homogeneous}, we have the estimate:
$$ |U(t) |_{X^0}^2 + |U_{2}(t)|_{L^2}^2 \leq C
 \Big( h^\f12  |U(\tau)|_{X^0}^2 +  |U_{2}(\tau)|^2_{L^2} \Big) 
e^{ \eps_{0} (c_{2}- c_{1})( t- \tau)/4},\qquad \forall  t\ge \tau \geq 0.$$
\end{prop}
\noindent{\bf Proof}.
Again, we shall give the proof  only for $\tau=0$.

 Let us consider the energy functional
\[
E_1(U(t))=(L_MU,\,U)-c_1(A\chi_1U,\,\chi_1
U)-c_2(A\chi_2U,\,\chi_2U)
\] 
with $A$ the symmetric operator
\[
A=\left(\begin{matrix}0 & \p_x\\ -\p_x & 0\end{matrix}\right)=\p_xJ.
\]
We shall prove the $E_{1}$ is an almost conserved quantity with some positivity property thanks to Proposition \ref{Propstab1dbis} applied
 to each solitary wave.

 We first write
\begin{align}
\nonumber
 \f12 \f d{dt}E_1(U(t))&=(\p_tU,\,L_M
U)+\f12([\p_t,\,L_M]U,\,U)-c_1(A\chi_1\p_tU,\,\chi_1U)
\\
\nonumber&-c_2(A\chi_2\p_tU,\chi_2U)-c_1(A(\p_t\chi_1)U,\,\chi_1U) -c_2
(A(\p_t\chi_2)U,\,\chi_2U)\\
\label{eqE1}&:=I_1+I_2+\cdots+I_6.\end{align} We will estimate these terms one by
one.
Towards this, let us set
\[U^1=\chi_1U,\quad\hbox{and}\quad U^2=\chi_2U\] 
with $\sum_{i=1,2}\chi^2_i=1$ from (\ref{unity decomposition}). 
We shall use very often 
  the following norm equivalence properties. 
\begin{lem} \label{relation between U
and U^i} There  exists  $C>0$ such that
for every $h \geq 1$, we have
\[
|U|^2_{X^0}\le C(\sum_{i=1,2}|U^i|^2_{X^0}+\f
1h|\ka(D)U_2|^2_{L^2})
\] and also
\[
\sum_{i=1,2}|U^i|^2_{X^0}\le C(|U|^2_{X^0}+\f1h |\ka(D)U_2|^2_{L^2})
\] 
with 
 $\hat\ka(\xi)$ is a smooth cut-off function with $\hat \ka(\xi)=1$ around $\xi=0$.
% and $h$ is sufficiently large.
 \end{lem}
\begin{proof}
Let us  prove the first estimate.
% and we omit the second 
%one. Since one recalls that
%\[|U|^2_{X^0}=|U_1|^2_{H^1}+ |\mathfrak P U_2|^2_{L^2},\quad\hbox{and}
%\quad \sum^3_{i=1}\chi^2_i=1 \quad\hbox{with}
%\quad\mathfrak P=\f{|D|}{(1+|D|^2)^\f 14},\]
From Lemma \ref{lemchi}, we first easily get that 
\[
|U_1|^2_{H^1}\le C(\sum_{i=1,2}|\chi_iU_1|^2_{X^0}+\f 1h|U_1|^2_{L^2}),
\] 
For the estimate on  $U_2$, we use the basic commutator estimate 
\begin{equation}\label{basic}
|[\mathfrak P,\,f]g|_{L^2}\lesssim |\p_xf|_{L^\infty}|g|_{L^2}.
\end{equation}
Indeed, one can write 
$
[\mathfrak P,\,f]g=(1-\partial_x^2)^{-\frac{1}{4}}(\partial_x f g)+[(1-\partial_x^2)^{-\frac{1}{4}},f]\partial_x g.
$
The $L^2$ norm of the first term is clearly bounded by the right hand-side of \eqref{basic} while the $L^2$ norm of the second can be bounded by
$C |\p_xf|_{L^\infty}|g|_{L^2}$ by invoking \cite[Proposition~3.6.B, estimate (3.6.35)]{Taylor_first}.
This proves \eqref{basic}.
Estimate \eqref{basic} yields
\beno
|\mathfrak P U_2|^2_{L^2}&\le & \sum_{i=1,2}|\mathfrak P\chi^2_iU_2|^2_{L^2}
\le  \sum_{i=1,2}\left(|[\mathfrak P,\,\chi_i]\chi_iU_2|^2_{L^2}+
|\chi_i\mathfrak P(\chi_i U_2)|^2_{L^2}\right)\\
&\le & C\sum_{i=1,2}(\f 1h|\chi_iU_2|^2_{L^2}
+|\mathfrak P(\chi_i U_2)|^2_{L^2}).
\eeno 
The proof of the other estimate is similar.  This ends  the proof of Lemma~\ref{relation between U and U^i} .
\end{proof}

Now we can go back  to the study of \eqref{eqE1}. First of all, from the linear system 
(\ref{new homogeneous linear equation}), we have
\[
I_1=(JL_MU,\,L_MU)=0.
\] Next, from the definition of  $L_M$ one can compute that
\[
[\p_t,\,L_M]=\left(\begin{matrix}-[\p_t,\,\cP_M]+\p_ta_M &
(\p_tv_M)\p_x\\
-\p_x((\p_tv_M)\cdot) & [\p_t,\,G_M]\end{matrix}\right)
\] and so by combining this with the decomposition of unity leads to
\[ 2I_2=([\p_t,\,L_M]\sum_{i=1,2}\chi^2_iU,\,U)\\
=\sum_{i=1,2}([\p_t,\,L_i]\chi_iU,\,\chi_iU)+2I_{2R}
\]where \beno
2I_{2R}&=&\sum_{i=1,2}\left\{([\p_t,\,L_M-L_i]\chi_iU,
\chi_i U)+([[\p_t,\,L_M],\,\chi_i]\chi_iU,\,U)\right\}\\
&=&\sum_{i=1,2}\{((-[\p_t,\,\cP_M-\cP_i]+\p_t(a_M-a_i))\chi_iU_1,\,
\chi_iU_1)+((\p_t(v_M-v_i))\p_x(\chi_iU_2),\,\chi_iU_1)\\
&& \quad +([\p_t,G_M-G_i]\chi_iU_2,\,\chi_iU_2)
-([[\p_t,\,\cP_M],\,\chi_i]\chi_iU_1,\,U_1)
+([[\p_t,\,G_M],\chi_i]\chi_iU_2,\,U_2)\}.
\eeno 
For the 
remainder term  $I_{2R}$, we have by using Lemma \ref{lemchi} the estimate
\[
|I_{2R}|\le \f 1h C(|U(t)|^2_{X^0}+|U_{2}(t)|^2_{L^2}).
\] Indeed,  let us study for example the first term in more details. We have 
\[\cP_M=b\p_x\left(\f{\p_x\cdot}{(1+|\p_x\eta_M|^2)^\f32}\right),\] 
and thus 
\[[\p_t,\,\cP_M-\cP_i]=b\p_x\left(\int^1_0[\p_t, \,\f{\p_x\eta_{is}
\p_x(\eta_M-\eta_i)}{(1+|\p_x\eta_{is}|
^2)^{\f32}}ds]\p_x\cdot\right)\] with the notation 
$\p_x\eta_{is}=\p_x\eta_i+s\p_x(\eta_M-\eta_i)$. Consequently, by using again  Lemma \ref{lemchi}, we get 
\[|[\p_t,\,\cP_M-\cP_i]\chi_iU_1,\,\chi_iU_1)|\le   \f1h C|U_1|^2_{H^1}.
\] One can estimate  the other  terms in the definition of $I_{2R}$
 by using the same arguments, 
  in particular, for the terms involving the Dirichlet-Neumann operator, we also have
  \begin{lem}\label{commutator lemma}
There exists a constant $C$ such that the following commutator estimates hold
\begin{eqnarray} 
\label{estlemcom1}&& ([\p_t,\,G_M-G_i]
\chi_iU_2,\,\chi_iU_2)\le \f1h C(|\mathfrak P U_2|^2_{L^2}+|U_2|^2_{L^2})
\\ 
\label{estlemcom2}&& ([[\p_t,\,G_M],\,\chi_i]\chi_iU_2,\,U_2)\le \f1h C(|\mathfrak P U_2|^2_{L^2}
+|U_2|^2_{L^2})\end{eqnarray}
 where   $\chi_i\,(i=1,2)$ are defined  in (\ref{unity decomposition}). 
\end{lem} 
This lemma will be proven in the appendix.
\newline

For  the $I_3$ term in \eqref{eqE1}, we get by using the system  (\ref{new homogeneous linear equation}) that
\beno I_3&=& -c_1(A\chi_1JL_MU,\,\chi_1U)=c_1(\p_x(\chi_1L_MU),\,\chi_1U)\\
&=& -c_1(L_1\chi_1U,\,\p_x(\chi_1U))-c_1((L_M-L_1)\chi_1U,\,\p_x(\chi_1U))
-c_1([\chi_1,\,L_M]U,\,\p_x(\chi_1U))\\
&=&\f12c_1([\p_x,\,L_1]\chi_1U,\,\chi_1U)+I_{3R}
 \eeno
with the remainder $I_{3R}$  defined by  \beno I_{3R}&=&
\f12c_1([\p_x,\,L_M-L_1]\chi_1U,\,\chi_1U)-c_1([\chi_1,\,L_M]U,\,\p_x(\chi_1U))\\
&=&-\f12 c_1([\p_x,\,\cP_M-\cP_1-a_M+a_1]\chi_1U_1,\,\chi_1U_1)
+c_1((\p_xv_M-\p_xv_1)\chi_1U_1,\,\p_x(\chi_1U_2))\\
&&\,+\f12c_1([\p_x,\,G_M-G_1]\chi_1U_2,
\,\chi_1U_2)-c_1([\chi_1,\,L_M]U,\,\p_x(\chi_1U)).
\eeno 
Note that the structure of   $I_{3R}$ is very  similar to the one of the  $I_{2R}$ term above  (basically $\partial_{t}$ is replaced by $\partial_{x}$
 in the commutators) 
and hence by using the same arguments as above, we get 
\[
|I_{3R}|\le \f 1h C(|U(t)|^2_{X^0}
+|U_2(t)|^2_{L^2}).
\]  

In a symmetric way, we also have  for the   $I_4$ term in \eqref{eqE1}  that 
\[ I_{4R}=\f12c_2([\p_x,\,L_2]\chi_2U,\,\chi_2U)+I_{4R} \]
with
% \beno
%I_{4R}&=&-\f12
%c_2([\p_x,\,\cP_M-\cP_2-a_M+a_2]\chi_2U_1,\,\chi_2U_1)
%+c_2((\p_xv_M-\p_xv_2)\chi_2U_1,\,\p_x(\chi_2U_2))\\
%&&\,+\f12c_2([\p_x,\,G_M-G_2]\chi_2U_2, \,\chi_2U_2)
%-c_2([\chi_2,\,L_M]U,\,\p_x(\chi_2U))
%\eeno 
  \[ |I_{4R}|\le \f 1h C(|U(t)|^2_{X^0}+|U_2(t)|^2_{L^2}).\]
Since the solitary waves have the dependence   \[ Q_1=Q_1(x-c_1t),\quad
Q_2=Q_2(x-h-c_2t),\]  we have  that  
$$[\p_t,\,L_i]=-c_i[\p_x,\,L_i].$$
%that is
%\[\p_tv_i=-c_i\p_xv_i,\quad
%[\p_t,G_i]=-c_i[\p_x,G_i],\quad [\p_t,
%P_i-a_i]=-c_i[\p_x, P_i-a_i]
%\] with $i=1,2$. 
This yields the crucial cancellation
 \[ I_2+I_3+I_4=I_{2R}+I_{3R}+I_{4R}\]
 and hence, we obtain the estimate
\[ |I_2+I_3+I_4|\le \f 1h C(|U(t)|^2_{X^0}
+|U_2(t)|^2_{L^2}).\]
By using integration by parts and \eqref{chiprop1}, we also easily get that
$$ |I_{5}| + | I_{6}| \leq {C \over h} \big( |U_{1}|_{H^1} + |U_{2}|_{L^2} \big).$$
Summing up the estimates, we get from \eqref{eqE1} that
  \beq\label{right hand
estimate of E_1(U)} \f12\f d{dt}E_1(U(t))\le \f 1h
C(|U(t)|^2_{X^0}+|U_2(t)|^2_{L^2}). \eeq 
The next step will be to get a minoration of 
 $E_1(U(t))$. By using the decomposition of unity (\ref{unity decomposition})
again, one has \begin{eqnarray} 
\nonumber E_1(U(t))&=&
(L_MU,\,U)-c_1(A\chi_1U,\,\chi_1U)-c_2(A\chi_2U,\,\chi_2U)\\
\nonumber &=& \sum_{i=1,2}(L_M\chi_iU,\,\chi_iU)-
c_1(A\chi_1U,\,\chi_1U)-c_2(A\chi_2U,\,\chi_2U)
+\sum_{i=1,2}([L_M,\,\chi_i]\chi_iU,\,U)\\
\label{E1U(t)}&:=& II_1+II_2+II_3\end{eqnarray}
 where \beno
 II_1&=&(L_M\chi_1U,\,\chi_1U)-c_1(A\chi_1U,\,\chi_1U),\\
II_2&=&(L_M\chi_2U,\,\chi_2U)-c_2(A\chi_2U,\,\chi_2U)\quad\hbox{and}
\quad
II_3=\sum_{i=1,2}([L_M,\,\chi_i]\chi_iU,\,U). \eeno 
We shall first handle $II_{1}$. We note that 
  \begin{eqnarray}
  \nonumber  II_1&=&
(L_1\chi_1U,\,\chi_1U)-c_1(A\chi_1U,\,\chi_1U)+((L_M-L_1)\chi_1U,\,\chi_1U)\\
\label{deuxI} &=& (\td L_1\chi_1U,\,\chi_1U)-((L_M-L_1)\chi_1U,\,\chi_1U) \end{eqnarray}
where
\[
\td L_1=\left(\begin{matrix}-\cP_1+g+v_1\p_xZ_1+\p_tZ_1 &
(v_1-c_1)\p_x\\
-\p_x((v_1-c_1)\cdot) & G[\eta_1]\end{matrix}\right).
\] Noticing that $Q_1=Q_1(x-c_1t)$ and so $\p_tZ_1=-c_1\p_xZ_1$, we
can rewrite $\td L_1$ as
\[
\td L_1=\left(\begin{matrix}-\cP_1+g+(v_1-c_1)\p_xZ_1 &
(v_1-c_1)\p_x\\
-\p_x((v_1-c_1)\cdot) & G[\eta_1]\end{matrix}\right).
\] 
Note that   the operators $\td L_1$ is the same operator as  $L_{c_{1}}$ studied 
 in section \ref{sectionstab} except  that it coefficients  depends on
$Q_1=Q_1(x-c_1t)$, we have
$$ \mathcal{T}_{c_{1} t} \tilde L_{1} = L_{c_{1}} \mathcal{T}_{c_{1}t}$$
 where $\mathcal{T}_{x_{0}}$ is the translation operator
 $$ (\mathcal{T}_{x_{0}} U)(x) = U(x+x_{0}).$$   
%%%%%%%%%%%%%%%%%%%%%%%%%%%%%% 
%%%%%%%%%%%%%%%%%%%%%%%%%%%%%%%%%%%%%%
Since $\mathcal{T}_{ct}$ is an isometry on $L^2$ and $X^0$, 
Proposition~\ref{Propstab1d}  applies to $\td L_1$.
 Let us use  
the notation that $Q'_1(x)=\p_x Q_{c_1}(x) $ and define
 $$\bar
U^1(t,y)=U^1(t,y+c_1t)= (\mathcal{T}_{c_{1}t} U^1)(t,y)= \chi_1(t,y+c_1t)U(t,y+c_1t).$$
 Thanks to Proposition \ref{Propstab1dbis}, we can use the decomposition
\begin{equation}
\label{U1dec}
\bar U^1(t,y)=\al_1 (t)JR_{1}Q'_1(y)  + \be_1(t) R_{1}Q'_1(y)+W^1(t,y)
\end{equation}
such that $W^1$ satisfies
\[
\big(W^1,\,(\eta'_1, 0)^t\big)=0,\quad (W^1,\,JR_{1}Q'_1)=0.
\] 
Note that we use again the short-hand $R_{i}= R_{c_{i}}$.
Since by definitions, we have 
$$ \big( \tilde L_{1} U^1, U^1 \big) = \big( \mathcal{T}_{-c_{1}t} L_{c_{1}} \mathcal{T}_{c_{1}t}  U^1, U^1\big)=
 \big(  L_{c_{1}} \mathcal{T}_{c_{1}t}  U^1,  \mathcal{T}_{c_{1}t} U^1\big) =  \big(  L_{c_{1}} \bar U^1,  \bar U^1 \big), $$
 we get  from Proposition \ref{Propstab1dbis},  that 
\beno (\td L_1U^1,\,U^1)  \geq c_{0} |W^1|_{X^0}^2 - C |\al_1|^2.\eeno
 Moreover,  as in the estimate for 
$I_{2R}$ above, we also get from Lemma \ref{lemchi}  that
\[
|((L_M-L_1)\chi_1U,\,\chi_1U)|\le \f 1h C(|U|^2_{X^0}+|U|^2_{L^2}).
\] 
Consequently, we get from the two previous estimates and \eqref{deuxI}
that 
\[
II_1\ge c_0|W^1|_{X^0}^2-C|\al_1|^2- \f 1h
C\big(|U|^2_{X^0} + |U|_{L^2}^2\big).
\] 
In a symmetric way, we can set 
$$\bar U^2(t,y)= (\mathcal{T}_{c_{2} t + h} U^2)(t,y)= U^2(t,y+h+c_2t)=\chi_2(t,y+h+c_2t)U(t,y+h+c_2t)$$
and use the decomposition
\begin{equation}
\label{U2dec}
\bar U^2(t,y)=\al_2 (t)JR_{2}Q'_2(y)+\be_2(t)R_{2}Q'_2(y)+W^2(t,y)
\end{equation} with
\[
(W^2,\,(\eta'_2, 0)^t)=0,\quad (W^2,\,JR_{2}Q'_2)=0
\]  
to obtain that 
\[
II_2\ge  c_0|W^2|_{X^0}^2-C|\al_2|^2- \f 1h
C \big(|U|^2_{X^0} + |U|_{L^2}^2\big).
\]
  Moreover, we also have from  Lemma \ref{lemchi} the estimate 
\[
|II_3| \le  \f 1h C(|U|^2_{X^0}+|U_2|^2_{L ^2})
\] 
by using again that  the commutator always involves at least one derivative of   $\chi_i$. 

In view of the decomposition \eqref{E1U(t)}, we have thus obtained that 
 \beq
 \label{minE1} E_1(U(t))\ge  c_0(|W^1|^2_{X^0}
+|W^2|^2_{X^0}) -C(|\al_1|^2+|\al_2|^2)-\f 1h
C(|U|^2_{X^0}+|U_2|^2_{L ^2}).
\eeq 
In order to conclude, we still need to  estimate $|U_2|_{L^2}$,
$|\al_i|$ and $|\be_i|$($i=1,2$). 

 For the $L^2$ norm, let us choose $\ka(D)$ where $\ka\in C^\infty_0(\R)$ and 
$\ka(\xi)=1$ around $\xi=0$. From the linear system
(\ref{new homogeneous linear equation}), we have  \[
\p_t U_2=(P_M-a_M)U_1-gU_1-v_M\p_x U_2,\]  
and thus we obtain
\beno
\f12 \f d{dt}|\ka(D) U_2|^2_{L^2}&=&(\ka(D)\p_tU_2, \,\ka(D)U_2)\\
&=&(\ka(D)(P_M-a_M)U_1,\,\ka(D)U_2)-g(\ka(D)U_1,\,\ka(D)U_2)\\
&&\,-(\ka(D)(v_M\p_xU_2),\, \ka(D)U_2).
\eeno 
By using that $\kappa$ is compactly supported, this  yields
$$ \f12 \f d{dt}|\ka(D) U_2|^2_{L^2} \leq C \big( |U_{1}|_{H^1} + | \mathfrak{P} U_{2}|_{L^2} \big) \big( |\mathfrak{P} U_{2}|_{L^2}
 + |\kappa(D)U_{2}|_{L^2}\big)$$
 and hence, we obtain from the Young inequality that
\beq \label{estimate for |U_2|_L^2}
\f12 \f d{dt}|\ka(D) U_2|^2_{L^2}\le C(\eps^{-1}|U|^2_{X^0}+\eps
|\ka(D)U_2|^2_{L^2})
\eeq where $\eps$ is a small constant to be fixed later. 

To estimate $\alpha_{i}$, we use the 
decompositions  \eqref{U1dec}, \eqref{U2dec} of $U^i$ ($i=1,2$).  We have
\beno
\f d{dt}\alpha_1&=& { 1 \over |J R Q_{1}' |_{L^2}^2}\p_t
(\bar \chi_1(t)\bar U(t),\,JR_{1}Q'_1)\\
&=& { 1 \over |J R Q_{1}' |_{L^2}^2} \Big( ((\p_t\bar\chi_1(t))\bar U(t),\,JR_{1}Q'_1)+
(\bar\chi_1(t)(J\overline{L_MU}+c_1\overline{\p_xU})(t),\,JR_{1}Q'_1) \Big)\\
&=& { 1 \over |J R Q_{1}' |_{L^2}^2}\Big( ((\p_t\bar\chi_1(t))\bar U(t),\,JR_{1}Q'_1)+
(\bar\chi_1(t)J(\overline{L_MU}-\overline{L_1U})(t),\,JR_{1}Q'_1)\\
&&\,+ (\bar\chi_1(t)(J\overline{L_1U}+c_1\overline{\p_xU})(t),\,JR_{1}Q'_1) \Big)
\eeno with the notation that $\bar f(t,x)=f(t,x+c_1t)$.  
We also have
\beno
&&(\bar\chi_1(t)(J\overline{L_1U}+c_1\overline{\p_xU})(t),\,JR_{1}Q'_1)\\
&&=(\bar\chi_1(t)J(\overline{L_1U}-c_1J\overline{\p_xU})(t),\,JR_{1}Q'_1)
=(\bar\chi_1(t)J\td L_1[Q_1(y)]
\bar U(t),\,J R_{1}Q'_1)\\
&&= \big(\big[\bar\chi_1(t),\,J\td L_1[Q_1(y)] \big]\bar U(t),\, JR_{1}Q'_1\big) - (\bar\chi_1(t)\bar U(t),\,
\td L_1[Q_1(y)] J^2 R_{1}Q'_1)\\
&&=\big(\big[\bar\chi_1(t),\,J\td L_1[Q_1(y)] \big]\bar U(t),\, JR_{1}Q'_1\big).
\eeno 
Indeed, to pass from the second to the third line, we have used the crucial cancellation 
 $$\big(\overline{L}_{1} - c_{1} J \partial_{x}\big) J^2 R_{1} Q_{1}'=-  \tilde{L}_{1}  R_{1}Q'_1= - L_{c_{1}} R_{c_{1}} Q_{c_{1}}'=0.$$
  By using again Lemma \ref{lemchi} to estimate the other terms, we thus obtain
\beq
\label{alpha11}  \big| \f d{dt}\al_1 \big|\le \f 1h C(|U|_{X^0}+|U_2|_{L^2}).\eeq
In a symmetric way, we also get  for $|\al_2|$ that 
\beq
\label{alpha21}   \big|\f d{dt}\al_2 \big|\le \f 1h C(|U|_{X^0}+|U_2|_{L^2}).\eeq
We still need the estimates of $|\be_1|$ and $|\be_2|$. 
By using \eqref{trois}, we note that we can write
$$ \be_{1}= \tilde \be_{1}-  \alpha_{1}  \big( JR_{1} Q_{1}', (\eta_{1}', 0)^t \big)$$
where
$$ \tilde \be_{1}= {\big( \bar U^1, (\eta'_{1}, 0)^t \big) \over  |\eta'|_{L^2}^2}.$$  
In particular, thanks to \eqref{alpha11}, we obtain that
\beq
\label{beta10} \big| {d \over dt } \beta_{1}\big| \leq \big| {d \over dt} \tilde \beta_{1} \big| +  \f 1h C(|U|_{X^0}+|U_2|_{L^2}).\eeq
 To estimate $\tilde \beta_{1}$, we directly compute
\beno
\f d{dt} \tilde \be_1(t)&=& {1 \over | \eta_{1}'|^2_{L^2}}\p_t(\bar U^1(t),\,(\eta_{1}', 0)^t) \\
&=&{1 \over | \eta_{1}'|^2_{L^2}} \Big( ((\p_t\bar\chi_1(t))\bar U(t),\, (\eta_1',0)^t)+(\bar\chi_1(t)
(\overline{\p_tU}+c_1\overline{\p_x U})(t),\, (\eta'_1, 0)^t) \Big)\\
&=&{1 \over | \eta_{1}'|^2_{L^2}}\Big(((\p_t\bar\chi_1(t))\bar U(t),\, (\eta'_1, 0)^t)+(\bar\chi_1(t)J
\overline{L_MU}(t),\,(\eta'_1, 0)^t)\\
&&\,  \quad  \quad \quad  + c_{1}(\bar\chi_1(t)\overline{\p_x U}(t),\,(\eta'_1, 0)^t)\Big)\eeno where in particular
\beno
(\bar\chi_1(t)J\overline{L_MU}(t),\,(\eta'_1, 0)^t)&=&(\bar\chi_1(t)J(\overline{L_MU}
-\overline{L_1U})
(t),\,(\eta'_1, 0)^t)+([\bar\chi_1(t),\,J\bar L_1]\bar U,\, (\eta'_1, 0)^t)\\
&&\, +(J\overline{L_1U^1}(t),\,(\eta'_1, 0)^t)\Big)
\eeno and
\[
c_1(\bar\chi_1(t)\overline{\p_x U}(t),\,(\eta'_1, 0)^t)=-c_1((\p_x\bar\chi_1(t)
\bar U(t),\, (\eta'_1, 0)^t)+c_1(\overline{\p_x U^1}(t),\,(\eta'_1, 0)^t).
\] 
Therefore, we get \beno
\big|\f d{dt}\tilde \be_1(t)\big|
& \leq & C \big| \big(  J \tilde{L}_{1} \overline{U}^1, (\eta_{1}', 0) \big)\big|
 +  \f 1h C(|U|_{X^0}+|U_2|_{L^2}).
%(t),\,Q'_1)\\
%&& +\al_1(t)([\bar\chi_1(t),\,J\bar L_1]\bar U,\, Q'_1)
%-\al_1(t)c_1((\p_x\bar\chi_1(t)\bar U(t),\, Q'_1)\\
%&& \,+\al_1(t)(J\td L_1[Q_1(y)]\bar U^1(t),\, Q'_1)\\
%&=&\al_1(t)((\p_t\bar\chi_1(t))\bar U(t),\, Q'_1)
%+\al_1(t)(\bar\chi_1(t)J(\overline{L_MU}
%-\overline{L_1U})
%(t),\,Q'_1)\\
%&&\, +\al_1(t)([\bar\chi_1(t),\,J\bar L_1]\bar U,\, Q'_1)
%-\al_1(t)c_1((\p_x\bar\chi_1(t)\bar U(t),\, Q'_1)\\
%&&\, +\al_1(t)\be_1(t)(J\td L_1[Q_1(y)](JQ'_1),\, Q'_1)+\al_1(t)
%(J\td L_1[Q_1(y)]W^1(t), \,Q'_1)
\eeno 
To conclude, we  use the decomposition \eqref{U1dec}, 
 since $\tilde{L}_{1} R_{1} Q_{1}'= 0$, we have
$$ \big|  J \tilde{L}_{1} \overline{U}^1, (\eta_{1}', 0) \big)\big| \leq  C \big( |\alpha_{1}| +  |W^1|_{X^0}\big)$$
and hence, we get from  \eqref{beta10}  that 
\beq
\label{beta11}
\big|\f d{dt}\be_1(t)\big|
 \leq   C\Big( \f 1h (|U|_{X^0}+|U_2|_{L^2}) + | \al_{1}| + |W^1|_{X^0}\Big).
\eeq
Similarly we have
\beq
\label{beta21}
\big| \f d{dt}\be_2 (t) \big|\le C\left(\f 1h(|U|_{X^0}+|U_2|_{L^2})
+|\al_2|+|W^2|_{X^0}\right).\eeq

From \eqref{alpha11}, \eqref{alpha21}, \eqref{beta11}, \eqref{beta21},  we finally obtain
\begin{align}
\label{alpha1}
& {d \over dt}  | \alpha_{i}|^2  \leq C | \alpha _{i}|  \f 1h(|U|_{X^0}+|U_2|_{L^2}), \quad i=1, \, 2 \\
\label{beta1}
&{d \over dt}  | \beta_{i}|^2 \leq    C | \beta_{i}|\Big( \f 1h (|U|_{X^0}+|U_2|_{L^2}) + | \al_{i}| + |W^i|_{X^0}\Big), \quad i= 1, \, 2.
\end{align}
To combine \eqref{alpha1}, \eqref{beta1},  \eqref{right hand estimate of E_1(U)},  \eqref{estimate for |U_2|_L^2} and  \eqref{minE1} in 
 an efficient way, 
we define a weighted energy
 \begin{equation}
 \label{E1tildedef}\td E_1(U(t))=\f 12 h^\f12 E_1(U(t))+\f12(
|\be_1(t)|^2+|\be_2(t)|^2)+Ch^\f12(|\al_1(t)|^2 +|\al_2(t)|^2)
\end{equation}
From \eqref{right hand estimate of E_1(U)} and \eqref{estimate for |U_2|_L^2}, \eqref{alpha1}, \eqref{beta1}, we obtain 
 \beno
&&\f d{dt} \td E_1(U(t))\\
&&\le h^{-\f14} C\big[h^{-\f14}
(|U|^2_{X^0}+|\ka(D)U_2|^2_{L^2})+h^{-\f34} (|\be_1|+|\be_2|)
(|U|_{X^0}+|U_2|_{L^2})\\
&&\qquad +h^\f14(|\be_1|+|\be_2|)(|\al_1|+|\al_2|+|W^1|_{X^0}+|W^2|_{X^0})
+h^{-\f14}(|\al_1|+|\al_2|)(|U|_{X^0}+|U_2|_{L^2})
\big]\\
&&\le h^{-\f14}C\big[h^\f12(|W^1|^2_{X^0}+|W^2|^2_{X^0})
+(|\be_1|^2+|\be_2|^2)+h^{\f12}(|\al_1|^2+|\al_2|^2)+h^{-\f14}|U_2|^2_{L^2})\big].
\eeno
 To get the last line, we have used the decompositions of  $U^1,\,U^2$, \eqref{U1dec}, \eqref{U2dec} and
 Lemma  \ref{relation between U and U^i}. Now, let us define
\[
F(t)=h^\f12(|W^1|^2_{X^0}+|W^2|^2_{X^0})
+(|\be_1|^2+|\be_2|^2) + h^\f12( |\alpha_{1}|^2 + |\alpha_{2}|^2 ) 
\] and integrate   the estimate for $\td E_1(U(t))$  with respect
 to time from $0$ to $t$,  we obtain
\[
\td E_1(U(t))\le \tilde E_1(U(0))+h^{-\f14} C\int^t_0(F(s)+h^{-\f14}|\ka(D)U_2(s)|^2_{L^2})
ds,\]
by using  that $|U_2|_{L^2}\le C(|\ka(D)U_2|_{L^2}+|U|_{X^0})$.
 On the other hand,  from the definition of $\td E_1(U)$ in  \eqref{E1tildedef} and the estimate \eqref{minE1},  we  also have
\beno
\td E_1(U(t))&\ge& \f12h^{\f12}c_0(|W^1|^2_{X^0}+|W^2|^2_{X^0})
+\f12(|\be_1|^2+|\be_2|^2)\\
&&\quad+h^{\f12}c_0(|\al_1|^2+|\al_2|^2)-h^{-\f12}C(\sum_{i=1,2}|U^i|^2_{X^0}
+|\ka(D)U_2|^2_{L^2})\\
&\ge& \bar c_0 F(t)-h^{-\f12}C|\ka(D)U_2|^2_{L^2}
\eeno
for some $\overline{c}_{0}>0$.  This yields 
\beq\label{estimate for F(t)}
F(t)-h^{-\f12}C|\ka(D)U_2(t)|^2_{L^2}\le h^\f12
C|U(0)|^2_{X^0}+h^{-\f14}
C\int^t_0(F(s)+h^{-\f14}|\ka(D)U_2(s)|^2_{L^2})ds. \eeq
By  taking  the integral of (\ref{estimate for |U_2|_L^2})
with respect to time, we also have  \beq\label{estimate for |U_2|_L^2
again} |\ka(D)U_2(t)|^2_{L^2}(t)\le
|U_2(0)|^2_{L^2}+C\int^t_0(\eps^{-1}F(s)+\eps|\ka(D)U_2(s)|^2_{L^2})ds \eeq where the
constant $\eps$  will  be fixed later. We take the sum of
(\ref{estimate for F(t)}) and (\ref{estimate for |U_2|_L^2 again})
multiplied with some  $\lam>0$ that will be also chosen later  to get that \beno
&&F(t)+(\lam-h^{-\f12}C)|\ka(D)U_2(t)|^2_{L^2}\\
&&\le h^\f12C|U(0)|^2_{X^0}+\lam|U_2(0)|^2_{L^2}+Ch^{-\f14}
\int^t_0(F(s)+h^{-\f14}|\ka(D)U_2(s)|
^2_{L^2})ds\\
&&\quad +\lam C\int^t_0(\eps^{-1}F(s)+\eps|\ka(D)U_2(s)|^2_{L^2})ds.
\eeno 
Let us choose   $\lam$ and $\eps$  in order to satisfy the
following conditions
\[\lam-h^{-\f12}C\ge \lam/2,\quad Ch^{-\f14}\le \ga/4,\quad
h^{-\f14}\le \lam/2,\quad \lam C\eps^{-1}\le \ga/4\quad
\hbox{and}\quad \eps^2\le \lam/2,\]where $\ga=\eps_0(c_2-c_1)$ comes
from Proposition~\ref{M(t,x) system}. One  can choose for example
$\lam=\f{\ga^2}{64C^2}$,  $\eps=\f{\sqrt{\lam}}2$ and $h$ large enough to get our final energy  estimate
\[
F(t)+\f \lam 2|\ka(D)U_2(t)|^2_{L^2}\le h^{\f12}C|U(0)|^2_{X^0}
+\lam|U_2(0)|^2_{L^2}+\f{\ga}4\int^t_0(F(s)+{\lambda \over 2 }|\ka(D)U_2(s)|
^2_{L^2})ds.
\] Applying Gronwall's inequality, we get 
\[
F(t)+\f \lam 2|\ka(D)U_2(t)|^2_{L^2}\le (h^\f12 C|U(0)|^2_{X^0}
+\lam|U_2(0)|^2_{L^2})e^{\ga t/4},\quad\hbox{for any}\quad t\ge 0.
\] 
From the definition of $F(t)$ and Lemma \ref{relation between U and U^i},
 we get that there exist constants $c_0$ and
$\bar c_0$ such that
\beno
F(t)&=& h^\f12(|W^1|^2_{X^0}+|W^2|^2_{X^0})
+(|\be_1|^2+|\be_2|^2) +  h^\f12( |\alpha_{1}|^2 + |\alpha_{2}|^2 )\\
&\ge &c_0\sum_{i=1,2}|U^i(t)|^2_{X^0}\ge\big( 1 -  \f{C}h \big)\bar c_0|U(t)|^2_{X^0}-
\f {c_0}h|\ka(D)U_2(t)|^2_{L^2}.
\eeno 
By choosing $h$ such that  
 $\f{c_0}h <\f\lam 4$, we find
\[
F(t)\ge \f12\bar c_0 |U|^2_{X^0}-\f\lam 4|\ka(D)U_2|^2_{L^2}.
\] Summing up the estimates above, we finally obtain
\[
\f{\bar c_0}2|U|^2_{X^0}+\f\lam 4|\ka(D)U_2|^2_{L^2}\le (h^\f12
C|U(0)|^2_{X^0} +\lam|U_2(0)|^2_{L^2})e^{\eps_0(c_2-c_1)
t/4},\quad\hbox{for any }\quad t\ge 0,
\]
that is to say
\[
|U|^2_{X^0}+|U_2|^2_{L^2}\le C(h^\f12|U(0)|^2_{X^0}+|U_2(0)|^2_{L^2})e^{\eps_0
(c_2-c_1)t/4},\qquad t\ge 0.
\]
This ends the proof of Proposition \ref{prop low order}.

\subsection{Proof of  Theorem \ref{theo homogeneous}}

We shall prove Theorem \ref{theo homogeneous} by induction. The $k=0$
case was already obtained in Proposition \ref{prop low order}. For the sake of clarity,
before the general induction argument, we shall explain the proof for the $k=1$ energy estimate.

\noindent{\bf  Order-1 energy estimates}.
Note that since the coefficients in the linear system (\ref{new homogeneous linear
 equation}) depend on $t$ and $x$, neither $\partial_{t}$ nor $\partial_{x}$ has a nice commutation property with the equation.
  We shall  thus use the operator
\beq\label{definition of D}
D(\p)=\p_t+c_1\chi^2_1\p_x+c_2\chi^2_2\p_x
\eeq
in order to take derivatives of the equation. This will take into account the fact that the solitary waves depend on
 $ x- c_{i} t$.

By applying this operator on both sides of system (\ref{new homogeneous linear equation}), 
we 
obtain   for $D(\p)U$ the system:
\beq\label{eqn for DU}
\p_tD(\p)U-JL_MD(\p)U=[\p_t,\,D(\p)]U-J[L_M,\,D(\p)]U.
\eeq A further computation shows that
\[
[\p_t,\,D(\p)]U=c_1(\p_t\chi^2_1)\p_xU+c_2(\p_t\chi^2_2)\p_xU,\] and
\beq\label{operator S}\begin{split}
J[L_M,\,D(\p)]U=&J[L_M,\,\p_t]U+J\sum_{i=1,2}c_i[L_M,\,\chi^2_i\p_x]U\\
=&J\sum_{i=1,2}\big(\chi^2_i[L_M-L_i,\,\p_t]U
+\chi^2_i[L_i,\,\p_t]U +c_i[L_M,\,\chi^2_i]\p_xU+c_i\chi^2_i[L_M
-L_i,\,\p_x]U\\
&\, +\chi^2_i[L_i,\,c_i\p_x]U\big)\\
=& J\sum_{i=1,2}\big(\chi^2_i[L_M-L_i,\,\p_t]U
+c_i[L_M,\,\chi^2_i]\p_xU +c_i\chi^2_i[L_M
-L_i,\,\p_x]U\big):= JSU.
\end{split}\eeq
 The crucial fact that we have used in the above computation is the cancellation:
\[
J\sum_{i=1,2}\chi^2_i[L_i,\,\p_t]U+J\sum_{i=1,2}\chi^2_i[L_i,\,c_i\p_x]U=
J\sum_{i=1,2}\chi^2_i[L_i,\,\p_t+c_i\p_x]U=0
\] 
 We  can thus  rewrite the system
(\ref{eqn for DU}) for $D(\p)U$ as \beq\label{eqn for DU new}
\p_tD(\p)U=JL_MD(\p)U-JSU+[\p_t,\,D(\p)]U:=JL_MD(\p)U+F_1(U)\eeq where we will
take $F_1(U)=-JSU+[\p_t,\,D(\p)]U$ as the source term. 

  From Proposition \ref{prop low order}, we have for the fundamental solution $S_{M}(t,\tau)$
  of the linear equation \eqref{new homogeneous linear equation} the estimate  
 \[
\|S_{M}(t,\tau)\|_{X^0\cap L^2\rightarrow X^0\cap L^2}\le h^\f14 C e^
{\eps_0(c_2-c_1)(t - \tau)/4},\qquad \forall t \geq \tau \geq 0.\]   
Consequently,  by using  Duhamel's Formula, we can rewrite \eqref{eqn for DU} as 
\[D(\p)U(t)=S_{M}(t,0)( D(\partial)U)(0)+ \int^t_0  S_{M}(t, \tau) F_1(\tau)d\tau,\] so we get 
\beq\label{estimate DU 1}\begin{split}&|D(\p)U(t)|_{X^0}+|D(\p)U_2|_{L^2}\\
%&\le |e^{tJL_M}D(\p)U(0)|_{X^0}+
%|e^{tJL_M}D(\p)U(0)|_{L^2}+\int^t_0\|e^{(t-\tau)JL_M}\|_{X^0\cap L^2}
%(|F_1(\tau)|_{X^0}+|F_1(\tau)|_{L^2})d\tau\\
& \le h^\f14C\big[e^
{\eps_0(c_2-c_1)t/4}(|U(0)|_{X^1}+\sum_{\al=0,1}|\p^\al_tU_2(0)|_{L^2})
+\int^t_0  e^
{\eps_0(c_2-c_1)(t-\tau)/4}(|F_1(\tau)|_{X^0}+|F_1(\tau)|_{L^2})d\tau\big].
\end{split}\eeq 
We still need to  estimate  the source term $F_{1}$. 
We shall first use the elliptic regularity of the leading spatial operators in \eqref{new homogeneous linear equation}
 to prove that time derivatives control higher order space derivatives.
 
 For any norm $\| \cdot\|$ on $x$ dependent vectors, we use the notation
 $$ \| \langle \partial_{t}\rangle^k U \|=  \sum_{0 \leq l \leq k} \| \partial_{t}^l U\|.$$
\begin{lem}\label{space derivative} 
Any smooth solution of  (\ref{new 
homogeneous linear equation}) satisfies the following a priori estimates:
\beq
\label{elliptic1}
 \forall l \geq 0, \, m \geq 0, \, \exists C_{k,l}, \, \quad \big|\partial_{t}^l \big( U_{1}, U_{2}) \big|_{H^{m+{5\over 2}} \times H^{m+2}}
 \leq C_{l,m}   | \langle \partial_{t}\rangle^{l+1} (U_{1}, U_{2}) |_{H^{m+1} \times H^{m+{1 \over 2} } }.
\eeq
%and more generally, for  $\al+\be=k$, we have the estimate: \[
%|(\p^\al_t\p^\be_xU_2,\,\p^\al_t\p^{\be+1}_xU_1)^T|_{X^0}\le C_k(|\p^k_tU|_{X^0}
%+\sum_{\al_1+\be_1\le k-1}|\p^{\al_1}_t\p^{\be_1}_xU|_{X^0}+\sum_{l\le k}|
%\p^l_t U_2|_{L^2})
%\]
% where  $C,\,C_k$ are constants.
\end{lem}
 
% {\color{red} J'ai chang\'e l\'enonc\'e du lemme, mais j'ai r\'ecrit l'in\'egalit\'e qui sert, il me semble que l'\'enonc\'e \'etait faux pour $\alpha=k$
% (dans la version pr\'ec\'edente),
% mais ce cas ne se produisait jamais par la suite. J'ai aussi pas mal \'elagu\'e la preuve}
 \begin{proof} 
This is an elliptic regularity result. Let us start with the case $l=0$. We first  rewrite \eqref{new homogeneous linear equation} under the form:
\beq
\label{systel1}\left\{\begin{array}{ll} & G_MU_2 = \partial_{t} U_{1} + \partial_{x}(v_{M} U_{1}),\\
 & \cP_M U_{1}=  \partial_{t} U_{2} + (a_M+g)U_1 + v_M\p_xU_2.
\end{array} \right.
\eeq
The operator $\cP_{M}$ is an elliptic operator of order two, therefore, by classical elliptic regularity results, we immediately obtain from the
 second line of the above system
$$  |U_{1} |_{H^{m+ {5 \over 2 }}} \leq C_{m} \Big(  |\partial_{t} U_{2}|_{H^{m+{1 \over 2}}} +  |\partial_{x} U_{2} |_{H^{m+ {1 \over 2 }}} + |U_{1}|_{H^m} \Big)$$
and thus by the  interpolation inequality
$$ |U_{1}|_{H^m} \leq \delta   |U_{1} |_{H^{m+ {5 \over 2 }}} + C_{\delta} |U_{1}|_{L^2},$$
 we actually obtain
\beq
\label{estel1} |U_{1} |_{H^{m+ {5 \over 2 }}} \leq C_{m} \Big( |\partial_{t} U_{2}|_{H^{m+{1 \over 2}}} + |U_{2} |_{m+ {3 \over 2 }} + |U_{1}|_{L^2} \Big).
\eeq
In a similar way, since the surface $\eta_{M}$ is smooth, the Dirichlet-Neumann operator $G[ \eta_{M}]$
 is an elliptic operator of order one. Actually,  we have
 \beq
 \label{symbolep}
 G[\eta_{M}]=|D_{x}| + R(t,x,D_{x})
 \eeq 
where $R$ is a pseudo differential operator of order zero. We refer to \cite{L}, \cite{Taylor} for the proof.
Note that we are in dimension one thus the principal symbol is very simple.
Consequently, we also get from the first equation of \eqref{systel1}
\beq
\label{estel2}
|U_{2}|_{H^{m+2}} \leq C_{m}\Big(|\partial_{t} U_{1}|_{H^{m+1}} + |U_{1}|_{H^{m+2}} + |U_{2}|_{L^2}\Big).
\eeq
By using again that
$$|U_{1}|_{H^{m+2}}  \leq \delta  |U_{1} |_{H^{m+ {5 \over 2 }}} + C_{\delta} |U_{1}|_{L^2}, \quad  |U_{2} |_{m+ {3 \over 2 }}  \leq \delta |U_{2}|_{H^{m+2}}
 + C_{\delta} |U_{2}|_{L^2},$$
 we get from \eqref{estel1}, \eqref{estel2} by choosing $\delta$ sufficiently small that
 $$ |U_{1} |_{H^{m+ {5 \over 2 }}} + |U_{2}|_{H^{m+2}} \leq C_{m} \Big( |\partial_{t} U_{1}|_{H^{m+1}} +   |\partial_{t} U_{2}|_{H^{m+{1 \over 2}}} + |U|_{L^2}
  \Big).$$
  We have thus proven \eqref{elliptic1} for $l=0$.
  
  We then proceed by induction on $l$. Assume that the result is proven for $l-1$ time derivatives.
  Let us apply $\partial_{t}^l$ to \eqref{systel1}.
   From the second line, we get
   $$  \mathcal{P}_{M} \partial_{t}^l U_{1}= \partial_{t}^{l+1} U_{2}  + \partial_{t}^l \big( (a_{M} + g ) U_{1} \big) + \partial_{t}^l \big( v_{M}
    \partial_{x} U_{2} \big) - [\partial_{t}^l, \mathcal{P}_{M} ] U_{1} := F^1$$
   and we observe that the right hand-side satisfies the estimate
   $$ |F^1|_{H^{m+{1 \over 2}}} \leq C_{l,m}\Big( |\partial_{t}^{l+1} U_{2}|_{H^{m+{1 \over 2}}} + | \langle \partial_{t}\rangle^l U_{2}|_{H^{m+{3 \over 2}}}
    + | \langle \partial_{t} \rangle^{l-1} U_{1} |_{H^{m+{5\over 2}}} \Big)$$
     and thus we get by elliptic regularity that
   \beq
   \label{estel3} | \partial_{t}^l U_{1}|_{H^{m+{5\over 2}}} \leq C_{l,m} \Big( |\partial_{t}^{l+1} U_{2}|_{H^{m+{1 \over 2}}} + | \langle \partial_{t} \rangle^l U_{2}|_{H^{m+{3 \over 2}}}
    + | \langle \partial_{t} \rangle^{l-1} U_{1} |_{H^{m+{5\over 2}}} \Big).
    \eeq
    In a similar way, for the first line of \eqref{systel1}, we get
    $$G_{M} \partial_{t}^lU_{2} = \partial_{t}^{l+1} U_{1} +  \partial_{t}^l \partial_{x}(v_{M} U_{1} ) - [\partial_{t}^l,  G_{M}] U_{2}  := F^2.$$
     By using Proposition \ref{propDN1} (6) to compute the commutator $[\partial_{t}^l,  G_{M}] U_{2} $, we obtain for the right hand side the estimate
     $$ |F^2|_{H^{m+1} } \leq C_{l,m} \Big( | \partial_{t}^{l+1} U_{1} |_{H^{m+1}} +  | \langle \partial_{t} \rangle^l U_{1} |_{H^{m+2}} + | \langle \partial_{t}
      \rangle^{l-1} U_{2} |_{H^{m+2}}\Big).$$
      Consequently, from the ellipticity of the Dirichlet-Neumann operator, we also get
      \beq
      \label{estel4}
      |\partial_{t}^l U_{2}|_{H^{m+2}} \leq C_{l,m}\Big(|   \partial_{t}^{l+1} U_{1} |_{H^{m+1}} + | \langle \partial_{t} \rangle^l U_{1} |_{H^{m+2}} + | \langle \partial_{t}
      \rangle^{l-1} U_{2} |_{H^{m+2}}\Big) .
      \eeq
      By combining \eqref{estel3} and \eqref{estel4} we get
     \begin{multline*} | \partial_{t}^l (U_{1}, U_{2})|_{H^{m+{5\over 2}} \times H^{m+2}} \leq C_{l,m} \Big(
      |\partial_{t}^{l+1}(U_{1}, U_{2})|_{H^{m+1} \times H^{m+{1 \over 2 } }} \\+ |  \partial_{t}^{l}(U_{1}, U_{2})|_{H^{m+2} \times H^{m+{3 \over 2 }} }
       + |\langle \partial_{t} \rangle^{l-1} (U_{1}, U_{2} )|_{H^{m+{5\over 2}} \times H^{m+2}}\Big).
       \end{multline*}
       To conclude, it suffices to use the interpolation inequality
       $$ |  \partial_{t}^{l}(U_{1}, U_{2})|_{H^{m+2} \times H^{m+{3 \over 2 }} } \leq \delta   |  \partial_{t}^{l}(U_{1}, U_{2})|_{H^{m+{5\over 2 }} \times H^{m+2} }
        + C_{\delta}|  \partial_{t}^{l} U|_{L^2} $$
 and the induction assumption. This ends the proof of Lemma \ref{space derivative}.
 \end{proof}
 %%%%%%%%%%%
 As a consequence of Lemma \ref{space derivative}, we get the following  inequalities that we will use many times:
 \begin{itemize}
 \item  By using the lemma with $l=0$, $m=0$, we have
 $$
  | \partial_{x} U |_{H^{3 \over 2} \times H^1} \leq C \big(  | \langle \partial_{t} \rangle U |_{H^1 \times H^{1 \over 2} } \big)
 $$ 
 and hence, since, we can use again that
 \beq
 \label{interpolbis} | \partial_{x} U |_{H^{1} \times H^{1 \over 2}} \leq  \epsilon | \partial_{x} U|_{H^{3 \over 2} \times H^1}  + C_{\epsilon} |U|_{L^2}\eeq
 we get that
 \beq
 \label{estimate for order-1 space derivative}
 |\partial_{x} U |_{X^0} \leq \epsilon ( |  \partial_{t}  U |_{X^0}+ |\partial_{t} U_{2}|_{L^2 }) + C_{\epsilon}(
  |U|_{X^0} + |U|_{L^2} \big)
 \eeq
 for any $\epsilon >0$.
 \item  We can also use Lemma  \ref{space derivative}, with $l=\alpha$, $m= \beta -1$ for  any $\alpha$ and $\beta$ such that 
  $\alpha + \beta = k$ and $\beta \geq 1$. We obtain
  $$ | \langle \partial_{t}\rangle^\alpha U |_{ H^{\beta + { 3 \over 2 }} \times  H^{\beta  + 1}}  \leq C_{k} | \langle \partial_{t}  \rangle^{\alpha + 1} U |_{H^{\beta }
  \times H^{\beta  -  {1 \over 2}}}$$
  and  we can  iterate the process  to obtain
$$ | \partial_{t}^\alpha \partial_{x}^\beta U|_{H^{3 \over 2} \times H^1} \leq  | \langle \partial_{t}\rangle^\alpha U |_{ H^{\beta + { 3 \over 2 }} \times  H^{\beta  + 1}}
  \leq  C_{k}   | \langle \partial_{t} \rangle^{ k} U  |_{H^{ 1} \times H^{1\over 2 } }.$$  
  Thanks to \eqref{interpolbis}, we thus obtain 
  \beq
  \label{estimate for high-order space derivative}
   | \partial_{t}^\alpha \partial_{x}^\beta U|_{X^0} \leq \epsilon \big( | \partial_{k} U|_{X^0}+ |\partial_{t}^k U_{2}|_{L^2} \big)
   + C_{\epsilon}  \big(  |U|_{X^{k-1}} + | \langle \partial_{t} \rangle^{k-1} U_{2}|_{L^2} \big)\eeq
    for $\alpha + \beta = k$,  $\beta \geq 1.$
 \end{itemize}

\bigskip
We can come back to the estimate for the source term $F_1(U(t))$ in the right hand side of
 \eqref{estimate DU 1}.  By using the expression of $F_1(U(t))$ in 
 (\ref{operator S}), we get 
\beno
|F_1(U)|_{X^0}&\le& |JSU|_{X^0}+|[\p_t,\,D(\p)]U|_{X^0}\\
&\le & |J\sum_{i=1,2}\chi^2_i[L_M-L_i,\p_t]U|_{X^0}+|J\sum_{i=1,2}c_i[L_M,\,
\chi^2_i] \p_xU|_{X^0}\\
&&\,+|J\sum_{i=1,2}c_i\chi^2_i[L_M-L_i,\,\p_x]U|_{X^0}+|[\p_t,\,D(\p)]U|_{X^0}.
\eeno  By using again  Lemma \ref{commutator lemma}  and Lemma \ref{lemchi} as before, 
we obtain 
\[|F_1(U)|_{X^0}\le \f 1h C(|(\p_xU_2,\,\p^2_xU_1)^t|_{X^0}+|\p_xU|_{X^0}+|U|
_{X^0}),\] and hence, by using  Lemma \ref{space derivative}, we find 
\[|F_1(U)|_{X^0}\le \f 1h C(|\p_tU|_{X^0}+|U|
_{X^0}+|\p_tU_2|_{L^2}+|U_2|_{L^2}).\]  
In a similar way, we also have 
\[|F_1(U)|_{L^2}\le \f 1h C(|\p_tU|_{X^0}+|U|
_{X^0}+|\p_tU_2|_{L^2}+|U_2|_{L^2}).\]
Going back to (\ref{estimate DU 1}), we obtain
 \beno &&|D(\p)U(t)|_{X^0}+|D(\p)U_2(t)|_{L^2}\\
&&\le h^\f14Ce^
{\eps_0(c_2-c_1)t/4}(|U(0)|_{X^1}+|U_2(0)|_{\td H^1})\\
&&\qquad+h^{-\f34} C\int^t_0  e^
{\eps_0(c_2-c_1)(t-\tau)/4}(|\p_tU(\tau)|_{X^0}+|U(\tau)|
_{X^0}+|\p_tU_2(\tau)|_{L^2}+|U_2(\tau)|_{L^2})d\tau.
\eeno
It remains to deduce an estimate of the time and space derivatives from the previous estimate
that gives only a control of the $D(\partial)$ derivatives of $U$. We first observe that 
\beno
&&|D(\p)U(t)|_{X^0}+|D(\p)U_2(t)|_{L^2}\\
&&\ge |\p_tU(t)|_{X^0}+|\p_tU_2(t)|_{L^2}
-\sum_{i=1,2}c_i
(|\chi^2_i\p_x U(t)|_{X^0}+|\chi^2_i\p_x U_2(t)|_{L^2})\\
&&\ge |\p_tU(t)|_{X^0}+|\p_tU_2(t)|_{L^2}-C(|\p_xU|_{X^0}+|\p_xU|_{L^2}).
\eeno 
We know from (\ref{estimate for order-1 space derivative}) that
\[|D(\p)U(t)|_{X^0}+|D(\p)U_2(t)|_{L^2}
\ge  |\p_tU(t)|_{X^0}+|\p_tU_2(t)|_{L^2}-C(|U(t)|_{X^0}+|U_2(t)|_{L^2}).\]
We have thus proven that 
\beq\label{estimate DU 2}
\begin{split}
& |\p_tU(t)|_{X^0}+|\p_tU_2(t)|_{L^2}-C(|U(t)|_{X^0}+|U_2(t)|_{L^2})\\
&\le  h^\f14Ce^
{\eps_0(c_2-c_1)t/4}(|U(0)|_{X^1}+|U_2(0)|_{L^2} + | \partial_{t} U_{2}(0) |_{L^2})\\
&\,+h^{-\f34} C\int^t_0  e^
{\eps_0(c_2-c_1)(t-\tau)/4}(|\p_tU(\tau)|_{X^0}+|U(\tau)|
_{X^0}+|U_2(\tau)|_{L^2}+|\p_tU_2(\tau)|_{L^2})d\tau.\end{split}
\eeq  
From   Proposition \ref{prop low order}, we have 
\[|U(\tau)|_{X^0}+|U_2(\tau)|_{L^2}\le h^\f14 C(|U(0)|_{X^0}+|U_2(0)|_{L^2})e^
{\eps_0(c_2-c_1)\tau/4},\qquad\hbox{for}\quad \tau\in [0,\,t],\]
and therefore
\[
h^{-\f34}C\int^t_0 e^
{\eps_0(c_2-c_1)(t-\tau)/4}(|U(\tau)|_{X^0}+|U_2(\tau)|_{L^2})d\tau
\le t\,h^{-\f12}C(|U(0)|_{X^0}
+|U_2(0)|_{L^2})e^
{\eps_0(c_2-c_1)t/4},
\] 
consequently, we finally deduce from  (\ref{estimate DU 2}) that 
\beno
|\p_tU(t)|_{X^0}+|\p_tU_2(t)|_{L^2}
&\le& h^\f14C(1+h^{-\f34}t)e^
{\eps_0(c_2-c_1)t/4}(|U(0)|_{X^1}+|U_2(0)|_{L^2} + |\partial_{t} U_2(0)|_{L^2}  )\\
&&\, +h^{-\f34} C\int^t_0  e^
{\eps_0(c_2-c_1)(t-\tau)/4}(|\p_tU(\tau)|_{X^0}+|\p_tU_2(\tau)|_{L^2})d\tau.
 \eeno
 Note that we could avoid the additional algebraic growth but that we do not need to refine.
 
 The  Gronwall's inequality yields that  there exists a constant $C_1$ such that
\[|\p_tU(t)|_{X^0}+|\p_tU_2(t)|_{L^2}\le h^\f14 C_1(1+h^{-\f34} t)e^{\eps_0(c_2-c_1)
t/4+h^{-\f34}C_1 t}(|U(0)|_{X^1}+\sum_{\al=0,1}|\p^\al_tU_2(0)|_{L^2}).\]
This yields the desired estimate for the time derivative by  taking $h$ large enough.
 We  finally deduce  from (\ref
{estimate for order-1 space derivative}) that the same  estimate also holds for 
$|\p_xU(t)|_{X^0}$. This ends the proof of  the order-1 energy estimate.

\noindent{\bf Higher-order energy estimates}. We will do this by an induction
argument. First of all, assume that we already have the estimates for $|U|_{X^{k-1}}$ 
and
$|\ka(D)\p^\al_tU_2|_{L^2}$ with $k\ge 2$ and $\al\le k-1$:
\beno
&&|U|_{X^{k-1}}+\sum_{\al\le k-1}|\p^\al_t U_2|_{L^2}\\
&&\le  
h^\f14 C_{k-1}(|U(0)|^2_{X^{k-1}}+\sum_{\al\le k-1}|\p^\al_tU_2(0)|^2_{L^2})
(1+h^{-\f34} t)^{k-1}e^{\eps_0(c_2-c_1)
t/4+h^{-\f34}C_1 t}.
\eeno As  previously, 
 we get start with the estimate of  $|D^k(\p)U|_{X^0}$.
One can write  the system solved by  $D^k(\p)U$ as
\beq\label{eqn for D^kU}
\p_tD^k(\p)U=JL_MD(\p)^kU+F_k(U)
\eeq where 
\[F_k(U)=\sum^{k-1}_{i=0}D(\p)^i\big( J[D(\p),\,L_M]+[\p_t,\,D(\p)]\big)D(\p)^{k-1-i}
U\] will be considered as the
 source term.  

From the Duhamel formula, we find 
\[
D(\p)^kU(t)=S_{M}(t,0)D(\p)^kU(0)+\int^t_0 S_{M}(t,\tau)F_k(U(\tau))d\tau
\] 
and therefore, we again obtain from Proposition \ref{prop low order} that                                     
\beq\label{estimate D^kU 1}\begin{split}&|D(\p)^kU(t)|_{X^0\cap L^2}\\
& \le h^\f14Ce^
{\eps_0(c_2-c_1)t/4}(|U(0)|_{X^k}+\sum_{\al\le k-1}|\p^\al_tU_2(0)|_{L^2})+h^\f14C
\int^t_0 e^{\eps_0(c_2-c_1)(t-\tau)/4}|F_k(\tau)|_{X^0\cap L^2}d\tau.
\end{split}\eeq 
It remains  to estimate  the source term $F_k(t)$.
We have 
\beno
&&|F_k(U)|_{X^0}\\
&&\le \sum^{k-1}_{i=0}\big(|JD(\p)^i[D(\p),\,L_M]D(\p)^{k-1-i}U|_{X^0}
+|D(\p)^i[\p_t,\,D(\p)]D(\p)^{k-1-i}U|_{X^0}\big)\\
&&\le \sum^{k-1}_{i=0}\sum_{j=1,2}
\big[|JD(\p)^i\chi^2_j[L_M-L_j,\p_t]D(\p)^{k-1-i}U|_{X^0}
+c_j|JD(\p)^i[L_M,\,\chi^2_j] \p_xD(\p)^{k-1-i}U|_{X^0}\\
&&\quad+c_j|JD(\p)^i\chi^2_j[L_M-L_j,\,\p_x]D(\p)^{k-1-i}U|_{X^0}\big]
+\sum^{k-1}_{i=0}|D(\p)^{i}[\p_t,\,D(\p)]D(\p)^{k-1-i}U|_{X^0}.
\eeno 
Consequently, by using again the same arguments to estimate the commutators
 and the observation (already used   in the order-1 estimate) that  the commutator 
$[D(\p),\,L_M]=S$ (computed in  (\ref{operator S})),  acts like 
$\f1h L_M$, we  
 can get   that
\[|F_k(U)|_{X^0}+|F_k(U)|_{L^2}\le \f 1h C_k\sum_{\stackrel{\al+\be=k}{\al\le k-1}}
(|(\p^\al_t\p^\be_xU_2,\,\p^\al_t\p^{\be+1}_xU_1)^t|_{X^0} +|\p^\al_t\p^\be_x
U|_{X^0})+\f 1h C|U|_{X^{k-1}},\] and so by using Lemma \ref{space derivative}
again  we arrive at  
\[|F_k(U)|_{X^0}+|F_k(U)|_{L^2}\le \f 1hC_k(|\p^k_t U|_{X^0}+|U|_{X^{k-1}}+
\sum_{\al\le k}|\p^\al_tU_2|_{L^2}).\] Going back to (\ref{estimate D^kU 1}) one has
\beno &&|D(\p)^kU(t)|_{X^0}+|D(\p)^kU_2(t)|_{L^2}\\
&& \le h^\f14Ce^
{\eps_0(c_2-c_1)t/4}(|U(0)|_{X^k}+\sum_{\al\le k}|\p^\al_tU_2(0)|_{L^2}))\\
&&\,+h^{-\f34} C_k\int^t_0  e^
{\eps_0(c_2-c_1)(t-\tau)/4}(|\p^k_t U(\tau)|_{X^0}+|U(\tau)|_{X^{k-1}}+
\sum_{\al\le k}|\p^\al_tU_2(\tau)|_{L^2})d\tau.
\eeno For the left-hand side, we observe that 
\[|D(\p)^kU(t)|_{X^0}+|D(\p)^kU_2(t)|_{L^2}\ge |\p^k_tU(t)|_{X^0}+|\p^k_tU_2(t)|_{L^2}-C(\sum_{\stackrel{\al+\be\le k}
{\al\le k-1}}|\p^\al_t\p^\be_x U(t)|_{X^0}+|U(t)|_{X^{k-1}}),\] which together with
(\ref{estimate for high-order space derivative}) allows to get 
\[ |D(\p)^kU(t)|_{X^0}+|D(\p)^kU_2(t)|_{L^2}
\ge |\p^k_t U(t)|_{X^0}+|\p^k_tU_2|_{L^2}-C(|U(t)|_{X^{k-1}}+\sum_{\al\le k-1}
|\p^\al_tU_2|_{L^2}).\]
Consequently, by using the induction assumption to estimate  $|U|_{X^{k-1}}+
\sum_{\al\le k-1}|\p^\al_tU_2|_{L^2}$, we find
\beno
|\p^k_t U(t)|_{X^0}+|\p^k_tU_2|_{L^2}&\le&
 h^\f14 C_{k-1}(1+h^{-\f34} t)^{k-1}e^{\eps_0(c_2-c_1)
t/4+h^{-\f34}C_1 t}(|U(0)|_{X^k}+\sum_{\al\le k}|\p^\al_tU_2(0)|_{L^2})\\
&&\,+h^{-\f34} C_k\int^t_0  e^
{\eps_0(c_2-c_1)(t-\tau)/4}(|\p^k_t U(\tau)|_{X^0}+|\p^k_tU_2(\tau)|_{L^2})d\tau. 
\eeno Using  the Gronwall's inequality as before yields
\[
|\p^k_t U(t)|_{X^0}+|\p^k_tU_2|_{L^2}\le h^\f14 C_k(|U(0)|_{X^k}+ \sum_{\alpha \leq k} \partial_{t}^\alpha |U_2(0)|)
(1+h^{-\f34}t)^k
e^{\eps_0(c_2-c_1)
t/4+h^{-\f34}C_1 t}.
\] 
Finally, by  using again Lemma \ref{space derivative} and  by taking $h$ large enough 
such that $h^{-\f34}C_1\le\eps_0(c_2-c_1)/4$  we  conclude that
 \[
|U(t)|_{X^k}+\sum_{
\al\le k}|\p^\al_tU_2|_{L^2} \le
h^\f14 C_k(|U(0)|_{X^k}+ \sum_{\alpha \leq k }|\partial_{t}^\alpha U_2(0)|)(1+\eps_0(c_2-c_1)t)^k
e^{\eps_0(c_2-c_1)t/2},\quad \forall t\ge 0. \]  This completes the proof of Theorem \ref{theo homogeneous}.

\begin{rem}
\label{remback}
Note that in the proof of Theorem \ref{theo homogeneous}, the only information about the positions of the solitary waves
 are in the interaction estimates given in  Lemma \ref{lemchi}. These  localization estimates  are  true
 on any interval of time for which  the distance between the centers of the solitary waves is bigger than $h$ which is considered
  as a large parameter. 
 We can use this remark to get an estimate for $S_{M}(t, \tau)$ for $0 \leq t \leq \tau.$
  Indeed by using again the reversibility symmetry of the water waves system
  and  the symmetries of the solitary waves of Theorem \ref{theoOS},  consider $U(t,x)$
   the solution of  \eqref{new homogeneous linear equation}  with initial data at $t=\tau$, then 
   $$ \tilde U (t,x)= \big( U_{1}(\tau-t, -x) ,  -U_{2}( \tau - t,  -x) \big)^t$$ 
   is still a solution  on $[0, \tau]$ of \eqref{new homogeneous linear equation} with
    $ M(t,x)= Q_{1}(x-c_{1} t )+ Q_{2}(x-c_{2} t - h)$
   replaced by 
    $$\tilde M(t,x)= Q_{1}(x + c_{1} \tau - c_{1}t ) + Q_{2}(x + c_{2} \tau + h - c_{2} t )$$
    and with initial data for $\tilde U$ at $t=0$. Consequently, we see that on $[0,\tau]$, 
     the center of the solitary waves are located at $ x= x_{1}= -c_{1} |t - \tau|$ and
      $x = x_{2}=- c_{2}|t- \tau | - h$. Consequently, the slow solitary wave $Q_{1}$
       is now located on the right. Nevertheless, we observe that  for $t \in [0, \tau]$, we still have 
       $$ x_{1} - x_{2} \geq (c_{2} - c_{1}) | t - \tau| + h \geq h$$
       and thus the solitary waves are still at least at distance $h$ uniformly in $\tau$. 
       Consequently, we  still get as in   Theorem \ref{theo homogeneous}   
       that
       $$  |\tilde{U}(t) |_{E^k} \leq h^{1 \over 4} C_{k} |\tilde{U}(0) |_{H^{s(k)}}\big( 1+ \epsilon_{0}(c_{2}- c_{1}) t^k \big) e^{\epsilon_{0}t \over 2}, \quad \forall t \in [0, \tau].$$
       This yields in the original time variable
       $$  |U(t) |_{E^k} \leq h^{1 \over 4} C_{k} |U(\tau) |_{H^{s(k)}}\big( 1+ \epsilon_{0}(c_{2}- c_{1}) (\tau - t)^k \big) e^{\epsilon_{0}(\tau - t ) \over 2}, \quad \forall t \in [0, \tau].$$
       By combining with Theorem \ref{theo homogeneous}, we thus obtain that the fundamental solution $S_{M}(t,\tau) $
        of system \eqref{new homogeneous linear equation} enjoys the estimate
        \beq
        \label{semigroup2}
          | S_{M}(t, \tau)U |_{E^k} \leq h^{1 \over 4} C_{k} |U|_{H^{s(k)}}\big( 1+ \epsilon_{0}(c_{2}- c_{1}) |t  - \tau |^k) \big) e^{\epsilon_{0}| t  - \tau |\over 2}, \quad \forall t, \, \tau \geq 0.
          \eeq
 \end{rem}

\subsection{Proof of Proposition~\ref{approximate solution M+V}:
Construction of the approximate solution} Now we can go back to the
study of the  linear systems solved  by  $V_l$ ($1\le l\le N$) 
\eqref{systV1}, \eqref{systVl} in order to prove Proposition \ref{approximate solution M+V}. 

We first note  from the fact that
 \eqref{homogeneous linear equation} and \eqref{new homogeneous linear equation}
 are equivalent via the transformation  $U(t)=RV(t)$, with $R$ invertible, we 
 get from Theorem  \ref{theo homogeneous} and \eqref{semigroup1}
 \beq
 \label{SM1}
| S_{M}^\Lambda (t,\tau) V |_{E^k}\le h^\f14
C_k(\epsilon_{0})|V|_{H^{s(k)}}(1+ |t- \tau|^k ) 
e^{\eps_0(c_2-c_1)|t- \tau|/2}, \quad \forall t, \,  \tau \geq 0
\eeq 
where $S_{M}^\Lambda(t, \tau)$ is the fundamental solution of   the  system \eqref{homogeneous linear equation}.

 Let us go  back to the
construction of the approximate solution $V(t,x)=\sum^N_{l=1}\del^lV_l(t,x)$ with $\del=e^{-\eps_0h}$.
 For 
$V_1(t)$, we have to solve 
\[
\p_tV_1-J\Lam[M]V_1=-\tilde R_M,
\] where the right hand side satisfies  (see \ref{estimate R_M}) the estimate 
\[
 |\td R_M|_{E^k}\le C_ke^{-\eps_0(c_2-c_1)t},\quad \forall t\ge
 0,k\in\N.
\] 
We choose the solution 
\[V_1(t,x)=-\int^\infty_tS_{M}^\Lambda(t, \tau)\td R_M(\tau)d\tau.\]
From the estimate \eqref{SM1} of the fundamental solution, $V_{1}$ is well-defined and satisfies the estimate 
 \beno |V_1(t)|_{E^k}&\le&
h^\f14
C_k \int^\infty_t(1+ | t - \tau|)^k
e^{\eps_0(c_2-c_1)(\tau-t)/2}  \sum_{l=0}^k   \| \partial_{t}^l\td R_M(\tau)\|_{H^{s(k -l)}}d\tau\\
&\le & h^\f14
C_k \int^\infty_t(1+ | t - \tau|)^k
e^{\eps_0(c_2-c_1)(\tau-t)/2}e^{-\eps_0
(c_2-c_1)\tau}d\tau\\
&\le& h^{\f14}C_{k, 1}(\epsilon_{0})e^{-\eps_0(c_2-c_1)t},\quad
\hbox{for}\quad t\ge 0.
 \eeno
  For the general case of $V_l$ ($2\le l\le N$), we have to solve
\begin{equation}
\label{eqVl}
\p_tV_l-J\Lam[M]V_l=\sum^l_{p=2}\sum_{\stackrel{1\le
l_1,\dots,l_p\le
N}{l_1+\dots+l_p=l}}\f1{p!}D^p\cF[M](V_{l_1},\dots,V_{l_p}):=R_l(t,x).
\end{equation}
We shall use an induction argument. Let us  assume that we
already  have already built   $V_j$, $1\le
j\le l-1$ that satisfy the estimate 
\[
|V_j(t)|_{E^k}\le
h^{\f{2j-1}4}C_{k,j}(\eps_0)e^{-j\eps_0(c_2-c_1)t},\quad
\forall  t\ge 0.
\] We get for the right-hand side of  \eqref{eqVl} that
\beno |R_l(t,x)|_{E^k}&=&\big|\sum^l_{p=2}\sum_{\stackrel{1\le
l_1,\dots,l_p\le
N}{l_1+\dots+l_p=l}}\f1{p!}D^p\cF[M](V_{l_1},\dots,V_{l_p})\big|_{E^k}\\
&\le&
h^\f{2l-1}4C_{k,l}(\eps_0)
e^{-l\eps_0(c_2-c_1)t}. \eeno Taking
\[V_l(t,x)=-\int^\infty_tS_{M}^\Lambda(t, \tau)R_l(\tau)d\tau\] as a solution,
we  get thanks to \eqref{SM1}
 \beno |V_l(t)|_{E^k}&\le &h^\f{2l-1}4C_{k,l}(\eps_0)
\int^\infty_t(1+\eps_0(c_2-c_1)(\tau-t))^ke^{\eps_0(c_2-c_1)(\tau-t)/2}
e^{-l\eps_0(c_2-c_1)\tau}
d\tau\\
&\le& h^\f{2l-1}4
C_{k,l}(\eps_0)
e^{-l\eps_0(c_2-c_1)t},\quad \forall   t\ge0. \eeno 
This ends the proof of  Proposition~\ref{approximate solution M+V}.
\ef
%%%%%%%%%%%%%%%%%%%%%%%%%%%%%%%%%%%%%%%%%%%%%%%%%%%%%%%%%%%%%%%%%
%%%%%%%%%%%%%%%%%%%%%%%%%%%%%%%%%%%%%%%%%%%%%%%%%%%%%%%%%%%%%%%%%%
\section{The nonlinear problem}\label{section end of the proof}
After the study of the approximate solution $U^a$ of water-wave
system, we need to consider the remainder solution $U^R=U-U^a$ where
$U$ is the solution of the water-wave system (\ref{simple form for WW system})
  and $U^R$ satisfies
\beq\label{eqn for U^R} \left\{\begin{array}{ll}
\p_tU^R=\cF(U^a+U^R)-\cF(U^a)-R_{ap},\quad t>0,\\
U^R(0) \hbox{ to be fixed later}
\end{array}\right. \eeq
\begin{prop}\label{existence for U^R} Let $m\ge 2$, $U^a\in
W^{m+s}_{[0,\infty)}$ and $R_{ap}\in X^{m+3}_{[0,\infty)}$. There
exists a solution $U^R=(\eta^R,\,\varphi^R)^t\in
L^\infty([0,\infty),\,H^{m+4}\times H^{m+\f72})$ for (\ref{eqn for
U^R}) with a fixed initial value $U^R(0)$ such that
\[H-\|\eta^a\|_{L^\infty}-\|\eta^R\|_{L^\infty} >0\] and
\[
|U^R(t)|_{H^{m+4}\times H^{m+\f72}}\le
C_{N,m}h^{\frac{2N+1}{4}}\del^{N+1}e^{-(N+1)\eps_0(c_2-c_1)t},\quad\hbox{for any}\quad
t\in [0,\infty).
\] 
\end{prop}
\begin{proof}
The proof is left to the reader, it suffices to use the same arguments as in the proof 
 of Theorem~\ref{Instability Thm}. 
 \end{proof}
%%%%%
\begin{proof}[End of the proof of Theorem~\ref{multi-soliton theorem}]
Since we
have already shown the global existence of the remainder solution $U^R$,
we know that there exits a (semi-) global solution
$U(t,x)=U^a(t,x)+U^R(t,x)=M(t,x)+V(t,x)+U^R(t,x)$ to the water-wave
system (\ref{water-wave system}).  It only remains to describe
 the asymptotic behavior of $U(t)$ when $t$ tends to $+\infty$.
Since 
\[U(t)=M(t,x)+\sum^N_{l=1}\del^lV_l(t)+U^R(t)\] 
with $\del=e^{-\eps_0h}$ and with $V_l$ and $U^R$ 
 that  satisfy the estimates from
Proposition~\ref{approximate solution M+V} and Proposition~\ref{existence for U^R}, we have
$$ 
|V_l(t)|_{E^s}\le
h^\f{2l-1}4C_{N,s}e^{-l\eps_0(c_2-c_1)t},\quad
|U^R(t)|_{X^{s+3}}\le
h^{\frac{2N+1}{4}}C_{N,s}\del^{N+1}e^{-(N+1)\eps_0(c_2-c_1)t}. 
$$
This gives in particular 
\[|V_l(t)|_{H^s}\le h^\f{2l-1}4C_{N,s}e^{-l\eps_0(c_2-c_1)t},\quad
|U^R(t)|_{H^s}\le h^{\frac{2N+1}{4}}C_{N,s}\del^{N+1}e^{-(N+1)\eps_0(c_2-c_1)t}\]
for any $t\in [0,\infty)$.  We thus get that
\[\lim_{t\rightarrow +\infty}|U(t)-M(t)|_{H^s}=0.\] 
This ends the proof of Theorem \ref{multi-soliton theorem}. 
\end{proof}
%%%%%%%%%%%%%%%%%%%%%%%%%%%%%%%%%%%%%%%%%%%%%%%%%%%%%%%%%
%%%%%%%%%%%%%%%%%%%%%%%%%%%%%%%%%%%%%%%%%%%%%%%%%%%%%%%%%%
\section{Appendix:  Proof  of  Lemma \ref{commutator lemma}}
\subsection{Proof of \eqref{estlemcom1}} We shall  only  prove the case $i=1$, the 
  case $i=2$ can then be obtained by symmetric arguments.  Let us recall that as in \eqref{defDNbis}, \eqref{elliptic problem-new},  the Dirichlet-Neumann operator
$G[\eta]$  can be defined by:
\[G[\eta]u=\p^P_n u^b|_{z=0}\] 
where $u^b$ satisfies the elliptic system on the flat strip $\cS=\R\times [-1,0]$
\beq\label{u^b eqn}
\left\{\begin{array}{ll}
\na_{x,z}\cdot P\na_{x,z}u^b=0,\quad \hbox{in} \quad \cS\\
u^b|_{z=0}=u,\,\qquad \p^P_n u^b|_{z=-1}=0. 
\end{array}\right.
\eeq with the notations
\[
P=P[\eta]=\left(\begin{matrix} H+\eta & -(z+1)\p_x\eta\\
-(z+1)\p_x\eta & \f{1+(z+1)^2(\p_x\eta)^2}{H+\eta}
\end{matrix}\right)\] and  $\p^P_n={\bf n}\cdot P\na_{x,z}$ where ${\bf n}= -e_{z}$ is
 the outward unit normal to the boundary 
 $z=-1$.  Note that  $\p^P_n\vert_{z=-1}=  - {1 \over H +  \eta} \partial_{z}.$

One can see by the Green's Formula that
\[(G[\eta] u,v)=\int_{\cS} P\na_{x,z} u^b\cdot \na_{x,z} v^b\] and so we have
\[([\p_t,\,G[\eta]]u,\,u)=\int_{\cS}\p_tP\na_{x,z}u^b\cdot\na_{x,z}u^b
-2\int_{\cS}P\na_{x,z}((\p_tu)^b-\p_tu^b)\cdot\na_{x,z}u^b.
\] In the following we shall  use the notations ($j=M,1,2$)
\[G[\eta_j]=G_j,\quad P[\eta_j]=P_j\quad
\hbox{and}\quad u^b[\eta_j]=u^b_j.\] 
  We  also  set   $u=\chi_1U_2$ for the sake
 of convenience.  Then we can write
\beq\label{commutator 1}\begin{split}
&([\p_t,\,G_M-G_1]u,\,u)\\
&=\int_{\cS}\p_t(P_M-P_1)\na_{x,z}u^b_M\cdot \na_{x,z}u^b_M +\int_{\cS}\p_tP_1
\na_{x,z}(u^b_M-u^b_1)\cdot\na_{x,z} u^b_M\\
&\quad +\int_{\cS}\p_tP_1\na_{x,z}u^b_1\cdot \na_{x,z}(u^b_M-u^b_1)
-2\int_{\cS}(P_M-P_1)\na_{x,z}u^b_M\cdot\na_{x,z}((\p_tu)^b_M-\p_tu^b_M)\\
&\quad -2\int_{\cS}P_1\na_{x,z}[(\p_tu)^b_M-\p_tu^b_M-(\p_tu)^b_1+\p_tu^b_1]
\cdot
\na_{x,z}u^b_M\\
&\quad -2\int_{\cS}P_1\na_{x,z}((\p_tu)^b_1-\p_tu^b_1)\cdot\na_{x,z}
(u^b_M-u^b_1).\end{split}\eeq
Let us recall that the solitary wave  $\eta_2$ satisfies the exponential decay 
estimate
\[ |\p^\al_x\eta_2(x-h-c_2t)|\le C_\al e^{-d (1+|x-h-c_2t|^2)^\f12}\quad
\hbox{for}\quad x,t\in\R
\]
Consequently, we shall use the weight
 $f$ defined by 
 $$f(t,x)=e^{-\eps(1+|x-h-c_2t|^2)^\f12}$$
  where $\eps \in [0, d]$ will be chosen sufficiently small. 
  
  From \eqref{commutator 1}, we first get the estimate
\begin{multline}
\label{commutator 2 bis}
 \big| ([\p_t,\,G_M-G_1]u,\,u) \big| \leq C
   \Big( \| f  \nabla_{x,z} u^b_{M} \|_{L^2(\mathcal{S})}  +   \|\nabla_{x, z} (u^b_{M} -u^b_{1} ) \|_{L^2(\mathcal{S})}
\\
  +  \| \nabla_{x,z} \big((\partial_{t} u)^b_{M} - \partial_{t} u^b_{M} -  (\partial_{t} u)^b_{1} + \partial_{t} u^b_{1} \big) \|_{L^2(\mathcal{S})}
  \Big) 
  \cdot   \Big(  \|\nabla_{x,z} u^b_{M} \|_{L^2(\mathcal{S})} + \|\nabla_{x, z} u^b_{1} \|_{L^2(\mathcal{S})} 
  \\
  + \| \nabla_{x,z} \big(  (\partial_{t} u)^b_{M} - \partial_{t} u^b_{M} \big)  \|_{L^2(\mathcal{S})} +   \| \nabla_{x,z} \big(  (\partial_{t} u)^b_{1} - \partial_{t} u^b_{1} \big)  \|_{L^2(\mathcal{S})}
 \Big).
\end{multline}  
From the elliptic problem \eqref{u^b eqn}, we first get the estimates (this follows for example by using the decomposition \eqref{decDN} of the solution, 
 we refer to \cite{AL}, \cite{RT} for example)
\beq
\label{comubm1est}
 \| \nabla_{x,z} u^b_{i} \|_{L^2(\mathcal{S})} \leq C | \mathfrak{P} u |_{L^2(\mathbb{R}^2)}
 \leq 
  C (|\mathfrak P U_2|^2_{L^2}+|U_2|^2
_{L^2}), \quad i= M, \, 1,
 \eeq
 where in the last step we used \eqref{basic}.
  To conclude, 
 we still  need the estimates for 
$ f  u^b_M$, $u^b_M-u^b_1$, $(\p_tu)^b_j-\p_tu^b_j$ 
and $(\p_tu)^b_M-\p_tu^b_M-(\p_tu)^b_1+\p_tu^b_1$ with $j=M,1$. We shall  deal with 
them one by one. 

\noindent{\bf 1)  Estimate for $e^{-\eps(1+|x-h-c_2t|^2)^\f12} u^b_M$}. 
As 
in the proof of Proposition~\ref{D-N e-decay estimate}, we get that  
$e^{-\eps(1+|x-h-c_2t|^2)^\f12} u^b_M$ solves the elliptic equation
\[
\left\{\begin{array}{ll}
\na_{x,z}\cdot P_M\na_{x,z}(f u^b_M)=
-[f ,\,\na_{x,z}\cdot P_M\na_{x,z}]u^b_M,\quad\hbox{in}
\quad \cS,\\
f u^b_M|_{z=0}=f u,\quad
\p^{P_M}_n(f u^b_M)|_{z=-1}=-[f ,\,\p^{P_M}_n]u^b_M|_{z=-1}= 0
\end{array}\right.
\] 
We introduce the decomposition  
$$fu^b_M
= m (z, |D|)(fu)+v$$ where
 $m(z, |D|)$ is the Fourier multiplier ${ \mbox{cosh }(|D|(z+1)) \over \mbox{cosh } |D| }.$ 
 We get for  $v$  the 
system
\[
\left\{\begin{array}{ll}
\na_{x,z}\cdot P_M\na_{x,z}v=-\na_{x,z}\cdot P_M\na_{x,z}( m (f u ) )
-[f ,\,\na_{x,z}\cdot P_M\na_{x,z}]u^b_M,
\quad\hbox{in}
\quad \cS,\\
v|_{z=0}=0,\quad
\p_{z}v|_{z=-1}= 0.
\end{array}\right.\]
 From an energy estimate (as in the proof of Proposition~\ref{D-N e-decay estimate}), we obtain that
\begin{multline}
\label{comenergie1}
  \| \nabla_{x,z} v \|_{L^2(\mathcal{S})}^2   \leq  C \Big( \| \nabla_{x, z}(m(z,|D|) (fu) ) \|_{L^2(\mathcal{S})}     \| \nabla_{x,z} v \|_{L^2(\mathcal{S})}
   \\+  ( \eps^2 \|f u^b_{M} \|_{L^2(\mathcal{S})}  + \eps \| \nabla_{x,z} (f u^{b}_{M}) \|_{L^2(\mathcal{S})} ) ( \|v\|_{L^2(\mathcal{S})} + \| \nabla_{x,z} v \|_{L^2(\mathcal{S})}
   ) \Big)
   \end{multline}
   where in particular we have use the fact that
   $$ | \nabla_{x,z}  f| \lesssim \eps f.$$
% Using Green's Formula leads to
%\beno
%\int_{\cS}P_M\na_{x,z}v\cdot\na_{x,z}v&=&-\int_{\cS}(\na_{x,z}\cdot P_M\na_{x,z} 
%v)\,v+\int_{z=-1}\p^{P_M}_n v\cdot v\\
%&=&\int_{\cS}\na_{x,z}\cdot P_M\na_{x,z}
%(\chi(z|D|)fu)\cdot v+\int_{\cS}[f,\,\na_{x,z}\cdot P_M\na_{x,z}]u^b_M\cdot v\\
%&&\quad -\int_{\R}\p^{P_M}_n(\chi(z|D|)fu)|_{-1}\cdot v|_{-1}
%-\int_{\R}[f,\,\p^{P_M}_n]u^b_M|_{-1}\cdot v|_{-1}\\
%&:=& A_1+A_2+A_3+A_4.
%\eeno From the proof of Prop \ref{D-N e-decay estimate} we can show
%\beno
%A_2&\le & C(\eps^2\|fu^b_M\|+\eps \|\na_{x,z}(f u^b_M)\|_2)\|v\|_2\\
%&\le & C(\eps\|\na_{x,z} v\|_2+\|\chi(z|D|)f u\|_2+\|\na_{x,z}(\chi(z|D|)f u)\|_2)
%\|v\|_2 \\
%&\le& \eps C\|\na_{x,z}v\|^2_2+\f1h C_\eps (|\mathfrak P U_2|^2_{L^2}+|U_2|^2
%_{L^2}),
%\eeno where we used 
Next, we can use the Poincar\'e inequality  in the strip $\mathcal{S}$ which yields
$$\|v\|_{L^2(\mathcal{S})}\le C\|\na_{x,z}v\|_{L^2(\mathcal{S})}$$
 and the  fact that $u = \chi_{1} U_{2}$ which yields thanks to \eqref{basic} and Lemma~\ref{lemchi}
$$ \|m(z,|D|)(f u)\|_{L^2(\mathcal{S})}+\|\na_{x,z}(m(z,|D|)(f u))
\|_{L^2(\mathcal{S})}\le \f 1h C(|\mathfrak P U_2|_{L^2}+|U_2|_{L^2})
$$ 
 to obtain from \eqref{comenergie1} by  taking $\eps$ small enough  that
 \[
 \|v\|_{L^2}^2 + \|\na_{x,z}v\|^2_{L^2}\le \f1h C (|\mathfrak P U_2|^2_{L^2}+|U_2|^2
_{L^2}).\] 
This yields 
\begin{equation}
\label{comfuBM}
\|\na_{x,z}(fu^b_M)\|^2_{L^2}+\|fu^b_M\|^2_{L^2}\le \f1h C (|\mathfrak P U_2|^2_{L^2}
+|U_2|^2_{L^2}).
\end{equation}

\noindent{\bf 2)  Estimate for $u^b_M-u^b_1$}.  Let us set 
 $${\bf g}=-(P_M-P_1)\na_{x,z}
u^b_M, $$
  we get that  $u^b_M-u^b_1$  solves  
\[
\left\{\begin{array}{ll}
\na_{x,z}\cdot P_1\na_{x,z}(u^b_M-u^b_1)=
\na_{x,z}\cdot{\bf g},\quad\hbox{in}
\quad \cS,\\
u^b_M-u^b_1|_{z=0}=0,\quad
\partial_{z}(u^b_M-u^b_1)|_{z=-1}= 0 
\end{array}\right.
\] 
The standard energy estimate  for this  problem yields 
\[\|\na_{x,z}(u^b_M-u^b_1)\|^2_{L^2}\le C\|{\bf g}\|^2_2\le C(\|fu^b_M\|^2_{L^2}+
\|\na_{x,z}(fu^b_M)\|^2_{L^2}),\] which results (combined with the Poincar\'e inequality and \eqref{comfuBM}) in 
\[\|u^b_M-u^b_1\|^2_{L^2}+
\|\na_{x,z}(u^b_M-u^b_1)\|^2_{L^2}\le \f1h C(|\mathfrak P
 U_2|^2_{L^2}+|U_2|^2_{L^2}).\]

\noindent{\bf 3)  Estimate for $(\p_tu)^b_j-\p_tu^b_j$}. 
 As previously, we can set  
${\bf g}_j=\p_t P_j\na_{x,z} u^b_j$ with $j=M,1$,  to get that $(\p_tu)^b_j-\p_tu^b_j$ satisfies
the system 
\[
\left\{\begin{array}{ll}
\na_{x,z}\cdot P_j\na_{x,z}((\p_tu)^b_j-\p_tu^b_j)=
\na_{x,z}\cdot{\bf g}_j,\quad\hbox{in}
\quad \cS,\\
(\p_tu)^b_j-\p_tu^b_j|_{z=0}=0,\quad
\p^{P_M}_n((\p_tu)^b_j-\p_tu^b_j)|_{z=-1}=-{\bf e_z}\cdot {\bf g}_j|_{z=-1}= 0
\end{array}\right.
\] 
 and we obtain that 
\begin{equation}
\label{comestubM10}\|(\p_tu)^b_j-\p_tu^b_j\|^2_{L^2}+\|\na_{x,z}((\p_tu)^b_j-\p_tu^b_j)\|^2_{L^2}\le
C\|{\bf g}_j\|^2_{L^2}\le C(|\mathfrak P U_2|^2_{L^2}
+| U_2|^2_{L^2})
\end{equation}
thanks to \eqref{comubm1est}.

\noindent{\bf 4)  Estimate for $(\p_tu)^b_M-\p_tu^b_M-(\p_tu)^b_1+\p_tu^b_1$}. We 
write $v=(\p_tu)^b_M-\p_tu^b_M-(\p_tu)^b_1+\p_tu^b_1$ here and we know from 3) 
that
\[
\left\{\begin{array}{ll}
\na_{x,z}\cdot P_1\na_{x,z}v=
\na_{x,z}\cdot[{\bf g}_M-{\bf g}_1-(P_M-P_1)\na_{x,z}((\p_tu)^b_M-\p_tu^b_M)],
\quad\hbox{in}
\quad \cS,\\
v|_{z=0}=0,\quad
\p^{P_M}_nv|_{z=-1}=-{\bf e_z}\cdot [{\bf g}_M-{\bf g}_1-(P_M-P_1)\na_{x,z}
((\p_tu)^b_M-\p_tu^b_M)]|_{z=-1}
\end{array}\right.
\] 
From a basic energy estimate, we obtain 
\begin{eqnarray}
\nonumber\|\na_{x,z} v\|_2&\le& C(\|(P_M-P_1)\na_{x,z}((\p_tu)^b_M-\p_tu^b_M)\|_{L^2}
+\|\p_t(P_M-P_1)\na_{x,z}u^b_M\|_{L^2}\\
\nonumber&&\, +\|\p_tP_1\na_{x,z} (u^b_M-u^b_1)\|_{L^2}\\
 \label{estcomubM3}&\le & C[\|f((\p_tu)^b_M-\p_tu^b_M)\|_{L^2}
+\|\na_{x,z}(f(\p_tu)^b_M-\p_tu^b_M)\|_{L^2}
+\|fu^b_M\|_{L^2}\\
\nonumber &&\,+\|\na_{x,z}(fu^b_M)\|_{L^2}+\|\na_{x,z}(u^b_M-u^b_1)\|_{L^2}]
\end{eqnarray} where  we use again $f=e^{-\eps(1+|x-h-c_2t|^2)^\f12}$.
 To conclude, we need to estimate  $w=f((\p_tu)^b_M-\p_tu^b_M)$.    We observe that 
  $w$ solves 
\[
\left\{\begin{array}{ll}
\na_{x,z}\cdot P_M\na_{x,z}w=-[f,\,\na_{x,z}\cdot P_M\na_{x,z}]((\p_tu)^b_M
-\p_tu^b_M)-f\na_{x,z}\cdot \p_tP_M\na_{x,z} u^b_M,
\quad\hbox{in}
\quad \cS,\\
w|_{z=0}=0,\quad
\partial_{z} w |_{z=-1}= 0. 
\end{array}\right.
\]  
We can proceed as in step 1) to obtain that 
%Decomposing $w$ into $w=w_1+w_2$ with $w_2$ satisfies 
%\[
%\left\{\begin{array}{ll}
%\na_{x,z}\cdot P_1\na_{x,z}w_2=\na_{x,z}\cdot (f \p_tP_M\na_{x,z} u^b_M),
%\quad\hbox{in}
%\quad \cS,\\
%w_2|_{z=0}=0,\quad
%\p^{P_M}_nw_2|_{z=-1}=-{\bf e_z}\cdot (f \p_tP_M\na_{x,z}u^b_M)|_{-1}.
%\end{array}\right.
%\]  one can see again from Prop 2.4 \cite{AL} that
\[
\|w\|^2_{L^2} + \|\na_{x,z} w \|^2_{L^2}\le C\|f  \p_tP_M\na_{x,z}u^b_M\|^2_{L^2}
 +  \f 1h C(|\mathfrak P U_2|^2_{L^2}+|U_2|^2_{L^2}) \leq  \f 1h C(|\mathfrak P U_2|^2_{L^2}+|U_2|^2_{L^2}) \] 
 thanks to \eqref{comfuBM}.
%It remains to deal with $w_1$. Indeed, using again Green's Formula one arrives at
%\beno
%&&\int_{\cS}P_M\na_{x,z} w_1\cdot \na_{x,z}w_1\\
%&&=\int_{\cS}[f,\,\na_{x,z}\cdot P_M
%\na_{x,z}]((\p_tu)^b_M-\p_tu^b_M) w_1-\int_{\cS}[f,\,\na_{x,z}]\cdot \p_tP_M
%\na_{x,z}u^b_M\,w_1\\
%&&\quad- \int_{\R}[f,\,\p^{P_M}_n]((\p_tu)^b_M-\p_tu^b_M)|_{-1}\cdot w_1|_{-1}\\
%&&\le C(\eps\|w\|_2+\eps\|\na_{x,z}w\|_2+\eps\|f\p_tP_M\na_{x,z}u^b_M\|_2)
%\|\na_{x,z}w_1\|_2\\
%&&\le C(\eps\|\na_{x,z}w_1\|^2_2+\|fu^b_M\|^2_2+\|\na_{x,z}(f u^b_M)\|^2_2).
%\eeno Combing this with the estimate for $fu^b_M$ leads to 
%\[
%\|w\|^2_2,\,\|\na_{x,z}w\|^2_2\le \f1h C(|\mathfrak PU_2|^2_{L^2}+|U_2|^2_{L^2})
%\] where we used  the fact that $\|g\|_2\le C \|\na_{x,z} g\|_2$ and 
%$|g|_{z=-1}|_{L^2}\le C|\na_{x,z} g\|_2$ if $g(x,z)|_{z=0}=0$. 
We can thus also obtain by combining the last estimate, \eqref{estcomubM3} and  the estimates of step 1) and step 2) that  
\[\|v\|^2_{L^2}+\|\na_{x,z} v\|^2_{L^2}\le \f1h C(|\mathfrak PU_2|^2_{L^2}+|U_2|^2
_{L^2}).\] 
Gathering all the previous estimates, we finally obtain \eqref{estlemcom1}.
% from \eqref{commutator 2}.

\subsection{Proof of \eqref{estlemcom2}} Since 
\beno
([[\p_t,\,G_M],\,\chi_1]u,\,v)&=& ([\p_t,\,G_M]\chi_1u,\,v)-([\p_t,\,G_M]u,\,\chi_1v)\\
&=& \p_t(G_M\chi_1u,\,v)-(G_M\chi_1u,\,\p_tv)-(G_M\p_t(\chi_1u),\,v)
-\p_t(\chi_1G_Mu,\,v)\\
&&\,+((\p_t\chi_1)G_Mu,\,v)+(\chi_1G_Mu,\,\p_tv)+(\chi_1G_M\p_tu,\,v),
\eeno we get by using   the  Green's Formula on the flat strip $\cS$  that 
\beno
 ([[\p_t,\,G_M],\,\chi_1]u,\,v)&=&\p_t\int_{\cS} P_M\na_{x,z}(\chi_1u)^b\cdot
\na_{x,z}v^\dagger-\p_t\int_{\cS}\chi_1P_M\na_{x,z}u^b\cdot \na_{x,z}v^\dagger\\
&&\, -\p_t\int_{\cS}(\na_{x,z}\chi_1)\cdot P_M\na_{x,z}u^b\cdot v^\dagger
-\int_{\cS}P_M\na_{x,z}(\chi_1 u )^b\cdot\na_{x,z}(\p_tv)^\dagger\\
&&\,  +\int_{\cS}\chi_1P_M\na_{x,z}u^b\cdot\na_{x,z}(\p_tv)^\dagger
+\int_{\cS}(\na_{x,z}\chi_1)\cdot P_M\na_{x,z}u^b\,(\p_tv)^\dagger\\
&&\, -\int_{\cS}P_M\na_{x,z}(\p_t(\chi_1u))^b\cdot \na_{x,z}v^\dagger
+\int_{\cS}(\p_t\chi_1)P_M\na_{x,z}u^b\cdot\na_{x,z}v^\dagger\\
&&\, +\int_{\cS}\na_{x,z}(\p_t\chi_1)\cdot P_M\na_{x,z}u^b\,v^\dagger
+\int_{\cS}\chi_1P_M\na_{x,z}(\p_tu)^b\cdot\na_{x,z}v^\dagger\\
&&\,  +\int_{\cS}(\na_{x,z}\chi_1)\cdot P_M\na_{x,z}(\p_tu)^b\, v^\dagger
\eeno where $u^b$ is $u^b_M$ are satisfying (\ref{u^b eqn}) with $\eta_M$,
  and $v^\dagger=\chi(z|D|)v$ with 
$\chi$ a smooth compactly supported  cut-off function  such that $\chi(0)= 1$. One can rewrite the above as
\beq\label{commutator 2}
\begin{split}
& ([[\p_t,\,G_M],\,\chi_1]u,\,v)\\
&=\int_{\cS}[\p_tP_M\na_{x,z}(\chi_1u)^b\cdot\na_{x,z}
v^\dagger-\p_tP_M\chi_1\na_{x,z}u^b\cdot\na_{x,z}v^\dagger-(\na_{x,z}\chi_1)
\cdot\p_tP_M\na_{x,z}u^b\,v^\dagger]\\
&\quad +\int_{\cS}[P_M\na_{x,z}(\p_t(\chi_1u)^b-(\p_t(\chi_1u))^b)\cdot
\na_{x,z}v^\dagger-\chi_1P_M\na_{x,z}((\p_tu)^b-\p_tu^b)\cdot \na_{x,z}
v^\dagger\\
&\quad -(\na_{x,z}\chi_1)\cdot P_M\na_{x,z}((\p_tu)^b-\p_tu^b)\,v^\dagger]\\
&:= A_1+A_2
\end{split}\eeq while noticing  that $\p_tv^\dagger=(\p_tv)^\dagger$.  We will deal with 
$A_1$ first. We have
\beno
A_1&=&\int_{\cS}\p_tP_M\na_{x,z}((\chi_1u)^b-\chi_1u^b)\cdot \na_{x,z}v^\dagger
+\int_{\cS}\p_tP_M(\na_{x,z}\chi_1)u^b\cdot\na_{x,z}v^\dagger\\
&&  \,-\int_{\cS}(\na_{x,z}\chi_1)\cdot \p_tP_M\na_{x,z}u^b\,v^\dagger,
\eeno 
 Since 
$(\chi_1u)^b-\chi_1u^b$ solves   the equation
\[\left\{\begin{array}{ll}
 \na_{x,z}\cdot P_M\na_{X,z}((\chi_1u)^b-\chi_1u^b)=[\chi_1,\,\na_{x,z}\cdot P_M
\na_{x,z}]u^b,\quad \hbox{in}\quad \cS,\\
(\chi_1u)^b-\chi_1u^b|_{z=0}=0\quad \p^{P_M}_n((\chi_1u)^b-\chi_1u^b)|_{z=-1}
=[\chi_1,\,\p^{P_M}_n]u^b_{z=-1}= 0, 
\end{array}\right.
\] 
by using the definition of $\chi_1$ and (\ref{chiprop1}),  we again obtain from an energy estimate that 
\begin{equation}
\label{comestubM5}\|\na_{x,z} ((\chi_1u)^b-\chi_1u^b)\|_{L^2}\le \f1h C(|\mathfrak P u|_{L^2}
+|u|_{L^2}).\eeq By  using  again (\ref{chiprop1}), we thus obtain
\[
A_1\le \f 1h C(|\mathfrak P u|^2_{L^2}+|\mathfrak P v|^2_{L^2}+|u|^2_{L^2}
+|v|^2_{L^2}),
\] where we used the fact that $\|\na_{x,z} v^\dagger\|_2\le C|\mathfrak P v|_{L^2}$.
On the other hand, one can rewrite
\beno
A_2&=& \int_{\cS}P_M\na_{x,z}\left(\p_t(\chi_1u)^b-(\p_t(\chi_1u))^b-\chi_1\p_t
u^b+\chi_1(\p_tu)^b\right)\cdot\na_{x,z}v^\dagger\\
&&\, +\int_{\cS}P_M(\na_{x,z}\chi_1)\cdot\left[(\p_tu^b-(\p_tu)^b)\na_{x,z}
v^\dagger-\na_{x,z}(\p_tu^b-(\p_tu)^b)\,v^\dagger\right].
\eeno We already have the estimate for $\p_tu^b-(\p_tu)^b$ from step 3) of  the  last subsection, it only remains to estimate   $w=\p_t(\chi_1u)^b
-(\p_t(\chi_1u))^b-\chi_1\p_tu^b+\chi_1(\p_tu)^b$. We get for $w$ the equation  in $\mathcal{S}$
\[ \left.\begin{array}{ll}\begin{aligned}
 \na_{x,z}\cdot P_M\na_{x,z}w=&-\na_{x,z}\cdot\p_tP_M\na_{x,z}(\chi_1u)^b+[\chi_1,\,
\na_{x,z}\cdot P_M\na_{x,z}](\p_tu^b-(\p_tu)^b)\\
 &\, +\chi_1\na_{x,z}\cdot\p_tP_M\na_{x,z}u^b
\end{aligned}
\end{array}
\right.
\]
that we can rewrite as
\[ \left.\begin{array}{ll}\begin{aligned}
 \na_{x,z}\cdot P_M\na_{x,z}w=&-\na_{x,z}\cdot\p_tP_M\na_{x,z}\big((\chi_1u)^b- \chi_{1} u^b \big)+[\chi_1,\,
\na_{x,z}\cdot P_M\na_{x,z}](\p_tu^b-(\p_tu)^b)\\
 &\, +[\chi_1, \na_{x,z}\cdot\p_tP_M\na_{x,z} ]u^b
\end{aligned}
\end{array}
\right.
\]
with the boundary conditions
\[  w|_{z=0}=0,\quad \p^{P_M}_nw|_{z=-1}= 0.
\] 
We thus get the estimate
$$  \| w \|_{L^2 }+ \| \nabla_{x, z} w \|_{L^2} \leq C \Big(  \|  \nabla_{x,z }\big((\chi_{1} u)^b - \chi_{1} u^b \big) \|_{L^2}
 + {1 \over h} \| \partial_{t} u^b - (\partial_{t} u)^b \|_{H^1} + { 1\over h} \|u^b \|_{H^1}\Big).$$
 Consequently, by using \eqref{comestubM5}, \eqref{comestubM10} and \eqref{comubm1est}, we obtain 
\[\|w\|_{L^2}+\|\na_{x,z}w\|_{L^2}\le \f 1h  (|\mathfrak P u|_{L^2}+|u|_{L^2})\]
 and hence
 $$ | A_{2}| \leq  \f 1h  (|\mathfrak P u|_{L^2}+|u|_{L^2})( |\mathfrak P v|_{L^2}+|v|_{L^2}).$$ 
 Gathering the estimates for $A_{1}$ and $A_{2}$, we 
 end the proof  by substituting  
$u=\chi_1U_2$ and $v=U_2$ into the estimates.
\ef

\bigskip

\bigskip

{\bf Aknowledgements}. 

The second author thanks P. Rapha\"el who introduced him  to the subject during a 
session of the  ANR ArDyPitEq (JC09-437086) in Toulouse.
He also wishes to thank the organizers of this session P. Gravejat and M. Maris. 

This work was partially supported by the ERC grant Dispeq.

\end{document}